\pdfoutput=1
\RequirePackage{ifpdf}
\ifpdf 
\documentclass[pdftex]{sigma}
\else
\documentclass{sigma}
\fi

\numberwithin{equation}{section}
\numberwithin{figure}{section}
\newtheorem{thm}{Theorem}[section]
\newtheorem{cor}[thm]{Corollary}
\newtheorem{lem}[thm]{Lemma}
\newtheorem{prop}[thm]{Proposition}
 { \theoremstyle{definition}

\newtheorem{exa}[thm]{Example}
\newtheorem{rem}[thm]{Remark} }

\def \C {\mathbb{C}}

\def \Q {\mathbb{Q}}
\def \Z {\mathbb{Z}}
\def \N {\mathbb{N}}
\def \re {\operatorname{Re}}
\def \im {\operatorname{Im}}
\def \diag {\operatorname{diag}}

\def \tri {\triangle}

\begin{document}

\newcommand{\arXivNumber}{1804.10369}

\renewcommand{\thefootnote}{}

\renewcommand{\PaperNumber}{113}

\FirstPageHeading

\ShortArticleName{PV Schlesinger-Type Equation}

\ArticleName{Three-Parameter Solutions\\ of the PV Schlesinger-Type Equation\\ near the Point at Infinity and the Monodromy Data\footnote{This paper is a~contribution to the Special Issue on Painlev\'e Equations and Applications in Memory of Andrei Kapaev. The full collection is available at \href{https://www.emis.de/journals/SIGMA/Kapaev.html}{https://www.emis.de/journals/SIGMA/Kapaev.html}}}

\Author{Shun SHIMOMURA}

\AuthorNameForHeading{S.~Shimomura}

\Address{Department of Mathematics, Keio University,\\
 3-14-1, Hiyoshi, Kohoku-ku, Yokohama 223-8522, Japan}
\Email{\href{mailto:shimomur@math.keio.ac.jp}{shimomur@math.keio.ac.jp}}

\ArticleDates{Received May 01, 2018, in final form October 03, 2018; Published online October 22, 2018}

\Abstract{For the Schlesinger-type equation related to the fifth Painlev\'e equation (V) via isomonodromy deformation, we present a three-parameter family of matrix solutions along the imaginary axis near the point at infinity, and also the corresponding monodromy data. Two-parameter solutions of (V) with their monodromy data immediately follow from our results. Under certain conditions, these solutions of (V) admit sequences of zeros and of poles along the imaginary axis. The monodromy data are obtained by matching techniques for a perturbed linear system.}


\Classification{34M55; 34M56; 34M40; 34M35; 34E10}

\begin{flushright}
\it To the memory of Andrei A.~Kapaev
\end{flushright}

\renewcommand{\thefootnote}{\arabic{footnote}}
\setcounter{footnote}{0}

\section{Introduction}\label{sc1}

The fifth Painlev\'e equation normalised in the form
\begin{gather}
\frac{{\rm d}^2y}{{\rm d} x^2}= \left(\frac 1{2y} + \frac 1{y-1} \right) \left(\frac{{\rm d} y}{{\rm d} x} \right)^{2}- \frac 1x \frac{{\rm d} y}{{\rm d} x}\nonumber\\
{} +\frac{(y-1)^2}{8x^2} \left((\theta_0-\theta_x+\theta_{\infty}
)^2 y - \frac{(\theta_0 -\theta_x - \theta_{\infty})^2 } {y}\right) + (1-\theta_0-\theta_x) \frac{y}{x} -\frac{y(y+1)}{2(y-1)} \tag*{(V)}
\end{gather}
with $\theta_0, \theta_x, \theta_{\infty} \in \C$ is derived from the isomonodromy deformation of a two-dimensional linear system of the form
\begin{gather}\label{1.1}
\frac{{\rm d}Y}{{\rm d}\lambda} =\left( \frac{A_0(x)}{\lambda} +\frac{A_x(x)}{\lambda - x} + \frac{J} 2 \right) Y
\end{gather}
with $J=\diag [1, -1]$ under a small change of $x$, where $A_0(x)$ and $A_x(x)$ satisfy the following:
\begin{enumerate}\itemsep=0pt
\item[(a)] the eigenvalues of $A_0(x)$ and $A_x(x)$ are $\pm \theta_0/2$ and $\pm \theta_x/2$, respectively;
\item[(b)] $(A_0(x) +A_x(x))_{11} = -(A_0(x) + A_x(x))_{22} \equiv -\theta_{\infty}/2$.
\end{enumerate}
Such matrices $A_0(x)$, $A_x(x)$ may be written in the form
\begin{gather*}
 A_0(x)=\begin{pmatrix}
 z+ \theta_0/2 & -u(z+ \theta_0) \\
z/u & -z-\theta_0/2
\end{pmatrix},\\
 A_x(x)=\begin{pmatrix}
-z-(\theta_0+\theta_{\infty})/2 & uy(z+(\theta_0-\theta_x +\theta_{\infty})/2) \\
-(uy)^{-1}(z+(\theta_0+\theta_x +\theta_{\infty})/2) & z+(\theta_0+\theta_{\infty})/2
\end{pmatrix},
\end{gather*}
and then
\begin{gather}\label{1.2}
y= \frac{A_x(x)_{12} (A_0(x)_{11} + \theta_0/2)} {A_0(x)_{12} (A_x(x)_{11} + \theta_x/2)}, \qquad z= A_0(x)_{11} - \theta_0/2, \qquad u= -\frac{A_0(x)_{12} } {A_0(x)_{11} + \theta_0/2}
\end{gather}
(cf.\ Andreev and Kitaev \cite{Andreev-Kitaev}, Jimbo and Miwa \cite[Appendix~C]{JM}). The functions $y$ and $z$ are the same as those in \cite{Andreev-Kitaev, JM}, and $u$ is written as $u=x^{-\theta_{\infty}}u_{\mathrm{AK}}$, where $u_{\mathrm{AK}}$ denotes the function $u$ of \cite{Andreev-Kitaev, JM}. System \eqref{1.1} has the isomonodromy property if and only if $(A_0(x), A_x(x))$ solves the Schlesinger-type equation
\begin{gather}\label{1.3}
x\frac{{\rm d}A_0}{{\rm d}x} = [A_x, A_0] , \qquad x\frac{{\rm d}A_x}{{\rm d}x} = [A_0, A_x] + \frac x2 [J, A_x]
\end{gather}
(for more concrete setting of monodromy matrices for \eqref{1.1} invariant under a change of $x$, see Section \ref{ssc2.2}); and then $y$ as in \eqref{1.2} solves~(V). Conversely, for any solution $y$ of~(V) there exists $(A_0(x), A_x(x))$ satisfying \eqref{1.2} and~\eqref{1.3} (cf.\ \cite[Section~3]{Jimbo}, \cite[Appendix C]{JM}).

Near $x=\infty$, two-parameter families of convergent solutions of~(V) were obtained by solving the Hamiltonian system for~(V) (cf.~\cite{S1, T}). Computing monodromy matrices for a system equivalent to~\eqref{1.1} by WKB analysis, and using these matrices, which should be independent of~$x$, Andreev and Kitaev~\cite{Andreev-Kitaev} obtained asymptotic solutions of~(V) near $x=0$ and $x=\infty$ on the positive real axis, and connection formulas for these solutions. Recently it was shown that, for~(V) near $x=0$ (respectively, $x= +\infty$ or $x= {\rm i} \infty$), a series expression of the tau-func\-tion~$\tau_{\mathrm{V}}(x)$ may be given by regular (respectively, irregular) conformal blocks (cf.\ Bonelli et al.~\cite{BLMST}, Gamayun et al.~\cite{GIL}, Nagoya~\cite{N}). Furthermore, using the $s$-channel representation of the PVI tau-function and confluence procedure, Lisovyy et al.~\cite{L} gave a conjectural connection formula for~$\tau_{\mathrm{V}}(x)$ between $x=0$ and $x={\rm i}\infty$ \cite[Conjecture~C]{L} and the ratios of multipliers of~$\tau_{\mathrm{V}}(x)$ as $x\to 0, +\infty, {\rm i}\infty$ \cite[Conjecture~D]{L}.

As the first step of giving tables of critical behaviours for (V) like those of Guzzetti \cite{G-Table} for the sixth Painlev\'e equation, the author~\cite{S3} presented some families of convergent solutions of~(V) near $x=0$ and the respective monodromy data parametrised by integration constants. In this paper we present a family of matrix solutions of the Schlesinger-type equation~\eqref{1.3} parametrised by three integration constants~$c_0$,~$c_x$, $\sigma$ as $x\to\infty$ along the imaginary axis, and also the corresponding monodromy data (note that~\eqref{1.3} under the restrictions~(a) and~(b) is regarded as a nonlinear system with respect to
$(y,z,u)$). For the PVI Schlesinger equation and for~\eqref{1.3} around $x=0$, such matrix solutions have been essentially given~\cite{S2, S3}. To find solutions of~\eqref{1.3} around $x=\infty$ we need quite different techniques. As explained later the monodromy data are computable by using $(A_0(x), A_x(x))$, which is an advantage of treating solutions of~\eqref{1.3} instead of those of~(V). Each entry of the solution $(A_0(x), A_x(x))$ is a convergent series in powers of $\big({\rm e}^x x^{\sigma-1}, {\rm e}^{-x} x^{-\sigma-1}\big)$ having coefficients given by asymptotic series in~$x^{-1}$. This expression is valid in a sector-like domain with opening angle zero, where ${\rm e}^x x^{\sigma-1}$ and ${\rm e}^{-x}x^{-\sigma-1}$ are sufficiently small. This domain is larger than that known for series solutions of~(V) (cf.\ Remark~\ref{rem2.91}), which is another advantage of solutions of \eqref{1.3}. Then we easily obtain a~two-parameter family of solutions of~(V) by using~\eqref{1.2}, whose corresponding monodromy data also follow by restricting~$c_0$ to~$1$. These monodromy data make it possible to know the parametric connection formulas between the solutions of~\eqref{1.3} or~(V) mentioned above and those near $x=0$ (respectively, those along the positive real axis near $x=\infty$ by~\cite{Andreev-Kitaev}). Furthermore, by virtue of the quotient expression~\eqref{1.2}, under certain conditions, we may find sequences of zeros and of poles of solutions of~(V) in the sector-like domain mentioned above.

Our results are described in Section \ref{sc2}: families of solutions of \eqref{1.3} are given in Theorems \ref{thm2.1} and \ref{thm2.2}; the monodromy data in Theorems~\ref{thm2.3}, \ref{thm2.4} and Corollary~\ref{cor2.4a};
families of solutions of~(V) in Theorems~\ref{thm2.6} and~\ref{thm2.7}; and sequences of zeros and of poles in Theorems~\ref{thm2.8} and~\ref{thm2.9}. To our goal we make an approach different from that in~\cite{Andreev-Kitaev}: first construct a~general solution $(A_0(x), A_x(x))$ of \eqref{1.3} containing the integration constants $c_0$, $c_x$, $\sigma$; insert it into~\eqref{1.1}, which becomes a perturbed system with respect to $x^{-1}$; and finally find the monodromy matrices by matching techniques. In Section \ref{sc3}, we define the families $\mathfrak{A}$, $\mathfrak{A}_+$ and $\mathfrak{A}_-$ consisting of power series in
$\big({\rm e}^x x^{\sigma-1}, {\rm e}^{-x} x^{-\sigma-1}\big)$, in ${\rm e}^x x^{\sigma_0-1}$ with $\sigma_0=-2\theta_x-\theta_{\infty}$ and in $ {\rm e}^{-x} x^{-\sigma'_0-1}$ with $\sigma'_0= 2\theta_0+\theta_{\infty}$, respectively, whose coefficients are asymptotic series in $x^{-1}$ in suitable sectors, and show several lemmas which are used in the construction of solutions. In Sections~\ref{sc4} and~\ref{sc5}, under the restrictions (a) and (b) we transform~\eqref{1.3} into a system of integral equations, and solve it by successive approximation to obtain solutions as in Theorems~\ref{thm2.1} and~\ref{thm2.2}. Section~\ref{sc6} is devoted to the proofs of Theorems~\ref{thm2.6} through~\ref{thm2.9} on solutions of~(V). In the final section we prove Theorems~\ref{thm2.3} and~\ref{thm2.4}. Application of matchings to asymptotic solutions of the perturbed system yields monodromy matrices for~\eqref{1.1} that apparently contains $x^{-1}$, and the desired matrices are obtained by letting $x\to \infty$, which is justified by the isomonodromy property. In this procedure, we use functions that are essentially WKB solutions, but for a technical reason we treat them in a method somewhat different from that in usual WKB analysis. For other Painlev\'e equations, WKB analysis and matching technique have been employed to find connection formulas, non-linear Stokes behaviour, distribution of poles or zeros, several examples of which are described in~\cite{FIKN, IN}. To this field Andrei Kapaev made pioneering contributions by using and developing the WKB matching technique in his works including~\cite{K1, K2, K3}. For basic techniques of WKB analysis and related materials see~\cite{F, O, W2}.

Throughout this paper the following symbols are used.
\begin{enumerate}\itemsep=0pt
\item[(1)] $I$, $J$, $\Delta_+$, $\Delta_-$ denote the matrices
\begin{gather*}
I= \begin{pmatrix}
1 & 0 \\ 0 & 1 \\
\end{pmatrix},
\qquad
J= \begin{pmatrix}
1 & 0 \\ 0 & -1 \\
\end{pmatrix},
\qquad
\Delta_+ = \begin{pmatrix}
0 & 1 \\ 0 & 0 \\
\end{pmatrix},
\qquad
\Delta_- = \begin{pmatrix}
0 & 0 \\ 1 & 0 \\
\end{pmatrix}.
\end{gather*}

\item[(2)] $\mathcal{R}(\C\setminus \{0\} )$ denotes the universal covering of $\C \setminus \{0 \}$.

\item[(3)] $\Q_{*}:= \Q\big[\theta_0, \theta_x, \theta_{\infty}, c_0, c_0^{-1}, c_x, c_x^{-1},\sigma\big]$.

\item[(4)] For $k\in \N\cup \{0\}$, $\big[x^{-k}\big]$ (respectively, $\big[x^{-k}\big]_*$) denotes a holomorphic function admitting an asymptotic representation of the form $f(x) \sim x^{-k} \sum\limits_{j\ge 0} f_j x^{-j}$ with $f_j\in \Q_*$ (respectively, $f_j \in \Q(\theta_0, \theta_x, \theta_{\infty},c, \sigma)$ with $c=(c')^{-1}=c_x/c_0$) as $x\to \infty$ through a~specified sector, each~$f_j$ not being necessarily nonzero (e.g., note that $\big[x^{-k}\big]$ with $k\in \N$ may also be denoted by~$[1]$ $\big({=}\big[x^0\big]\big)$). Furthermore in Sections~\ref{sc4} and \ref{sc5}, for simplicity, we often denote by $(1)_x$ a~function given by $1+\big[x^{-1}\big]$.

\item[(5)] {\allowdisplaybreaks Sectors and domains:
\begin{gather*}
\Sigma(\sigma, \varepsilon, x_{\infty}, \delta)\colon \ \!
|\arg x - \pi/2| <\pi/2 -\delta, \!\!\quad \big|{\rm e}^x x^{\sigma-1} \big| <\varepsilon, \!\!\quad
 \big|{\rm e}^{-x} x^{-\sigma-1} \big| <\varepsilon, \!\!\quad |x|> x_{\infty}, \\
 D(B_*, \varepsilon, x_{\infty}, \delta) = \bigcup_{\sigma\in B_*}
\{\sigma \} \times \Sigma(\sigma, \varepsilon, x_{\infty}, \delta),\\
\Sigma_*(\varepsilon, x_{\infty}, \delta)\colon \ -(\pi/2-\delta)< \arg x - \pi/2 <\pi -\delta, \qquad \big|{\rm e}^x x^{\sigma_0-1} \big| <\varepsilon, \qquad |x|> x_{\infty},\\
\Sigma'_*(\varepsilon, x_{\infty}, \delta)\colon \
-(\pi-\delta) <\arg x - \pi/2 <\pi/2 -\delta, \qquad\!\! \big|{\rm e}^{-x} x^{-\sigma'_0-1} \big| <\varepsilon, \qquad\!\! |x|> x_{\infty},\!\!\\
\Sigma_{**}(x_{\infty}, \delta)\colon \ |\arg x - \pi/2| <\pi -\delta, \qquad |x|> x_{\infty}, \\
\Sigma_{0}(x_{\infty}, \delta)\colon \ |\arg x - \pi/2| <\pi/2 -\delta, \qquad |x|> x_{\infty},\\
\Sigma_{\pi}(\Theta_1, \Theta_2; x_{\infty})\colon \ (\pi/2 <)\ \Theta_1 <\arg x <\Theta_2 \ ( < 3\pi/2), \qquad |x|> x_{\infty}.
\end{gather*}

}

\item[(6)] For the integration constants $c_0$, $c_x$ and $\sigma$, we frequently write
\begin{gather*}
\gamma^0_+ := c_0(\sigma +2\theta_0 -\theta_{\infty})/4, \qquad \gamma^0_- := c_0^{-1}(-\sigma +2\theta_0 +\theta_{\infty})/4, \qquad \\
\gamma^x_+ := c_x(-\sigma +2\theta_x -\theta_{\infty})/4, \qquad \gamma^x_- := c_x^{-1}(\sigma +2\theta_x +\theta_{\infty})/4,\\
\mathbf{c} := (c_0, c_x), \qquad c:=c_x/c_0, \qquad c':=c_0/c_x.
\end{gather*}

\item[(7)] For a sequence $\{ \phi^j \}$, $\tri \phi^j := \phi^j -\phi^{j-1}$ in Section \ref{sc5}.
\end{enumerate}

\section{Results}\label{sc2}

\subsection{Solutions of the Schlesinger-type equation}\label{ssc2.1}

For $\delta$, $\varepsilon$, $x_{\infty}$ satisfying $\delta<\pi/2$, $x_{\infty} >\varepsilon^{-1}$, and for each $\sigma\in B_* \subset \C$, define $\Sigma(\sigma,\varepsilon, x_{\infty}, \delta) \subset \mathcal{R}(\C\setminus \{0\})$ by
\begin{gather*}
\Sigma(\sigma,\varepsilon, x_{\infty}, \delta)\colon \ |\arg x-\pi/2 |<\pi/2-\delta, \quad \big|{\rm e}^x x^{\sigma-1} \big|< \varepsilon, \quad \big|{\rm e}^{-x} x^{-\sigma-1} \big|< \varepsilon, \quad |x|>x_{\infty}
\end{gather*}
and $D(B_*,\varepsilon, x_{\infty},\delta) \subset B_* \times \mathcal{R}(\C\setminus\{0\}) \subset \C \times \mathcal{R}(\C\setminus\{0\})$ by
\begin{gather*}
D(B_*,\varepsilon, x_{\infty},\delta) := \bigcup_{\sigma\in B_*} \{\sigma\} \times \Sigma(\sigma,\varepsilon, x_{\infty}, \delta).
\end{gather*}

\begin{thm}\label{thm2.1} Let $B_*\subset \C$ and $B_0, B_x \subset \C\setminus \{0\}$ be given bounded domains, and let $\delta$ be a~given positive number such that $\delta<\pi/2$. Then equation \eqref{1.3} admits a~three-parameter family of solutions
\begin{gather*}
\{ (A_0(\mathbf{c}, \sigma, x), A_x(\mathbf{c}, \sigma, x));\, (\mathbf{c},\sigma) :=(c_0,c_x,\sigma) \in B_0\times B_x \times B_* \}
\end{gather*}
with
\begin{gather*}
A_0(\mathbf{c}, \sigma, x)= f_0(\mathbf{c},\sigma, x)J + f_+(\mathbf{c},\sigma, x) \Delta_+ + f_-(\mathbf{c},\sigma, x)\Delta_-,\\
 A_x(\mathbf{c}, \sigma, x)= g_0(\mathbf{c},\sigma, x)J + g_+(\mathbf{c},\sigma, x)\Delta_+ + g_-(\mathbf{c},\sigma, x)\Delta_-
\end{gather*}
satisfying the conditions $(\mathrm{a})$ and $(\mathrm{b})$. The entries are holomorphic in $(\mathbf{c}, \sigma, x) \in B_0\times B_x \times D(B_*, \varepsilon, x_{\infty}, \delta)$, and are represented by the convergent series in powers of $\big({\rm e}^xx^{\sigma-1}, {\rm e}^{-x} x^{-\sigma-1}\big)$ as follows:{\allowdisplaybreaks
\begin{gather*}
f_0(\mathbf{c}, \sigma, x)= (\sigma-\theta_{\infty})/4 -\big((\sigma+\theta_{\infty}) \gamma^0_+\gamma^0_- +(\sigma-\theta_{\infty})\gamma^x_+\gamma^x_-\big)x^{-2}/2+\big[x^{-3}\big]\\
\qquad{}+\gamma^0_-\gamma^x_+\big(1+\big[x^{-1}\big]\big){\rm e}^x x^{\sigma-1} +\sum_{n=2}^{\infty} \big(\gamma^0_-\gamma^x_+\big)^n\big[x^{-n+1}\big]\big({\rm e}^x x^{\sigma-1}\big)^n\\
\qquad{}+\gamma^0_+\gamma^x_-\big(1+\big[x^{-1}\big]\big){\rm e}^{-x} x^{-\sigma-1} +\sum_{n=2}^{\infty} \big(\gamma^0_+\gamma^x_-\big)^n\big[x^{-n+1}\big]\big({\rm e}^{-x} x^{-\sigma-1}\big)^n,\\
 g_0(\mathbf{c}, \sigma, x) = -\theta_{\infty}/2 - f_0(\mathbf{c},\sigma, x),\\
x^{(\sigma +\theta_{\infty}) /2} f_+( \mathbf{c},\sigma, x) =\gamma^0_+\big(1+\big[x^{-1}\big]\big) -\gamma^x_+\big((\sigma-\theta_{\infty})/2 +\big[x^{-1}\big]\big){\rm e}^x x^{\sigma-1}\\
\qquad{} -\gamma^0_-\big(\gamma^x_+\big)^2 \big(1+\big[x^{-1}\big]\big)\big({\rm e}^x x^{\sigma-1}\big)^2
 +\sum_{n=3}^{\infty}\gamma^x_+\big(\gamma^0_-\gamma^x_+\big)^{n-1}\big[x^{-n+2}\big]\big({\rm e}^x x^{\sigma-1}\big)^n\\
\qquad{}
+2\big(\gamma^0_+\big)^2\gamma_-^x \big(1+\big[x^{-1}\big]\big){\rm e}^{-x}x^{-\sigma-2}
 +\sum_{n=2}^{\infty}\gamma^0_+\big(\gamma^0_+\gamma^x_-\big)^n \big[x^{-n}\big]\big({\rm e}^{-x} x^{-\sigma-1}\big)^n,\\
{\rm e}^{-x} x^{-(\sigma -\theta_{\infty}) /2} g_+( \mathbf{c},\sigma, x) =
\gamma^x_+\big(1+\big[x^{-1}\big]\big) +2\gamma_-^0\big(\gamma_+^x\big)^2\big(1+\big[x^{-1}\big]\big) {\rm e}^x x^{\sigma-2}\\
\qquad{} +\sum_{n=2}^{\infty}\gamma_+^x\big(\gamma^0_-\gamma^x_+\big)^{n} \big[x^{-n}\big]\big({\rm e}^x x^{\sigma-1}\big)^n
 - \gamma^0_+\big((\sigma+\theta_{\infty})/2+\big[x^{-1}\big]\big){\rm e}^{-x} x^{-\sigma-1}\\
\qquad{} -\big(\gamma^0_+\big)^2\gamma^x_- \big(1+\big[x^{-1}\big]\big)\big( {\rm e}^{-x} x^{-\sigma-1} \big)^2
 +\sum_{n=3}^{\infty} \gamma^0_+ \big(\gamma^0_+\gamma^x_-\big)^{n-1}\big[x^{-n+2}\big]\big({\rm e}^{-x} x^{-\sigma-1}\big)^n,\\
x^{-(\sigma +\theta_{\infty}) /2} f_-( \mathbf{c},\sigma, x) =\gamma^0_-\big(1+\big[x^{-1}\big]\big) +2\big(\gamma_-^0\big)^2\gamma_+^x\big(1+\big[x^{-1}\big]\big) {\rm e}^x x^{\sigma-2}\\
\qquad{} +\sum_{n=2}^{\infty} \gamma^0_- \big(\gamma^0_-\gamma^x_+\big)^n \big[x^{-n}\big]\big({\rm e}^x x^{\sigma-1}\big)^n
 -\gamma^x_-\big((\sigma-\theta_{\infty})/2 +\big[x^{-1}\big]\big) {\rm e}^{-x} x^{-\sigma-1}\\
\qquad{} -\gamma^0_+\big(\gamma^x_-\big)^2 \big(1+\big[x^{-1}\big]\big)\big({\rm e}^{-x} x^{-\sigma-1}\big)^2 +\sum_{n=3}^{\infty} \gamma^x_-
\big(\gamma^0_+\gamma^x_-\big)^{n-1} \big[x^{-n+2}\big]\big({\rm e}^{-x} x^{-\sigma-1}\big)^n, \\
{\rm e}^x x^{(\sigma -\theta_{\infty}) /2 } g_-( \mathbf{c},\sigma, x) =
 \gamma^x_-\big(1+\big[x^{-1}\big]\big) -\gamma^0_-\big((\sigma+\theta_{\infty})/2+\big[x^{-1}\big]\big) {\rm e}^{x} x^{\sigma-1}\\
\qquad{} -\big(\gamma^0_-\big)^2\gamma^x_+ \big(1+\big[x^{-1}\big]\big)\big({\rm e}^{x} x^{\sigma-1}\big)^2 +\sum_{n=3}^{\infty}\gamma^0_-\big(\gamma^0_-\gamma^x_+\big)^{n-1}
\big[x^{-n+2}\big]\big({\rm e}^x x^{\sigma-1}\big)^n\\
\qquad{} +2\gamma^0_+\big(\gamma_-^x\big)^2\big(1+\big[x^{-1}\big]\big) {\rm e}^{-x}x^{-\sigma-2} +\sum_{n=2}^{\infty} \gamma_-^x\big(\gamma^0_+\gamma^x_-\big)^n
\big[x^{-n}\big]\big({\rm e}^{-x} x^{-\sigma-1}\big)^n.
\end{gather*}
Here}
\begin{enumerate}\itemsep=0pt
\item[$(i)$] $\varepsilon=\varepsilon(B_0, B_x, B_*,\delta)$ $($respectively, $x_{\infty}=x_{\infty}(B_0,B_x, B_*, \delta) > \varepsilon^{-1})$ is a~sufficiently small $($respectively, large$)$ positive number depending on $(B_0, B_x, B_*,\delta)$;
\item[$(ii)$] $\gamma_{\pm}^0=\gamma_{\pm}^0(\mathbf{c}, \sigma)$, $\gamma_{\pm}^x=\gamma_{\pm}^x(\mathbf{c}, \sigma)$ denote
\begin{gather*}
\gamma^0_+={c_0} (\sigma +2\theta_0 -\theta_{\infty})/4, \qquad \gamma^0_-=c_0^{-1} (-\sigma+ 2\theta_0 + \theta_{\infty})/4,\\
\gamma^x_+={c_x} (-\sigma +2\theta_x -\theta_{\infty})/4, \qquad \gamma^x_-=c_x^{-1} (\sigma +2\theta_x +\theta_{\infty})/4;
\end{gather*}
\item[$(iii)$] the asymptotic series for $\big[x^{-1}\big]$, $\big[x^{-n}\big]$, $\ldots$ are valid uniformly in $(\mathbf{c}, \sigma)\in B_0\times B_x \times B_*$ as~$x$ tends to $\infty$ through the sector $|\arg x -\pi/2|<\pi/2-\delta$, $|x|>x_{\infty}$.
\end{enumerate}
\end{thm}

\begin{rem}\label{rem2.0} The restriction (a) implies the relations $f_0(\mathbf{c}, \sigma, x)^2+ f_+(\mathbf{c}, \sigma, x)f_-(\mathbf{c}, \sigma, x) \equiv \theta_0^2/4$ and $g_0(\mathbf{c}, \sigma, x)^2 + g_+(\mathbf{c}, \sigma, x)g_-(\mathbf{c}, \sigma, x) \equiv \theta_x^2/4$.
\end{rem}

\begin{rem}\label{rem2.1}More precisely, $\varepsilon=\varepsilon(B_0, B_x, B_*,\delta)$ may be chosen in such a way that
\begin{gather*}
\big(\big|\gamma^0_-\gamma^x_+\big| \!+\! \big|\gamma_+^0\gamma^x_-\big| \!+\! \big|\gamma_+^0\big| \!+\! \big|\gamma_-^0\big| \!+\! \big|\gamma_+^x\big| \!+\! \big|\gamma_-^x\big| \!+\! 1\big)
\big( \big|\gamma_+^0\big| \!+\! \big|\gamma_-^0\big| \!+\! \big|\gamma_+^x\big| \!+\! \big|\gamma_-^x\big| \!+\! 1\big)\varepsilon \le r_0(\delta)
\end{gather*}
for every $(c_0, c_x, \sigma) \in B_0\times B_x\times B_*$, where $r_0(\delta) <1$ is a sufficiently small positive number depending on~$\delta$ (see Sections~\ref{ssc5.2},~\ref{ssc5.4} and Proposition~\ref{prop5.2}).
\end{rem}

\begin{rem}\label{rem2.2} The sector-like domain $\Sigma(\sigma, \varepsilon,x_{\infty},\delta)$ is given by $|x|> x_{\infty}$ and
\begin{gather*}
-(1+ \re \sigma)\log |x| +\im \sigma\cdot \arg x + \log\big(\varepsilon^{-1}\big)\\
\qquad{} < \re x < (1-\re \sigma) \log |x| +\im \sigma\cdot \arg x -\log\big(\varepsilon^{-1}\big),
\end{gather*}
where $\im \sigma\cdot \arg x = O(1)$ since $|\arg x-\pi/2|<\pi/2 -\delta$ (cf.\ Fig.~\ref{fig1}).
\end{rem}

\begin{figure}[htb]
\small
\begin{center}
\unitlength=0.6mm
\begin{picture}(55,68)(-35,-20)
\put(-35,-5){\line(1,0){42}}
\put(9,-7){\makebox(10,7)[br]{\scriptsize $\re x$}}
\put(0,-10){\line(0,1){50}}
\put(-5,40){\makebox(10,7)[tr]{\scriptsize $\im x$}}
\thicklines
\qbezier(-14.1,10)(-23.2,20)(-28.3,40)
\qbezier(-5,10)(-8.5,20)(-10,40)
\qbezier[7](-5,10)(-8,12)(-14.1,10)
\put(-35,-25){\makebox(55,7){\small (a)\,\, $\re \sigma>1$}}
\end{picture}
\hskip1.0cm
\begin{picture}(55,68)(-20,-20)
\put(-20,-5){\line(1,0){45}}
\put(27,-7){\makebox(10,7)[br]{\scriptsize $\re x$}}
\put(0,-10){\line(0,1){50}}
\put(-5,40){\makebox(10,7)[tr]{\scriptsize $\im x$}}
\thicklines
\qbezier(10,10)(17,20)(20,40)
\qbezier(-7.1,10)(-11.6,20)(-14.1,40)
\qbezier[15](-7.1,10)(2,13)(10,10)
\put(-20,-25){\makebox(55,7){\small (b)\,\, $|\re \sigma|<1$}}
\end{picture}
\hskip1.0cm
\begin{picture}(55,68)(-5,-20)
\put(-5,-5){\line(1,0){42}}
\put(39,-7){\makebox(10,7)[br]{\scriptsize $\re x$}}
\put(0,-10){\line(0,1){50}}
\put(-5,40){\makebox(10,7)[tr]{\scriptsize $\im x$}}
\thicklines
\qbezier(14.1,10)(23.2,20)(28.3,40)
\qbezier(5,10)(8.5,20)(10,40)
\qbezier[7](5,10)(8,12)(14.1,10)
\put(-5,-25){\makebox(55,7){\small (c)\,\, $\re \sigma<-1$}}
\end{picture}
\end{center}
\caption{$\Sigma(\sigma, \varepsilon, x_{\infty}, \delta)$.}\label{fig1}
\end{figure}
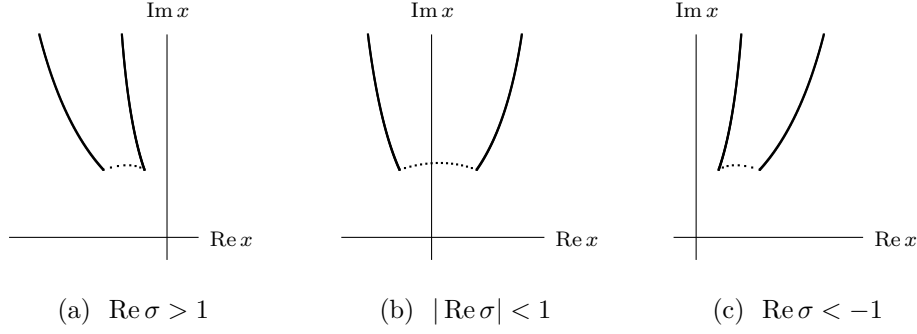

\begin{rem}\label{rem2.3} The asymptotic series for $\big[x^{-1}\big]$, $\big[x^{-n}\big]$, $\dots$ seem to be valid in an extended sector $|\arg x -\pi/2|<\pi -\delta$ (cf.~\cite{S1, T}).
\end{rem}

\begin{rem}\label{rem2.4} For $k\in \Z$, let $\Sigma_k(\sigma, \varepsilon,x_{\infty}, \delta)\subset \mathcal{R}(\C\setminus\{0\})$ be the sector-like domain defined by
\begin{gather*}
 |\arg x -(1/2 +k)\pi |<\pi/2 -\delta, \qquad \big|{\rm e}^x x^{\sigma-1} \big|<\varepsilon, \qquad \big|{\rm e}^{-x} x^{-\sigma-1} \big|<\varepsilon, \qquad |x|>x_{\infty}
\end{gather*}
(note that $\Sigma_0(\sigma,\varepsilon,x_{\infty},\delta) =\Sigma(\sigma,\varepsilon, x_{\infty},\delta)$). Then in the domain $D_k(B_*,\varepsilon,x_{\infty},\delta):= \bigcup_{\sigma\in B_*} \{\sigma\} \times
\Sigma_k(\sigma,\varepsilon,x_{\infty},\delta)$ equation~\eqref{1.3} admits a~family of solutions $\big\{ \big(A^{(k)}_0(\mathbf{c}, \sigma, x),A^{(k)}_x(\mathbf{c}, \sigma, x) \big) \big\}$ having an expression of the same form as in Theorem~\ref{thm2.1} with $\big[x^{-1}\big]$, $\big[x^{-n}\big]$, $\dots$ in the sector $|\arg x -(1/2+k)\pi|<\pi/2-\delta$.
\end{rem}

\begin{rem}\label{remA}By Theorem \ref{thm2.1} and Remark \ref{rem5.11} the tau-function is given by
\begin{gather*}
\frac{{\rm d}}{{\rm d}x} \log \tau_{\mathrm{V}}(x) = x^{-1} \operatorname{tr}(A_0A_x) -\operatorname{tr}(A_0J/2)-\theta_{\infty}/2\\
\hphantom{\frac{{\rm d}}{{\rm d}x} \log \tau_{\mathrm{V}}(x)}{} = (A_x)_{11} + x^{-1} \bigl( 2(A_0)_{11}(A_x)_{11} +(A_0)_{12}(A_x)_{21}+(A_0)_{21}(A_x)_{12} \bigr)\\
\hphantom{\frac{{\rm d}}{{\rm d}x} \log \tau_{\mathrm{V}}(x)}{} = - (\sigma+\theta_{\infty})/4 - \big(\sigma^2-\theta_{\infty}^2\big)x^{-1}/8\\
\hphantom{\frac{{\rm d}}{{\rm d}x} \log \tau_{\mathrm{V}}(x)=}{} - \big((\sigma+\theta_{\infty})\gamma_+^0\gamma_-^0 +(\sigma-\theta_{\infty})\gamma_+^x\gamma_-^x \big) x^{-2} /2+\big[x^{-3}\big]\\
\hphantom{\frac{{\rm d}}{{\rm d}x} \log \tau_{\mathrm{V}}(x)=}{} - \gamma^0_-\gamma^x_+\big(1+\big[x^{-1}\big]\big) {\rm e}^x x^{\sigma-2} + \gamma^0_+\gamma^x_-\big(1+\big[x^{-1}\big]\big) {\rm e}^{-x} x^{-\sigma-2}\\
\hphantom{\frac{{\rm d}}{{\rm d}x} \log \tau_{\mathrm{V}}(x)=}{} +\sum_{n=2}^{\infty} \big(\gamma_-^0\gamma_+^x\big)^{n}\big[x^{-n+1}\big]\big({\rm e}^xx^{\sigma-1}\big)^n
+\sum_{n=2}^{\infty} \big(\gamma_+^0\gamma_-^x\big)^{n}\big[x^{-n+1}\big]\big({\rm e}^{-x} x^{-\sigma-1}\big)^n.
\end{gather*}
It may be checked that first some terms agree with those of $\tau(t \to {\rm i}\infty)$ in \cite[equation~(1.12a)]{L}. It was conjectured by \cite{BLMST, N} that $\tau_{\mathrm{V}}(x)$ is represented by an infinite sum of irregular conformal blocks, whose full structure may be observed explicitly. Conformal field theory (with the Fredholm determinant) yields such expansions not via $({\rm d}/{\rm d}x)\log \tau_{\mathrm{V}}(x)$. On the other hand, from the fourth-order bilinear equation~\cite{Andreev-Kitaev, Jimbo, JM}
\begin{gather*}
 x^3\big(\tau \tau^{(4)} - 4\tau'\tau^{(3)} +3(\tau'')^2 \big)+ 4x^2\big(\tau \tau^{(3)} -\tau'\tau''\big)-\big(x^2-2\theta_{\infty}x +\theta_0^2+\theta_x^2)x (\tau\tau''-(\tau')^2\big)\\
\qquad{} + 2x \tau \tau'' +\big(\theta_{\infty} x-\theta_0^2 -\theta_x^2\big) \tau\tau'-\theta_x^2\theta_{\infty} \tau^2/2=0,
\end{gather*}
which does not apparently contain the logarithmic derivative, first some terms of $\tau_{\mathrm{V}}(x)$ may be obtained by an argument similar to that in \cite[Section~3]{GIL1} for $\tau_{\mathrm{VI}}(x)$. It seems difficult to derive the full expansion without finding a suitable structure of this equation.
\end{rem}

For special values of $\sigma$, $c_0$, $c_x$, we have a two-parameter or one-parameter family of solutions. If $\sigma=-2\theta_x -\theta_{\infty}$, namely, $\gamma_-^x =0$, we have

\begin{thm}\label{thm2.2} Suppose that $\theta_x \not=0$ and $\sigma=\sigma_0: =- 2\theta_x - \theta_{\infty} $. Let $\Sigma_*(\varepsilon,x_{\infty}, \delta) \subset \mathcal{R}(\C \setminus \{0\})$ be the domain defined by
\begin{gather*}
\Sigma_*(\varepsilon,x_{\infty}, \delta) \colon \ -(\pi/2-\delta) <\arg x -\pi/2 <\pi -\delta, \qquad \big|{\rm e}^x x^{\sigma_0-1} \big|<\varepsilon, \qquad |x |>x_{\infty},
\end{gather*}
where $\delta<\pi/2$ is a given positive number. Let $\tilde{B} \subset \C$ be a given bounded domain, and $B_0$ as in Theorem~{\rm \ref{thm2.1}}. Then equation \eqref{1.3} admits a two-parameter family of solutions
\begin{gather*}
\{ (A_0(\mathbf{c}, x), A_x(\mathbf{c}, x)); \, \mathbf{c}=(c_0,c_x) \in B_0\times \tilde{B} \big\}
\end{gather*}
with
\begin{gather*}
A_0(\mathbf{c}, x)= f_0(\mathbf{c}, x)J + f_+(\mathbf{c}, x)\Delta_+ + f_-(\mathbf{c}, x)\Delta_-,\\
 A_x(\mathbf{c}, x)= g_0(\mathbf{c}, x)J + g_+(\mathbf{c}, x)\Delta_+ + g_-(\mathbf{c}, x)\Delta_-
\end{gather*}
such that the entries are holomorphic in $(\mathbf{c}, x) \in B_0\times \tilde{B} \times \Sigma_*(\varepsilon,x_{\infty},\delta)$ and are represented by the convergent series in powers of $ {\rm e}^xx^{\sigma_0-1} $ as follows:
\begin{gather*}
f_0(\mathbf{c}, x)=- (\theta_x+\theta_{\infty})/2 +\theta_{x} {\gamma_+^0}_*{\gamma_-^0}_* x^{-2} +\big[x^{-3}\big]\\
\qquad{} +{\gamma_-^0}_*{\gamma_+^x}_*\big(1+\big[x^{-1}\big]\big){\rm e}^x x^{\sigma_0-1}
 +\sum_{n=2}^{\infty}\big({\gamma_-^0}_*{\gamma_+^x}_*\big)^n \big[x^{-n+1}\big]\big({\rm e}^x x^{\sigma_0-1}\big)^n ,\\
g_0(\mathbf{c}, x) = -\theta_{\infty}/2 - f_0(\mathbf{c}, x),\\
x^{- \theta_{x}} f_+( \mathbf{c}, x) ={\gamma_+^0}_*\big(1+\big[x^{-1}\big]\big) +{\gamma_+^x}_*\big(\theta_x +\theta_{\infty} +\big[x^{-1}\big]\big){\rm e}^x x^{\sigma_0-1}\\
\qquad{} -{\gamma_-^0}_*\big({\gamma_+^x}_*\big)^2 \big(1+\big[x^{-1}\big]\big)\big({\rm e}^x x^{\sigma_0-1}\big)^2
 +\sum_{n=3}^{\infty} {\gamma_+^x}_*\big({\gamma_-^0}_*{\gamma_+^x}_*\big)^{n-1} \big[x^{-n+2}\big]\big({\rm e}^x x^{\sigma_0-1}\big)^n, \\
 {\rm e}^{-x} x^{\theta_{x}+\theta_{\infty}} g_+( \mathbf{c}, x) ={\gamma_+^x}_*\big(1+\big[x^{-1}\big]\big) + 2{\gamma_-^0}_*\big({\gamma_+^x}_*\big)^2 \big(1+\big[x^{-1}\big]\big)
{\rm e}^x x^{\sigma_0-2}\\
\qquad{} +\sum_{n=2}^{\infty} {\gamma_+^x}_*\big({\gamma_-^0}_*{\gamma_+^x}_*\big)^n \big[x^{-n}\big]\big({\rm e}^x x^{\sigma_0-1}\big)^n + {\gamma_+^0}_*\big(\theta_{x} +\big[x^{-1}\big]\big) {\rm e}^{-x} x^{-\sigma_0-1},\\
 x^{\theta_{x}} f_-( \mathbf{c}, x) ={\gamma_-^0}_*\big(1+\big[x^{-1}\big]\big) + 2\big({\gamma_-^0}_*\big)^2 {\gamma_+^x}_* \big(1+\big[x^{-1}\big]\big)
{\rm e}^{x} x^{\sigma_0-2}\\
\qquad{} +\sum_{n=2}^{\infty}{\gamma_-^0}_*\big({\gamma_-^0}_*{\gamma_+^x}_*\big)^n \big[x^{-n}\big]\big({\rm e}^x x^{\sigma_0-1}\big)^n ,\\
 {\rm e}^x x^{-\theta_{x}-\theta_{\infty} } g_-( \mathbf{c},x) = {\gamma_-^0}_*(\theta_{x} +\big[x^{-1}\big]) {\rm e}^{x} x^{\sigma_0-1} -\big({\gamma_-^0}_*\big)^2{\gamma_+^x}_* \big(1+\big[x^{-1}\big]\big)\big({\rm e}^{x} x^{\sigma_0-1}\big)^2\\
\qquad{} +\sum_{n=3}^{\infty} {\gamma_-^0}_*\big( {\gamma_-^0}_*{\gamma_+^x}_*\big)^{n-1} \big[x^{-n+2}\big]\big({\rm e}^x x^{\sigma_0-1}\big)^n.
\end{gather*}
Here
\begin{enumerate}\itemsep=0pt
\item[$(i)$] $\varepsilon=\varepsilon(B_0, \tilde{B}, \delta)$ $($respectively, $x_{\infty}=x_{\infty}(B_0, \tilde{B}, \delta))$ is a sufficiently small $($respectively, large$)$ positive number depending on $(B_0, \tilde{B}, \delta);$
\item[$(ii)$] ${\gamma_{\pm}^0}_{*}:=\gamma^0_{\pm}(\mathbf{c},\sigma_0)$, ${\gamma_+^x}_* := \gamma^x_+(\mathbf{c},\sigma_0)$, that is,
\begin{gather*}
{\gamma_+^0}_*= {c_0} (\theta_0 -\theta_x -\theta_{\infty})/2, \qquad {\gamma_-^0}_*= c_0^{-1} (\theta_0 +\theta_x +\theta_{\infty})/2, \qquad {\gamma_+^x}_*= {c_x} \theta_x ;
\end{gather*}
\item[$(iii)$] the asymptotic series for $\big[x^{-1}\big]$, $\big[x^{-n}\big]$, $\ldots$ are valid uniformly in $\mathbf{c}\in B_0\times \tilde{B}$ as $x$ tends to~$\infty$ through the sector $-(\pi/2-\delta) <\arg x -\pi/2 <\pi-\delta$, $|x|>x_{\infty}$, and the coefficients of the series are in $\Q\big[\theta_0, \theta_x,\theta_{\infty}, c_0, c_0^{-1}, c_x \big] \subset \Q_*$.
\end{enumerate}

In addition to $\sigma=\sigma_0$, if $c_x=0$, then \eqref{1.3} admits a one-parameter family of solutions
\begin{gather*}
\{ (A_0({c_0}, x), A_x({c_0}, x)); \, c_0 \in B_0 \}
\end{gather*}
with
\begin{gather*}
A_0({c_0}, x)= f_0({c_0}, x)J + f_+({c_0}, x)\Delta_+ + f_-({c_0}, x)\Delta_-,\\
A_x({c_0}, x)= g_0({c_0}, x)J + g_+({c_0}, x)\Delta_+ + g_-({c_0}, x)\Delta_-,
\end{gather*}
whose entries are holomorphic in $(c_0, x) \in B_0 \times \Sigma_{**} (x_{\infty},\delta)$ with $\Sigma_{**}(x_{\infty},\delta)$: $|\arg x -\pi/2|<\pi -\delta$, $|x|> x_{\infty}$ for some $x_{\infty}=x_{\infty}(B_0,\delta)$, and are represented by the asymptotic series
\begin{gather*}
f_0(c_0, x)= -(\theta_x+\theta_{\infty})/2 + \theta_x \big(\theta_0^2-(\theta_x +\theta_{\infty})^2\big) x^{-2}/4 +\big[x^{-3}\big],\\
g_0(c_0,x)= -\theta_{\infty}/2 - f_0(c_0,x),\\
 x^{-\theta_x} f_+(c_0, x) = c_0(\theta_0 -\theta_x -\theta_{\infty})\big(1/2 + \big[x^{-1}\big]\big),\\
 x^{-\theta_x +1} g_+(c_0, x) = c_0 (\theta_0 -\theta_x -\theta_{\infty})\big(\theta_x/2 + \big[x^{-1}\big]\big),\\
 x^{\theta_x} f_-(c_0, x) = c_0^{-1}(\theta_0 +\theta_x +\theta_{\infty})\big(1/2 + \big[x^{-1}\big]\big),\\
 x^{\theta_x +1} g_-(c_0, x) = c_0^{-1} (\theta_0 +\theta_x +\theta_{\infty})\big(\theta_x/2 + \big[x^{-1}\big]\big),
\end{gather*}
uniformly in $c_0\in B_0$ as $x$ tends to $\infty$ through $\Sigma_{**}(x_{\infty},\delta)$, the coefficients of $\big[x^{-1}\big]$, $\dots$ being in $\Q\big[\theta_0,\theta_x,\theta_{\infty}, c_0, c_0^{-1}\big]$.
\end{thm}

\begin{rem}\label{rem2.5}If we put $\sigma=\sigma'_0=2\theta_0+\theta_{\infty}$ under the supposition $\theta_0 \not=0$, we get a two-parameter family of solutions represented by a power series in ${\rm e}^{-x} x^{-\sigma'_0-1}$ in the domain $\Sigma'_*(\varepsilon, x_{\infty}, \delta)$: \smash{$-(\pi -\delta)$} $<\arg x- \pi/2 <\pi/2-\delta$, $\big|{\rm e}^{-x} x^{-\sigma'_0-1}\big|< \varepsilon$, $|x|>x_{\infty}$. If $\sigma=2\theta_x -\theta _{\infty}$, \smash{$\theta_x \not=0$} (respectively, $\sigma=-2\theta_0 + \theta_{\infty}$, \smash{$\theta_0 \not=0$}), then there exist solutions expanded into series in ${\rm e}^{-x} x^{-2\theta_x +\theta_{\infty} -1}$ (respectively, ${\rm e}^{x} x^{-2\theta_0 +\theta_{\infty} -1}$).
\end{rem}

\subsection{Monodromy data}\label{ssc2.2}
System \eqref{1.1} with (a) and (b) admits a fundamental matrix solution of the form
\begin{gather}\label{2.1}
Y(x, \lambda)=\big(I+O\big(\lambda^{-1}\big)\big) {\rm e}^{(\lambda/2)J}\lambda^{-(\theta_{\infty}/2)J}
\end{gather}
as $\lambda \to \infty$ through the sector $-\pi/2 < \arg \lambda < 3\pi/2$. Denote by $Y_1(x,\lambda)$ and $Y_2(x,\lambda)$ the matrix solutions having asymptotic representations of the same form as in~\eqref{2.1} in the sectors $-3\pi/2 < \arg\lambda <\pi/2$ and $\pi/2 < \arg\lambda <5\pi/2$, respectively. In accordance with \cite[Section~2]{Andreev-Kitaev}, \cite[Section~2.4]{S3} let $S_1=I+s_1\Delta_-$ and $S_2=I+ s_2 \Delta_+$ be the Stokes multipliers given by
\begin{gather*}
Y(x,\lambda)= Y_1(x,\lambda)S_1, \qquad Y_2(x,\lambda) = Y(x,\lambda)S_2,
\end{gather*}
and let $M_0, M_x, M_{\infty} \in {\rm SL}_2(\C)$ be the monodromy matrices defined by the analytic continuation of $Y(x,\lambda)$ along loops $l_0, l_x, l_{\infty} \in \pi_1(P^1(\C) \setminus \{0,x,\infty\})$ located as in Fig.~\ref{loops} for $x$ such that $-\pi <\arg x < \pi$. They surround, respectively, $\lambda=0$, $x$, $\infty$ in the positive sense and satisfy $l_0l_xl_{\infty}= \mathrm{id}$, which implies $M_{\infty} M_x M_0 = I$.

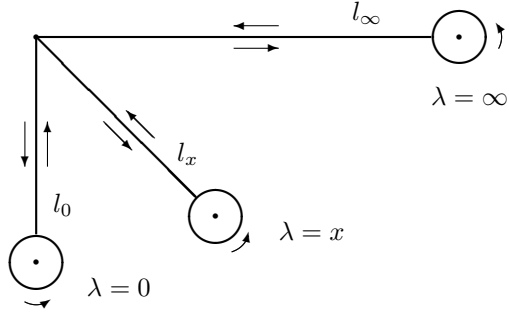
\begin{figure}[htb]
\small
\begin{center}
\unitlength=0.75mm
\begin{picture}(80,60)(-10,-8)
\thicklines
\put(0,40){\circle*{1}}
\put(0,40){\line(0,-1){35}}
\put(0,0){\circle{10}}
\put(0,0){\circle*{1}}
\put(0,40){\line(1,-1){28.3}}
\put(31.8, 8.2){\circle{10}}
\put(31.8, 8.2){\circle*{1}}
\put(0,40){\line(1,0){70}}
\put(75,40){\circle{10}}
\put(75,40){\circle*{1}}
\thinlines
\put(-2,25){\vector(0,-1){8}}
\put(2,17){\vector(0,1){8}}
\qbezier(-2,-7)(0,-8)(2,-7)
\put(2.6,-6.5){\vector(3,2){0}}
\put(12,25){\vector(1,-1){5}}
\put(21,22){\vector(-1,1){5}}
\qbezier(34.7, 1.9)(36.8, 2.6)(37.6, 4.7)
\put(37.7, 5.4){\vector(1,4){0}}
\put(35,38){\vector(1,0){8}}
\put(43,42){\vector(-1,0){8}}
\qbezier(82, 38)(83, 40)(82, 42)
\put(81.5, 42.5){\vector(-2,3){0}}
\put(70,28){\makebox{\small $\lambda=\infty$}}
\put(43,4){\makebox{\small $\lambda=x$}}
\put(9,-6){\makebox{\small $\lambda=0$}}
\put(56,43){\makebox{\small $l_{\infty}$}}
\put(25,18){\makebox{\small $l_{x}$}}
\put(3,9){\makebox{\small $l_{0}$}}
\end{picture}
\end{center}
\caption{$l_0$, $l_x$ and $l_{\infty}$.}\label{loops}
\end{figure}

System \eqref{1.1} has the isomonodromy property, that is, the matrices $M_0$, $M_x$, $S_1$, $S_2$ are invariant under the change of $x$ if and only if $(A_0, A_x)$ solves \eqref{1.3}. Then each solution of~\eqref{1.3} corresponds to some $(M_0,M_x) \in {\rm SL}_2(\C)^2$ not depending on $x$, and then, by
\begin{gather} \label{2.2}
M_{\infty}=M^{-1}_0 M_x^{-1}= S_2 {\rm e}^{\pi {\rm i}\theta_{\infty}J} S_1
\end{gather}
(cf.\ \cite[Section~2]{Andreev-Kitaev}), we have $(M_xM_0)_{21}= -{\rm e}^{-\pi {\rm i} \theta_{\infty}}s_1$, $(M_xM_0)_{12}= -{\rm e}^{-\pi {\rm i} \theta_{\infty}}s_2$, $ \operatorname{tr}(M_xM_0) = 2\cos \pi\theta_{\infty}+ {\rm e}^{-\pi {\rm i}\theta{\infty}} s_1s_2$. As will be seen in Remark~\ref{rem2.6} and Corollary~\ref{cor2.4a}, using the relations in the following theorems we may explicitly represent $(M_0, M_x, S_1, S_2)$ for each solution of~\eqref{1.3} in terms of $\theta_0$, $\theta_x$, $\theta_{\infty}$ and the integration constants $c_0$, $c_x$, $\sigma$.

\begin{thm}\label{thm2.3} Suppose that $\theta_0$, $\theta_x \not\in \Z$. Then, for each $(A_0(\mathbf{c}, \sigma, x), A_x(\mathbf{c}, \sigma, x))$ of Theorem~{\rm \ref{thm2.1}}, the corresponding matrices $M_0$, $M_x$, $ S_1$, $ S_2$ satisfy
\begin{gather}\label{2.3}
S_1 M_x M_0 M_x^{-1} S_1^{-1}= \big(C_0^1\big)^{-1} {\rm e}^{\pi {\rm i}\theta_{0}J} C_0^1, \qquad S_2^{-1}M_0S_2= \big(C_0^2\big)^{-1} {\rm e}^{\pi {\rm i}\theta_{0}J} C_0^2, \\
\label{2.4} M_x= C_x^{-1} {\rm e}^{\pi {\rm i}\theta_x J} C_x,
\end{gather}
where $C^1_0$, $C^2_0$ and $C_x$ are given by
\begin{gather*}
C^1_0=V_0 S_*^{-1} {\rm e}^{\pi {\rm i} (\sigma+\theta_{\infty})J/4} c_0^{-J/2}, \qquad C^2_0=V_0 S_{**}^{-1} {\rm e}^{-\pi {\rm i}(\sigma+\theta_{\infty})J/4} c_0^{-J/2},\qquad C_x=V_x c_x^{-J/2}
\end{gather*}
with
\begin{gather*}
 V_0=\begin{pmatrix}
\dfrac{{\rm e}^{\pi {\rm i}(\sigma-2\theta_0 -\theta_{\infty}) /4} \Gamma(-\theta_0)}{\Gamma(1-(\sigma+2\theta_0 -\theta_{\infty})/4)}
&
\dfrac{ \Gamma(-\theta_0)}{\Gamma(1+(\sigma-2\theta_0 -\theta_{\infty})/4)}
\vspace{1mm}\\
\dfrac{{\rm e}^{\pi {\rm i}(\sigma+2\theta_0 -\theta_{\infty}) /4} \Gamma(\theta_0)}{\Gamma(-(\sigma-2\theta_0 -\theta_{\infty})/4)}
&
- \dfrac{ \Gamma(\theta_0)}{\Gamma((\sigma+2\theta_0 -\theta_{\infty})/4)}
\end{pmatrix},\\
V_x= V_0 \big|_{(\sigma,\theta_0) \mapsto (-\sigma,\theta_x)}, \qquad \text{i.e., the result of the substitution $(\sigma,\theta_0)\mapsto (-\sigma,\theta_x)$ in $V_0$},\\
 S_*= I - \frac{2\pi {\rm i}}{\Gamma(-(\sigma-2\theta_0-\theta_{\infty})/4)\Gamma(1-(\sigma+2\theta_0-\theta_{\infty})/4) } \Delta_-,\\
 S_{**}= I + \frac{2\pi {\rm i} {\rm e}^{-\pi {\rm i}(\sigma-\theta_{\infty})/2}}{\Gamma((\sigma+2\theta_0-\theta_{\infty})/4)\Gamma(1+(\sigma-2\theta_0-\theta_{\infty})/4) } \Delta_+.
\end{gather*}
\end{thm}

\begin{thm}\label{thm2.4}In the case where $\theta_0$ or $\theta_x$ is an integer, the formulas in Theorem~{\rm \ref{thm2.3}} are to be replaced as follows:
\begin{enumerate}\itemsep=0pt
\item[$(1)$] if $\theta_0\in \Z$, then
\begin{gather}\label{2.5}
S_1M_x M_0 M_x^{-1} S_1^{-1} = \big(\hat{C}^1_0\big)^{-1} {\rm e}^{2\pi {\rm i}\Delta_*} \hat{C}^1_0,\qquad S_2^{-1}M_0 S_2 = \big(\hat{C}^2_0\big)^{-1} {\rm e}^{2\pi {\rm i}\Delta_*} \hat{C}^2_0,
\end{gather}
where $\Delta_*$ denotes $\Delta_+$ if $\theta_0 \in \N\cup\{0\}$, and $\Delta_-$ if $-\theta_0 \in \N$, and $\hat{C}^1_0$ and $\hat{C}^2_0$ are given by
\begin{gather*}
\hat{C}^1_0 =\hat{V}_0 S_*^{-1} {\rm e}^{\pi {\rm i}(\sigma+\theta_{\infty})J/4} c_0^{-J/2}, \qquad \hat{C}^2_0 =\hat{V}_0 S_{**}^{-1}{\rm e}^{-\pi {\rm i}(\sigma +\theta_{\infty})J/4} c_0^{-J/2}
\end{gather*}
with $\hat{V}_0$ written in the form
\begin{gather*}
\hat{V}_0=
\begin{pmatrix}
\big(\hat{V}_0\big)_{11} & \big(\hat{V}_0\big)_{12}
\\
1 & 1
\end{pmatrix}\\
\hphantom{\hat{V}_0=}{}\times
 \operatorname{diag} \left[ \frac{{\rm e}^{\pi {\rm i}(\sigma+2\theta_0 -\theta_{\infty})/4} }
{\theta_0 ! \Gamma(1-(\sigma+2\theta_0 -\theta_{\infty})/4) } \,\,\,
 \frac{(-1)^{\theta_0}}{\theta_0 ! \Gamma(1+(\sigma-2\theta_0 -\theta_{\infty})/4) } \right],\\
\big(\hat{V}_0\big)_{11}=\psi(-(\sigma-2\theta_0 -\theta_{\infty})/4) -\psi(1) -\psi(1+\theta_0)-\pi {\rm i},\\
 \big(\hat{V}_0\big)_{12} =\psi(1+(\sigma+2\theta_0 -\theta_{\infty})/4) -\psi(1) -\psi(1+\theta_0)
\end{gather*}
if $\theta_0 \in \N\cup\{0\}$, and
\begin{gather*}
\hat{V}_0=
\begin{pmatrix}
1 & 1
\\
\big(\hat{V}_0\big)_{21} & \big(\hat{V}_0\big)_{22}
\end{pmatrix}\\
\hphantom{\hat{V}_0=}{}\times \operatorname{diag} \left[ \frac{-{\rm e}^{\pi {\rm i}(\sigma-2\theta_0 -\theta_{\infty})/4} }
{(-\theta_0)! \Gamma(-(\sigma-2\theta_0 -\theta_{\infty})/4) } \,\,\,
 \frac{(-1)^{\theta_0}} {(-\theta_0)! \Gamma((\sigma+2\theta_0 -\theta_{\infty})/4) } \right],\\
\big(\hat{V}_0\big)_{21} =\psi(-(\sigma+2\theta_0 -\theta_{\infty})/4) -\psi(1) -\psi(1-\theta_0)-\pi {\rm i},\\
 \big(\hat{V}_0\big)_{22} =\psi(1+(\sigma-2\theta_0 -\theta_{\infty})/4) -\psi(1) -\psi(1-\theta_0)
\end{gather*}
if $- \theta_0 \in \N$, $\psi(t)$ being the di-Gamma function $\psi(t) =\Gamma'(t)/\Gamma(t)$;

\item[$(2)$] if $\theta_x \in \Z$, then
\begin{gather*}
M_x= \hat{C}^{-1}_x {\rm e}^{2\pi {\rm i}\Delta_*} \hat{C}_x, \qquad \hat{C}_x =\hat{V}_x c_x^{-J/2},
\end{gather*}
where $\Delta_*$ is as in $(1)$, and $\hat{V}_x= \hat{V}_0\big|_{(\sigma,\theta_0) \mapsto (-\sigma,\theta_x)}$.
\end{enumerate}
\end{thm}

\begin{rem}\label{rem2.6} Combining $M_xM_0= S_1^{-1}{\rm e}^{-\pi {\rm i} \theta_{\infty}J} S_2^{-1}$ (cf.~\eqref{2.2}) with the first relation in \eqref{2.3} we have
\begin{gather*}
S_2^{-1} M_x^{-1} S_1^{-1} =S_2^{-1} C_x^{-1} {\rm e}^{-\pi {\rm i}\theta_x J}C_xS_1^{-1}={\rm e}^{\pi {\rm i} \theta_{\infty}J} \big(C_0^1\big)^{-1} {\rm e}^{\pi {\rm i}\theta_0 J} C_0^1,
\end{gather*}
if $\theta_0, \theta_x \not\in \mathbb{Z}$. The (2,1)- and (1,2)-entries of this yield $s_1$ and $s_2$, respectively, which reveal the Stokes multipliers $S_1$ and $S_2$, and then $(M_0, M_x, S_1,S_2)$ may be written in terms of $\theta_0$, $\theta_x$, $\theta_{\infty}$, $\sigma$, $c_0$, $c_x$ as in the corollary below (note that $\operatorname{tr} M_0=2\cos \pi\theta_0$, $\operatorname{tr} M_x=2\cos \pi\theta_x$, $\det M_0=\det M_x=1$). In the case where $\theta_0$ or $\theta_x\in \mathbb{Z}$ as well, these matrices are obtained by the same argument. Such $(M_0, M_x)$ is a point on the manifold of the monodromy data (cf.\ \cite[Proposition~2.2, Remark~2.3]{Andreev-Kitaev}).
\end{rem}

\begin{cor}\label{cor2.4a}If $\theta_0$, $\theta_{x} \not\in \Z$, then{\allowdisplaybreaks
\begin{gather*}
 (M_0)_{11} = {\rm e}^{\pi {\rm i}(\sigma-\theta_{\infty})/2} \biggl( 1
 - \frac{2\pi {\rm i} c_0^{-1} }{\Gamma(1-(\sigma+2\theta_0 -\theta_{\infty})/4)\Gamma(-(\sigma-2\theta_0 -\theta_{\infty})/4)}\\
 \hphantom{(M_0)_{11} = {\rm e}^{\pi {\rm i}(\sigma-\theta_{\infty})/2} \biggl(}{} \times \frac{2\pi {\rm i} c_x}
{\Gamma(1-(\sigma+2\theta_x +\theta_{\infty})/4)\Gamma(-(\sigma-2\theta_x +\theta_{\infty})/4) } \biggr),\\
 (M_0)_{21} = \frac{2\pi {\rm i} {\rm e}^{-\pi {\rm i}\theta_{\infty} } c_0^{-1} }{\Gamma(1-(\sigma+2\theta_0 -\theta_{\infty})/4)\Gamma(-(\sigma-2\theta_0 -\theta_{\infty})/4)},\\
(M_x)_{11} = {\rm e}^{-\pi {\rm i}(\sigma+\theta_{\infty})/2},\\
 (M_x)_{12} = \frac{2\pi {\rm i} c_x }{\Gamma(1-(\sigma+2\theta_x +\theta_{\infty})/4)\Gamma(-(\sigma-2\theta_x +\theta_{\infty})/4)},\\
 s_1 = - \frac{2\pi {\rm i} {\rm e}^{\pi {\rm i}(\sigma+\theta_{\infty}) /2} c_0^{-1} }{\Gamma(1-(\sigma+2\theta_0 -\theta_{\infty})/4)\Gamma(-(\sigma-2\theta_0 -\theta_{\infty})/4)}\\
\hphantom{s_1 =}{} - \frac{2\pi {\rm i} c_x^{-1} }{\Gamma(1+(\sigma-2\theta_x +\theta_{\infty})/4)\Gamma((\sigma+2\theta_x +\theta_{\infty})/4)},\\
 s_2 = - \frac{2\pi {\rm i} {\rm e}^{\pi {\rm i}\theta_{\infty} } c_0 }{\Gamma(1+(\sigma-2\theta_0 -\theta_{\infty})/4)\Gamma((\sigma+2\theta_0 -\theta_{\infty})/4)}\\
\hphantom{s_2 =}{} - \frac{2\pi {\rm i} {\rm e}^{\pi {\rm i}(\sigma+\theta_{\infty})/2} c_x } {\Gamma(1-(\sigma+2\theta_x +\theta_{\infty})/4)\Gamma(-(\sigma-2\theta_x +\theta_{\infty})/4)}.
\end{gather*}}
\end{cor}

\begin{rem}\label{2.6a}The results above combined with the monodromy data for solutions of (V) around $x=0$ \cite{Andreev-Kitaev, Jimbo, S3} yield the parametric connection formula between $x=0$ and $x={\rm i}\infty$ (cf.~Remark~\ref{rem2.9}), which corresponds to that for $\tau_{\mathrm{V}}(x)$ of \cite[Conjecture~C]{L}.
\end{rem}

Suppose that, for every $k\in \Z\setminus \{0\}$, $\big(A^{(k)}_0\big(\mathbf{c}^{(k)}, \sigma^{(k)}, x\big), A^{(k)}_x\big(\mathbf{c}^{(k)},\sigma^{(k)}, x\big)\big)$ with $\big(\mathbf{c}^{(k)},\sigma^{(k)} \big) =\big(c_0^{(k)}, c_x^{(k)}, \sigma^{(k)}\big)\in (\C \setminus \{0\})^2 \times \C$ is the analytic continuation of $(A_0(\mathbf{c},\sigma, x), A_x(\mathbf{c},\sigma, x))$ to the domain $\Sigma_{k} \big(\sigma^{(k)}, \varepsilon, x_{\infty}, \delta\big)$ (cf.\ Remark~\ref{rem2.4}). For every $j\in \Z$, let $l_0^{(j)}$ and $l_x^{(j)}$ be the loops in the $\lambda$-plane defined for $(2j-1)\pi < \arg x < (2j+1)\pi$ in the same way as in Fig.~\ref{loops}, and let $\big(M_0^{(k)}, M_x^{(k)}\big)$ with $k=2j$ or $2j-1$ correspond to $\big(A^{(k)}_0\big(\mathbf{c}^{(k)}, \sigma^{(k)}, x\big), A^{(k)}_x\big(\mathbf{c}^{(k)},\sigma^{(k)}, x\big)\big)$ for $x \in \Sigma_{k} \big(\sigma^{(k)}, \varepsilon, x_{\infty}, \delta\big)$, where $M_0^{(k)}$ and $M_x^{(k)}$ are the monodromy matrices given by the analytic continuation of $Y(x, \lambda)$ along $l_0^{(j)}$ and $l_x^{(j)}$, respectively (note that $l_0^{(0)}=l_0$, $l_x^{(0)}=l_x$, $M_0^{(0)}=M_0$, $M_x^{(0)}=M_x$). Then by definition, for every $j\in \Z$,
\begin{gather*}
 \big(M_0^{(2j)}, M_x^{(2j)}\big)= \big(M_0^{(2j)}, M_x^{(2j)}\big)\big(\mathbf{c}^{(2j)},\sigma^{(2j)}\big) =(M_0, M_x) |_{(\mathbf{c},\sigma)\mapsto (\mathbf{c}^{(2j)},\sigma^{(2j)}) },\\
 \big(M_0^{(2j-1)}, M_x^{(2j-1)}\big)= \big(M_0^{(2j-1)}, M_x^{(2j-1)}\big) \big(\mathbf{c}^{(2j-1)},\sigma^{(2j-1)}\big)\\
\hphantom{\big(M_0^{(2j-1)}, M_x^{(2j-1)}\big)}{} =\big(M_0^{(-1)}, M_x^{(-1)}\big)|_{(\mathbf{c}^{(-1)},\sigma^{(-1)})\mapsto(\mathbf{c}^{(2j-1)},\sigma^{(2j-1)}) }.
\end{gather*}

\begin{rem}\label{rem2.7}For $\big(A^{(-1)}_0\big(\mathbf{c}^{(-1)}, \sigma^{(-1)}, x\big),A^{(-1)}_x\big(\mathbf{c}^{(-1)},\sigma^{(-1)}, x\big)\big)$ in $\Sigma_{-1}\big(\sigma^{(-1)},\varepsilon, x_{\infty}, \delta\big)$
the corresponding matrices $M_0^{(-1)}$ and $M_x^{(-1)}$ are defined along the loops $l_0$ and $l_x$, respectively, as in Fig.~\ref{loops} for $-\pi <\arg x <\pi$. If $\theta_0$, $\theta_x\not \in \mathbb{Z}$,
\begin{gather*}
 M_0^{(-1)}=C^{-1}_0 {\rm e}^{\pi {\rm i}\theta_0 J} C_0,\\
S_1 M_x^{(-1)} S_1^{-1}= \big(C^1_x\big)^{-1} {\rm e}^{\pi {\rm i}\theta_x J} C^1_x, \qquad S_2 \big(M_0^{(-1)}\big)^{-1} M_x^{(-1)} M_0^{(-1)} S_2^{-1}= \big(C^2_x\big)^{-1}{\rm e}^{\pi {\rm i}\theta_x J} C^2_x,
\end{gather*}
where
\begin{gather*}
C_0=\tilde{V}_0 {\rm e}^{-\pi {\rm i}(\sigma+\theta_{\infty})J/4} \big(c_0^{(-1)}\big)^{-J/2}, \qquad C^1_x =\tilde{V}_x\tilde{S}^{-1}_* \big(c_x^{(-1)}\big)^{-J/2},\\
 C^2_x =\tilde{V}_x\tilde{S}^{-1}_{**} {\rm e}^{\pi {\rm i}(\sigma-\theta_{\infty})J/2}\big(c_x^{(-1)}\big)^{-J/2}
\end{gather*}
with
\begin{gather*}
 \tilde{V}_0 =V_0 |_{\sigma\mapsto \sigma^{(-1)}}, \qquad \tilde{V}_x =V_x |_{\sigma\mapsto \sigma^{(-1)}},\\
 \tilde{S}_* = S_{*}\big|_{(\sigma,\theta_0)\mapsto (-\sigma^{(-1)},\theta_x)},\qquad \tilde{S}_{**} = S_{**}\big|_{(\sigma,\theta_0)\mapsto (-\sigma^{(-1)}, \theta_x)}.
\end{gather*}
If $\theta_0$ or $\theta_x \in \Z$, then $\tilde{V}_0$ or $\tilde{V}_x$ is to be replaced by $\hat{V}_0|_{\sigma\mapsto \sigma^{(-1)}}$ or $\hat{V}_x|_{\sigma\mapsto \sigma^{(-1)}}$, respectively, as in Theorem~\ref{thm2.4} (cf.\ Section~\ref{ssc7.5}). The isomonodromy property implies $(M_0, M_x)(\mathbf{c},\sigma)= \big(M_0^{(-1)}, M_x^{(-1)}\big)\big(\mathbf{c}^{(-1)},\sigma^{(-1)}\big)$, which gives the relation between $(\mathbf{c}^{(-1)},\sigma^{(-1)})$ and $(\mathbf{c},\sigma)$.
\end{rem}

The following proposition gives the connection formulas between $\big(\mathbf{c}^{(k\pm 2)}, \sigma ^{(k\pm 2)}\big)$ and \linebreak $\big(\mathbf{c}^{(k)}, \sigma^{(k)}\big)$. This is obtained by deformation of the loops $l_0^{(j)}$, $l_x^{(j)}$ or by action of the braid~$\beta_1^2$ (see \cite[Section~1.2.3]{DM}, \cite[p.~331]{G-Critical}).

\begin{prop}For every $k\in \Z$
\begin{gather*}
 M_0^{(k+2)}= M_x^{(k)} M_0^{(k)} \big(M_x^{(k)}\big)^{-1}, \qquad M_x^{(k+2)}= M_x^{(k)} M_0^{(k)} M_x^{(k)} \big(M_0^{(k)}\big)^{-1}\big(M_x^{(k)}\big)^{-1},\\
M_0^{(k-2)}= \big(M_0^{(k)}\big)^{-1} \big(M_x^{(k)}\big)^{-1} M_0^{(k)} M_x^{(k)} M_0^{(k)},\qquad M_x^{(k-2)}= \big(M_0^{(k)}\big)^{-1} M_x^{(k)} M_0^{(k)}.
\end{gather*}
\end{prop}

\begin{rem}\label{rem2.8} For the solution $(A_0(\mathbf{c}, x), A_x(\mathbf{c}, x))$ with $\sigma_0 =- 2\theta_x-\theta_{\infty}$ in Theorem~\ref{thm2.2}, the corresponding monodromy matrices are given by $(M_0,M_x)|_{\sigma=\sigma_0=-2\theta_x-\theta_{\infty}}$. For the solution $(A_0(c_0,x), A_x(c_0,x))$ we have $\big(M_0, {\rm e}^{\pi {\rm i}\theta_x J}\big)|_{({c_x},\sigma)= (0,\sigma_0)}$.
\end{rem}

\subsection{Fifth Painlev\'e transcendents, zeros and poles}\label{ssc2.3}

From Theorem \ref{thm2.1} and \eqref{1.2} we may derive a solution of (V) written in the form
\begin{gather}\label{2.6}
y= \frac{g_+(\mathbf{c}, \sigma, x)(f_0(\mathbf{c}, \sigma, x)+ \theta_0/2)}{f_+(\mathbf{c}, \sigma, x)(g_0(\mathbf{c}, \sigma, x)+ \theta_x/2)}
\end{gather}
parametrised by $(c,\sigma)=(c_x/c_0,\sigma)$ or $(c',\sigma)=(c_0/c_x, \sigma)$. This is meromorphic in $\Sigma(\sigma,\varepsilon, x_{\infty}, \delta)$ and is expanded into a convergent series in a subdomain of $\Sigma(\sigma,\varepsilon, x_{\infty}, \delta)$. Let $\delta$ be a given positive number such that $\delta <\pi/2$.

\begin{thm}\label{thm2.6} Let $B \subset \C\setminus\{0\}$ and $B_* \subset \C$ be given domains. Suppose that $\mathrm{dist}(\{ -2\theta_0+\theta_{\infty},2\theta_x -\theta_{\infty} \}, B_*) >0$. Then $(\mathrm{V})$ admits a two-parameter family of solutions $\{ y(c, \sigma, x);\, (c,\sigma)\in B\times B_*\}$ such that $y(c,\sigma,x)$ is holomorphic in $(c,\sigma,x)\in B\times D(B_*,\varepsilon', x'_{\infty},\delta)$ and expanded into the convergent series in $\big({\rm e}^x x^{\sigma-1}, {\rm e}^{-x} x^{-\sigma-1}\big)$
\begin{gather*}
y(c, \sigma, x)= c\big(1+ \big[x^{-1}\big]_*\big) {\rm e}^x x^{\sigma} \\
\hphantom{y(c, \sigma, x)= }{} \times \left( 1+ \sum_{n=1}^{\infty} \big(a_n+\big[x^{-1}\big]_*\big)\big({\rm e}^x x^{\sigma-1}\big)^n
+ \sum_{n=1}^{\infty} \big(b_n +\big[x^{-1}\big]_*\big)\big({\rm e}^{-x} x^{-\sigma-1}\big)^n \right)
\end{gather*}
with $a_n$, $b_n \in \Q_0:=\Q\big[\theta_0, \theta_x,\theta_{\infty}, c, c^{-1}, \sigma, (\sigma+2\theta_0 -\theta_{\infty})^{-1}, (-\sigma+ 2\theta_x -\theta_{\infty})^{-1}\big]$, in particular,
\begin{gather*}
a_1= c(-\sigma + \theta_0+\theta_x )/2, \qquad b_1= c^{-1}(\sigma +\theta_0+\theta_x )/2.
\end{gather*}
Here $\varepsilon'=\varepsilon'(B, B_*, \delta)$ $($respectively, $x'_{\infty}=x'_{\infty}(B, B_*, \delta))$ is a sufficiently small $($respectively, large$)$ positive number depending on $(B, B_*,\delta)$, and each $\big[x^{-1}\big]_*$ is represented by an asymptotic series with coefficients in $\Q_0$ valid in $|\arg x -\pi/2|<\pi/2 -\delta$, $|x|>x'_{\infty}$.
\end{thm}

\begin{rem}\label{rem2.9} For $y(c,\sigma, x)$ the corresponding monodromy matrices are obtained by putting $(c_0,c_x)=(1, c)$ in $M_0$, $M_x$ of Theorem~\ref{thm2.3}.
\end{rem}

\begin{rem}\label{rem2.91} The solution $y(c,\sigma,x)$ corresponds to $y_{\mathrm{V},0,*}(\mathbf{c},x)$ of \cite[Theorem~2.10 and Section~2.3]{S1} (see also~\cite{T}) and converges in the domain larger than the previously known one (note that (V)$|_{x\mapsto {\rm i}x}$ is treated in~\cite{S1}).
\end{rem}

\begin{thm}\label{thm2.7}Let $\tilde{B}\subset \C$ be a given domain.
\begin{enumerate}\itemsep=0pt
\item[$(1)$] Suppose that $\theta_x(\theta_0-\theta_x-\theta_{\infty})\not=0$, and set $\sigma_0=-2\theta_x-\theta_{\infty}$. Then $(\mathrm{V})$ admits a one-parameter family of solutions $\{y_+(c,x);\,c\in \tilde{B} \}$ such that $y_+(c,x)$ is holomorphic in $(c,x)\in \tilde{B} \times \Sigma_*(\varepsilon', x'_{\infty}, \delta)$ and is expanded into a convergent series in ${\rm e}^x x^{\sigma_0 -1}$ of the form
\begin{gather*}
y_+(c, x)= \frac 12 (\theta_0-\theta_x-\theta_{\infty}) x^{-1}\big(1+\big[x^{-1}\big]_*\big)\\
\hphantom{y_+(c, x)=}{} + c\big(1+ \big[x^{-1}\big]_*\big) {\rm e}^x x^{\sigma_0} \left( 1+ \sum_{n=1}^{\infty} \big(\tilde{a}_n +\big[x^{-1}\big]_*\big)\big({\rm e}^x x^{\sigma_0-1}\big)^n \right)
\end{gather*}
with $\tilde{a}_n \in \Q_1 :=\Q\big[\theta_0, \theta_x, \theta_{\infty}, c, (\theta_0-\theta_x-\theta_{\infty})^{-1}, \theta_x^{-1}\big]$. Here $\varepsilon'=\varepsilon'(\tilde{B},\delta)$ $($respectively, $x'_{\infty}=x'_{\infty}(\tilde{B},\delta))$ is a sufficiently small $($respectively, large$)$ positive number depending on $(\tilde{B},\delta)$, and each $\big[x^{-1}\big]_*$ is represented by an asymptotic series with coefficients in $\Q_1$ valid in $-(\pi/2-\delta)< \arg x-\pi/2 <\pi -\delta$, $|x|>x'_{\infty}$.

\item[$(2)$] Suppose that $\theta_0(\theta_0-\theta_x+\theta_{\infty})\not=0$, and set $\sigma'_0=2\theta_0+\theta_{\infty}$. Then $(\mathrm{V})$ admits a one-parameter family of solutions $\{y_-(c',x);\,c'\in \tilde{B} \}$ such that $y_-(c',x)$ is holomorphic in $(c',x)\in \tilde{B} \times \Sigma'_*(\varepsilon'', x''_{\infty}, \delta)$ and that the reciprocal is expanded into a convergent series in ${\rm e}^{-x} x^{-\sigma'_0 -1}$ of the form
\begin{gather*}
1/y_-(c', x)= \frac 12 (\theta_0-\theta_x+\theta_{\infty}) x^{-1}\big(1+\big[x^{-1}\big]_*\big)\\
\hphantom{1/y_-(c', x)=}{} + c'\big(1+ \big[x^{-1}\big]_*\big) {\rm e}^{-x} x^{-\sigma'_0} \left( 1+ \sum_{n=1}^{\infty} \big(\tilde{b}_n +\big[x^{-1}\big]_*\big)\big({\rm e}^{-x} x^{-\sigma'_0-1}\big)^n \right)
\end{gather*}
with $\tilde{b}_n \in \Q_2 :=\Q\big[\theta_0, \theta_x, \theta_{\infty}, c', (\theta_0-\theta_x+\theta_{\infty})^{-1}, \theta_0^{-1}\big]$. Here $\varepsilon''=\varepsilon''(\tilde{B},\delta)$ $($respectively, $x''_{\infty}=x''_{\infty}(\tilde{B},\delta))$ is a sufficiently small $($respectively, large$)$ positive number depending on $(\tilde{B},\delta)$, and each $\big[x^{-1}\big]_*$ is represented by an asymptotic series with coefficients in $\Q_2$ valid in $-(\pi-\delta)< \arg x-\pi/2 <\pi/2 -\delta$, $|x|>x''_{\infty}$.
\end{enumerate}
\end{thm}

\begin{rem}\label{rem2.10} The reciprocal $1/y_-(c',x)$ itself solves (V) with $(\theta_0, \theta_x, -\theta_{\infty})$.
\end{rem}

\begin{rem}\label{rem2.11}There exist the asymptotic solutions $y_+(0,x)=(1/2)(\theta_0-\theta_x-\theta_{\infty}) x^{-1}\big(1+\big[x^{-1}\big]_*\big)$ and $y_-(0,x)=2(\theta_0-\theta_x+\theta_{\infty})^{-1} x\big(1+\big[x^{-1}\big]_*\big)$ in the sector $|\arg x-\pi/2|<\pi-\delta$.
\end{rem}

\begin{rem}\label{rem2.12}For $\sigma= 2\theta_x-\theta_{\infty}$ (respectively, $\sigma=-2\theta_0 +\theta_{\infty} $) as well, under the condition $\big(\theta_0^2-(\theta_x-\theta_{\infty})^2\big)\theta_x\not=0$ (respectively, $\big(\theta_x^2-(\theta_0-\theta_{\infty})^2\big)\theta_0\not=0$), there exists a family of solutions $\{\tilde{y}_-(c',x);\, c'\in\tilde{B}\}$ (respectively, $\{\tilde{y}_+(c,x);\, c\in\tilde{B}\}$) such that
\begin{gather*}
 1/\tilde{y}_-(c', x)= - \frac 12 (\theta_0-\theta_x+\theta_{\infty}) x^{-1} \big(1+\big[x^{-1}\big]_*\big)\\
\hphantom{1/\tilde{y}_-(c', x)=}{} + c'\big(1+ \big[x^{-1}\big]_*\big) {\rm e}^{-x} x^{-2\theta_x+ \theta_{\infty}} \left( 1+ \sum_{n=1}^{\infty} \big(\tilde{\tilde{b}}_n +\big[x^{-1}\big]_*\big)
\big({\rm e}^{-x} x^{-2\theta_x +\theta_{\infty}-1}\big)^n \right)\\
 \Bigg(\text{respectively,} \quad \tilde{y}_+(c, x)= - \frac 12 (\theta_0-\theta_x-\theta_{\infty}) x^{-1}\big(1+\big[x^{-1}\big]_*\big)\\
 \qquad{} + c\big(1+ \big[x^{-1}\big]_*\big) {\rm e}^{x} x^{-2\theta_0 +\theta_{\infty}} \left( 1+ \sum_{n=1}^{\infty} \big(\tilde{ \tilde{a}}_n +\big[x^{-1}\big]_*\big)
\big({\rm e}^{x} x^{-2\theta_0+ \theta_{\infty}-1}\big)^n \right) \ \Bigg)
\end{gather*}
in $\Sigma'_*(\varepsilon'',x_{\infty}'', \delta)$ (respectively, $\Sigma_*(\varepsilon',x_{\infty}', \delta)$).
\end{rem}

\begin{rem}\label{rem2.13} Using \eqref{1.2} we may derive convergent series representations for $z(x)$ and $u(x)= x^{-\theta_{\infty}} u_{\mathrm{AK}}(x)$, which are parametrised by $(\sigma, c_x/c_0)$ and $(\sigma, c_x/c_0, c_0)$, respectively.
\end{rem}

From the quotient expression \eqref{2.6} we obtain a sequence of zeros or poles of $y(c,\sigma,x)$.

\begin{thm}\label{thm2.8}Let $R_0$ be a given positive number.
\begin{enumerate}\itemsep=0pt
\item[$(1)$] Suppose that $ \theta_x(\theta_0 \pm \theta_x -\theta_{\infty})\not=0$. Then there exists a small positive number $\varepsilon_0 =\varepsilon_0(R_0) =\varepsilon_0(R_0,\theta_0, \theta_x, \theta_{\infty})$ such that, for every $(c,\sigma)\in \C^2$ satisfying $0<|c|<R_0$, $\sigma\not= \pm 2\theta_0 +\theta_{\infty}, \pm 2\theta_x-\theta_{\infty}$, $|\sigma +2\theta_x+\theta_{\infty}| < R_0|c|$ $($respectively, $|\sigma + 2\theta_0 -\theta_{\infty}| < R_0|c|)$ and $|4c/(\sigma +2\theta_0 -\theta_{\infty})|<\varepsilon_0$ $($respectively, $|4c/(\sigma+2\theta_x +\theta_{\infty})|<\varepsilon_0)$, the solution $y(c,\sigma, x) $ has a sequence of zeros $\big\{x_m^{(0)} \big\}$, $m\ge m_0$, such that
\begin{gather*}
x^{(0)}_m = 2m \pi {\rm i}-(\sigma+1)\log(2m\pi {\rm i})-\log( \rho_0(\sigma)c) +O\big(m^{-1}\log m\big),
\end{gather*}
where $m_0$ is some large positive integer and $\rho_0(\sigma)= -4/(\sigma+2\theta_0 -\theta_{\infty})$ $($respectively, $-4/(\sigma + 2\theta_x +\theta_{\infty}))$.

\item[$(2)$] Suppose that $\theta_0(\pm\theta_0-\theta_x +\theta_{\infty}) \not=0$. Then there exists a small positive number $\varepsilon'_0 =\varepsilon'_0(R_0) =\varepsilon'_0(R_0,\theta_0, \theta_x, \theta_{\infty})$ such that, for every $(c,\sigma)\in \C^2$ satisfying $|c|>1/R_0$, $\sigma\not= \pm 2\theta_0 +\theta_{\infty}, \pm 2\theta_x-\theta_{\infty}$, $|c(\sigma - 2\theta_0 -\theta_{\infty})| < R_0$ $($respectively, $|c(\sigma - 2\theta_x +\theta_{\infty})| < R_0)$ and $\big|4c^{-1}/(\sigma-2\theta_x +\theta_{\infty})\big|<\varepsilon'_0$ $($respectively, $\big|4c^{-1}/(\sigma-2\theta_0 -\theta_{\infty})\big|<\varepsilon'_0\,)$, the solution $y(c,\sigma, x) $ has a sequence of poles $\big\{x_m^{(\infty)} \big\}$, $m\ge m_0$, such that
\begin{gather*}
x^{(\infty)}_m = 2m \pi {\rm i}-(\sigma-1)\log(2m\pi {\rm i})-\log( \rho_{\infty}(\sigma)c) +O\big(m^{-1}\log m\big),
\end{gather*}
where $\rho_{\infty}(\sigma)= -(\sigma -2\theta_x +\theta_{\infty})/4$ $($respectively, $-(\sigma-2\theta_0 -\theta_{\infty})/4)$.
\end{enumerate}
\end{thm}

For one-parameter solutions we have
\begin{thm}\label{thm2.9} Let $y_+(c,x)$ and $y_-(c',x)$ be the solutions given above.
\begin{enumerate}\itemsep=0pt
\item[$(1)$] Suppose that $\theta_x(\theta_0-\theta_x -\theta_{\infty})\not=0$. If $c\not=0$ is sufficiently small, then $y_+(c,x) $ has a~sequence of zeros $\big\{x_m^{+(0)} \big\}$, $m\ge m_0$, such that
\begin{gather*}
x^{+(0)}_m = 2m \pi {\rm i}-(\sigma+1)\log(2m\pi {\rm i})-\log(-{2c}/(\theta_0-\theta_x -\theta_{\infty})) +O\big(m^{-1}\log m\big),
\end{gather*}
where $m_0$ is some large positive integer.
\item[$(2)$] Suppose that $\theta_0(\theta_x-\theta_0 -\theta_{\infty})\not=0$. If $c'\not=0$ is sufficiently small, then $y_-(c',x) $ has a~sequence of poles $\big\{x_m^{-(\infty)} \big\}$, $m\ge m_0$, such that
\begin{gather*}
x^{-(\infty)}_m = 2m \pi {\rm i}-(\sigma-1)\log(2m\pi {\rm i})+\log({2c'}/(\theta_x-\theta_0 -\theta_{\infty})) +O\big(m^{-1}\log m\big).
\end{gather*}
\end{enumerate}
\end{thm}

\begin{rem}\label{rem2.15}To $y(c,\sigma,x)$ of Theorem \ref{thm2.6} applying the B\"acklund transformation and the substitution~$\pi$: $ (\theta_0-\theta_x, \theta_0+\theta_x,\theta_{\infty})\mapsto (1-\theta_{\infty}, 1-\theta_0+\theta_x, \theta_0+\theta_x -1)$, we obtain another solution of (V) given by
\begin{gather*}
\hat{y}(c,\sigma, x)^{\pi} = \frac{Y(x, y(c,\sigma, x))^{\pi}}{ 1+ Y(x, y(c,\sigma, x))^{\pi}}
\end{gather*}
with
\begin{gather*}
Y(x,y) = x^{-1}(y-1) \big( (A_x)_{11} +\theta_x/2 -((A_x)_{11} -\theta_x/2)y^{-1}\big)\\
\hphantom{Y(x,y)}{} = -2(A_x)_{11} x^{-1} + ((A_x)_{11}+\theta_x/2)x^{-1}y + ((A_x)_{11}-\theta_x/2)x^{-1}y^{-1}
\end{gather*}
(cf.\ \cite{Gromak}, \cite[Lemma 6.1]{S3}). This is expressed as
\begin{gather*}
 \hat{y}(c,\sigma,x)^{\pi} = \big[x^{-1}\big] -\big((\sigma+2\theta_0-\theta_{\infty}-1)/4 +\big[x^{-1}\big]\big) c{\rm e}^x x^{\sigma-1}+ \sum_{n=2}^{\infty} [1]_*\big({\rm e}^x x^{\sigma-1}\big)^n\\
\hphantom{\hat{y}(c,\sigma,x)^{\pi} =}{} -\big((\sigma+2\theta_x+\theta_{\infty}-1)/4+\big[x^{-1}\big]\big) c^{-1}{\rm e}^{-x} x^{-\sigma-1}
+ \sum_{n=2}^{\infty} [1]_*\big({\rm e}^{-x} x^{-\sigma-1}\big)^n
\end{gather*}
for $\big|{\rm e}^xx^{\sigma-1}\big|<\varepsilon''$, $\big|{\rm e}^{-x}x^{-\sigma-1}\big| <\varepsilon''$, $\varepsilon''$ being sufficiently small, and admits a sequence of zeros $\{ \tilde{x}_m \}$ with $\tilde{x}_m= -\log c + 2m\pi {\rm i} -\sigma\log(2m\pi {\rm i}) +O\big(m^{-1}\log m\big)$ in the domain $|{\rm e}^xx^{\sigma}|$, $|{\rm e}^{-x}x^{-\sigma}| \ll 1$.
\end{rem}

\section{Families of series}\label{sc3}

\subsection[Family $\mathfrak{A}$]{Family $\boldsymbol{\mathfrak{A}}$}\label{ssc3.1}

Let $B_0$, $B_x$ and $B_*$ be as in Section \ref{sc2}, and $\Sigma_0(x_{\infty}, \delta)$ the sector $|\arg x-\pi/2|<\pi/2-\delta$, $|x|>x_{\infty}$. Denote by $ \hat{\mathfrak{A}} = \hat{\mathfrak{A}}(B_0, B_x, B_*,\Sigma_0(x_{\infty},\delta))$ the family of pairs $\big(\phi, \big\{p^+_n(x), p^-_n(x), p_0(x)\big\}_{n\in \N} \big)$, where $\phi$ is a~formal series of the form
\begin{gather*}
\phi=\phi(\mathbf{c}, \sigma, x) = \sum_{n=1}^{\infty} p^+_n(x)\big({\rm e}^x x^{\sigma-1}\big)^n + \sum_{n=1}^{\infty} p^-_n(x)\big({\rm e}^{-x} x^{-\sigma-1}\big)^n +p_0(x)x^{-1},
\end{gather*}
and $p^{+}_n(x)$, $p^-_n(x)$ and $p_0(x)$ are holomorphic in $(\mathbf{c}, \sigma, x)=(c_0,c_x, \sigma, x)\in B_0\times B_x \times B_*\times \Sigma_0(x_{\infty},\delta)$ and admit asymptotic representations
\begin{gather*}
p^+_n(x) \sim \sum_{m=0}^{\infty} p^+_{nm} x^{-m}, \qquad p^-_n(x) \sim \sum_{m=0}^{\infty} p^-_{nm} x^{-m}, \qquad p_0(x) \sim \sum_{m=0}^{\infty}p_{0m}x^{-m}
\end{gather*}
with coefficients $p^{\pm}_{nm}=p^{\pm}_{nm}(\mathbf{c},\sigma)$, $p_{0m}= p_{0m}(\mathbf{c},\sigma) \in \Q_*$ uniformly in $(\mathbf{c},\sigma) \in B_0 \times B_x \times B_*$ as $x\to\infty$ through the sector~$\Sigma_0(x_{\infty}, \delta)$. Note the example $p_1(x) {\rm e}^x x^{\sigma-1} + p_0(x) x^{-1} \equiv 0\cdot {\rm e}^x x^{\sigma-1} + 0\cdot x^{-1}$ with $p_1(x)= {\rm e}^{2{\rm i} x/\delta} \sim 0$, $p_0(x) = - {\rm e}^{(2{\rm i}/\delta +1) x} x^{\sigma} \sim 0$ in $\Sigma_0(x_{\infty}, \delta)$. To avoid such an ambiguity $\hat{\mathfrak{A}}$ is defined as the set of the pairs as above. For simplicity, however, keeping the strict definition above in mind, we regard and deal with $\hat{\mathfrak{A}}$ as the family of the formal series $\phi=\phi(\mathbf{c}, \sigma, x)$. To $\phi\in \hat{\mathfrak{A}}$ written as above, we assign the function
\begin{gather*}
\|\phi \| = \|\phi\|(x,\eta)= \|\phi\|_{\mathbf{c},\sigma}(x,\eta)\\
\hphantom{\|\phi \|}{} =\sum_{n=1}^{\infty} M(p^+_n, |x|) \big|\eta x^{-1}\big|^n + \sum_{n=1}^{\infty} M(p^-_n, |x|) \big|\eta^{-1} x^{-1}\big|^n + M(p_0, |x|) |x|^{-1},
\end{gather*}
where $M(p, |x|)$ is a function of $(\mathbf{c}, \sigma, |x|)$ given by
\begin{gather*}
M(p,|x|): = M_{\mathbf{c},\sigma}(p,|x|) = \sup\{ |p(\mathbf{c},\sigma,\xi)|;\, |\xi|\ge |x|,\, \xi \in \Sigma_0(x_{\infty},\delta) \}.
\end{gather*}
Suppose that $x_{\infty}>\varepsilon^{-1}$. Let $\mathfrak{A}= \mathfrak{A}(B_0, B_x, B_*, \Sigma_0(x_{\infty},\delta), \varepsilon)$ $(\subset \hat{\mathfrak{A}})$ be the family of $\phi\in \hat{\mathfrak{A}}$ such that $\|\phi\|_{\mathbf{c},\sigma}(x,\eta)$ converges uniformly in $(\mathbf{c},\sigma, x, \eta)\in B_0\times B_x\times B_* \times \Xi(\Sigma_0(x_{\infty}, \delta),\varepsilon)$, where
\begin{gather*}
\Xi(\Sigma_0(x_{\infty},\delta),\varepsilon)= \bigcup_{x\in \Sigma_0(x_{\infty},\delta)} \{x\} \times \big\{\eta; \, \big|\eta x^{-1}\big|<\varepsilon,\, \big|\eta^{-1}x^{-1}\big|<\varepsilon \big\}.
\end{gather*}
Let $D(B_*, \varepsilon, x_{\infty}, \delta)$ be as in Section \ref{sc2}. Then, as shown below, the sum and the product are canonically defined in~$\mathfrak{A}$.

\begin{prop} \label{prop3.1}\quad
\begin{enumerate}\itemsep=0pt
\item[$(1)$] Every $\phi \!\in\! \mathfrak{A}(B_0, B_x, B_*, \Sigma_0(x_{\infty},\delta), \varepsilon)$ is holomorphic in $(\mathbf{c},\sigma, x)\!\in\! B_0\times B_x \times D(B_*, \varepsilon, x_{\infty}, \delta)$, and satisfies $|\phi(\mathbf{c}, \sigma, x)| \le \|\phi\|_{\mathbf{c},\sigma}(x, {\rm e}^xx^{\sigma})$.
\item[$(2)$] Let $\phi$, $\psi \in \mathfrak{A}$. Then $\phi+\psi$, $\phi\psi \in \mathfrak{A}$, and $\|\phi + \psi\| \le \|\phi\| +\|\psi\|$, $\|\phi\psi \|\le \|\phi\| \|\psi\|$. If $a=a(\mathbf{c},\sigma) \in \Q_*$, then $a\phi \in \mathfrak {A}$ and $\|a\phi\| = |a| \|\phi \|$.
\end{enumerate}
\end{prop}

\begin{proof}Suppose that $\phi, \psi \in\mathfrak{A}$ are written as
\begin{gather*}
\phi = \sum_{n=1}^{\infty} p_n^+(x)\big({\rm e}^x x^{\sigma-1}\big)^n + \sum_{n=1}^{\infty} p_n^-(x)\big({\rm e}^{-x} x^{-\sigma-1}\big)^n + p_0(x)x^{-1},\\
\psi = \sum_{n=1}^{\infty} q_n^+(x)\big({\rm e}^x x^{\sigma-1}\big)^n + \sum_{n=1}^{\infty} q_n^-(x)\big({\rm e}^{-x} x^{-\sigma-1}\big)^n + q_0(x)x^{-1}.
\end{gather*}
It is natural to set
\begin{gather*}
\phi\psi = \sum_{n=1}^{\infty} \varpi_n^+(x)\big({\rm e}^x x^{\sigma-1}\big)^n + \sum_{n=1}^{\infty} \varpi_n^-(x)\big({\rm e}^{-x} x^{-\sigma-1}\big)^n + \varpi_0(x)x^{-1},
\end{gather*}
where each coefficient as a formal series is given by
\begin{gather*}
\varpi^{\pm}_n(x) = \sum_{\nu=1}^{n-1} p^{\pm}_{\nu}(x) q^{\pm}_{n-\nu}(x)+ x^{-1} \big(p_n^{\pm}(x)q_0(x) + p_0(x)q_n^{\pm}(x) \big)+ \sum_{\nu=n+1}^{\infty} x^{-2(\nu-n)} \varpi_{n,\nu}^{\pm}(x),\\
 \varpi_0(x) = x^{-2} p_0(x)q_0(x)+ \sum_{\nu=1}^{\infty} x^{-2\nu} \varpi_{\nu}^{0}(x)
\end{gather*}
with $\varpi_{n,\nu}^{\pm}(x)= p^{\pm}_{\nu}(x) q^{\mp}_{\nu-n}(x) +p^{\mp}_{\nu-n}(x) q^{\pm}_{\nu}(x)$, $\varpi_{\nu}^{0}(x)= p^{+}_{\nu}(x) q^{-}_{\nu}(x) +p^{-}_{\nu}(x) q^{+}_{\nu}(x)$. If $(\mathbf{c}, \sigma, x) \in B_0 \times B_x \times B_* \times \Sigma_0(x_{\infty}, \delta)$, then, by the definition of $\|\cdot\|$, for $\big|\eta x^{-1}\big|, \big|\eta^{-1} x^{-1}\big|<\varepsilon$
\begin{gather*}
|\varpi_{n,\nu}^+(x)| \le M(p^+_{\nu}, |x|) M(q^-_{\nu-n}, |x|) + M(p^-_{\nu-n},|x|) M(q^+_{\nu}, |x|)\\
\hphantom{|\varpi_{n,\nu}^+(x)|}{} \le 2\|\phi\|(x,\eta) \|\psi\|(x,\eta)\big|\eta x^{-1}\big|^{-\nu} \big|\eta^{-1} x^{-1}\big|^{-(\nu -n)} \ll \varepsilon^{-2\nu+ n},
\end{gather*}
and hence $\big|x^{-2(\nu-n)}\varpi_{n,\nu}^+(x) \big|\ll \varepsilon^{-n} (\varepsilon x) ^{-2(\nu-n)}$, the implied constant not depending on $(\varepsilon, x)$. This implies that $\varpi^+_n(x) $ is holomorphic in $(\mathbf{c},\sigma, x) \in B_0 \times B_x \times B_* \times \Sigma_0(x_{\infty},\delta)$. Furthermore, for a given integer $N\ge 1$, we have $\Sigma_{\nu\ge n+N} |x^{-2(\nu-n)}\varpi_{n,\nu}^+(x) |\ll \varepsilon^{-n} (\varepsilon x)^{-2N}$ in the domain $B_0\times B_x \times B_* \times \Sigma_0 (2x_{\infty},\delta)$, which implies that $\varpi_n^+(x)$ is represented by an asymptotic series in $\Sigma_0(x_{\infty},\delta)$ uniformly in~$(\mathbf{c}, \sigma) \in B_0\times B_x \times B_*$. Thus we have shown that $\phi\psi \in \hat{\mathfrak{A}}$. To evaluate $\|\phi\psi\|$ we note that, for $\nu > n \ge 1$,
\begin{gather*}
\big\| p_{\nu}^+ (x)\big({\rm e}^x x^{\sigma-1}\big)^{\nu} \cdot q_{\nu-n}^-(x) \big({\rm e}^{-x}x^{-\sigma-1}\big)
^{\nu-n} \big\| =\big\| \big({\rm e}^x x^{\sigma-1}\big)^{n} p_{\nu}^+(x) q_{\nu-n}^-(x)x^{-2(\nu-n)} \big\|\\
\qquad{} = \big|\eta x^{-1}\big|^{n} \sup \big\{ \big|p^+_{\nu}(\xi) q_{\nu-n}^-(\xi) \xi^{-2(\nu-n)}\big|; \,
|\xi| \ge |x|,\, \xi \in \Sigma_0(x_{\infty}, \delta) \big\}\\
\qquad{} \le \big|\eta x^{-1}\big|^{\nu} \big|\eta^{-1} x\big|^{\nu-n} |x|^{-2(\nu-n)}
\sup \big\{ |p^+_{\nu}(\xi) |; \, \ldots \big\} \sup \big\{ |q^-_{\nu-n}(\xi) |; \, \ldots \big\}\\
\qquad{} =\big|\eta x^{-1}\big|^{\nu} M(p^+_{\nu}, |x|) \big|\eta^{-1} x^{-1}\big|^{\nu-n} M\big(q^-_{\nu-n}, |x|\big)\\
\qquad{} =\big\| p_{\nu}^+(x)\big({\rm e}^x x^{\sigma-1}\big)^{\nu}\big\|\big\| q_{\nu-n}^-(x)\big({\rm e}^{-x}x^{-\sigma-1}\big)^{\nu-n} \big\|,
\end{gather*}
and that, for $\nu \ge 1$,
\begin{gather*}
\big\| p_{\nu}^+(x)\big({\rm e}^x x^{\sigma-1}\big)^{\nu} \cdot q_0(x)x^{-1} \big\|
= \big|\eta x^{-1}\big|^{\nu} \sup \big\{ \big|p^+_{\nu}(\xi) q_0(\xi) \xi^{-1}\big|; \, |\xi| \ge |x|,\, \xi \in \Sigma_0(x_{\infty}, \delta) \big\}\\
\qquad{} \le \big|\eta x^{-1}\big|^{\nu} |x|^{-1} \sup \{ |p^+_{\nu}(\xi) |; \, \ldots \} \sup \{ |q_0(\xi) |; \, \ldots \}\\
\qquad {}=\big\| p_{\nu}^+(x)\big({\rm e}^x x^{\sigma-1}\big)^{\nu}\big\| \big\| q_0(x) x^{-1} \big\|.
\end{gather*}
Using these inequalities we have $\|\phi\psi\|\le \|\phi\| \|\psi\|$.
\end{proof}

\begin{exa}\label{exa3.1} In the sector $|\arg x- \pi/2|<\pi-\delta$, we may take a path $\gamma_*(x)$ $(\ni \xi)$ starting from $x$ in such a way that $\big|{\rm e}^{-\xi}{\rm e}^x\big|$ or $\big|{\rm e}^{\xi} {\rm e}^{-x}\big|$ is monotone decreasing and decays exponentially along $\gamma_*(x)$. Then for $n\in \N$
\begin{gather*}
 -\int_{\gamma_*(x)} \big({\rm e}^{\xi}\xi^{\sigma-1}\big)^n {\rm d}\xi =\big({\rm e}^x x^{\sigma-1}\big)^n P_n^{(1)}(x), \\
 -\int_{\gamma_*(x)} \big({\rm e}^{-\xi}\xi^{-\sigma-1}\big)^n {\rm d}\xi = \big({\rm e}^{-x} x^{-\sigma-1}\big)^n P_n^{(2)}(x),
\end{gather*}
where
$P^{(\iota)}_n(x) \sim \sum\limits_{m=0}^{\infty} p^{(\iota)}_{nm}(\sigma)x^{-m}$ with $p^{(\iota)}_{nm}(\sigma) \in \Q[\sigma]$, $\iota=1,2$, as $x\to \infty$ through the sector $|\arg x -\pi/2 |<\pi-\delta$. Furthermore if $g(x)\sim \sum\limits_{m=0}^{\infty}g_mx^{-m}$ as $x\to\infty$ through this sector, we have
\begin{gather*}
-\int_{\gamma_*(x)} \big({\rm e}^{\xi}\xi^{\sigma-1}\big)^n g(\xi) {\rm d}\xi = \big({\rm e}^x x^{\sigma-1}\big)^n G_n(x), \qquad G_n(x)\sim \sum_{m=0}^{\infty} G_{nm}(\sigma)x^{-m}.
\end{gather*}
Clearly these integrals belong to $\mathfrak{A}$.
\end{exa}

\begin{prop}\label{prop3.2} Let $\{\phi_k\}_{k\ge 1} \subset \mathfrak{A}(B_0,B_x, B_*, \Sigma_0(x_{\infty}, \delta), \varepsilon)$ with
\begin{gather*}
\phi_k =\sum_{n=1}^{\infty} p^{(k)+}_n(x) \big({\rm e}^x x^{\sigma-1}\big)^n + \sum_{n=1}^{\infty} p^{(k)-}_n(x) \big({\rm e}^{-x} x^{-\sigma-1}\big)^n + p_0^{(k)} (x) x^{-1}
\end{gather*}
be such that $\|\phi_k\| \ll \tilde{\varepsilon}^{k-N}|x|^{-N+k_0}$ for every pair of integers $(k, N)$ with $1\le k_0 \le N \le k$ if $\big|\eta x^{-1}\big|< \varepsilon$, $\big|\eta^{-1} x^{-1}\big|<\varepsilon$, $|x|>x_{\infty}$, where $k_0$ is a given positive integer, $\tilde{\varepsilon}$ is some number satisfying $0<\tilde{\varepsilon} <1$, and the implied constant does not depend on $k$. Then, $\phi^{\infty} =\sum\limits_{k=1}^{\infty} \phi_k$ also belongs to $ \mathfrak{A}(B_0,B_x,B_*, \Sigma_0(x_{\infty},\delta), \varepsilon)$ and $\phi^{\infty}$ is given by
\begin{gather*}
\phi^{\infty} =\sum_{n=1}^{\infty} p^{\infty+}_n(x) \big({\rm e}^x x^{\sigma-1}\big)^n +
\sum_{n=1}^{\infty} p^{\infty-}_n(x) \big({\rm e}^{-x} x^{-\sigma-1}\big)^n + p_0^{\infty}(x) x^{-1}
\end{gather*}
with $p^{\infty \pm}_n(x) =\sum\limits_{k=1}^{\infty} p_n^{(k)\pm}(x) =[1]$, $p^{\infty }_0(x) =\sum\limits_{k=1}^{\infty} p_0^{(k)}(x) =[1]$.
\end{prop}

\begin{proof} By the condition with $N=k_0$, $\sum\limits^{\infty}_{k=k_0} \phi_k$ as a double series converges uniformly and absolutely, and hence $p^{\infty \pm}_n(x) = \sum\limits^{\infty}_{k=1} p_n^{(k)\pm}(x)$, $p^{\infty}_0(x)=\sum\limits^{\infty}_{k=1} p_0^{(k)}(x)$ are holomorphic in $\Sigma_0(x_{\infty}, \delta)$. It is sufficient to show that $\phi^{\infty} \in \hat{\mathfrak{A}} (B_0,B_x, B_*, \Sigma_0(x_{\infty},\delta), \varepsilon)$, that is, $p^{\infty \pm}_n(x)$ and $p_0^{\infty}(x)$ may be represented by asymptotic series. For a given positive number $N\ge k_0$, if $k\ge N$ and if $\big|\eta x^{-1}\big|<\varepsilon$, $\big|\eta^{-1} x^{-1}\big|<\varepsilon$, $|x|>x_{\infty} > \varepsilon^{-1}$, then
\begin{gather*}
M\big(p_n^{(k)+}, |x|\big) \le \|\phi_k\|(x,\eta) \big|\eta x^{-1}\big|^{-n} \ll \big|\eta x^{-1}\big|^{-n} \tilde{\varepsilon}^{k-N} |x|^{-N+k_0},
\end{gather*}
the implied constant possibly depending on $N$ only. For $x\in \Sigma_0(x_{\infty},\delta)$, letting $\eta \to \varepsilon x$, we have $\big|p_n^{(k)+}(x)\big| \ll \varepsilon^{-n} \tilde{\varepsilon}^{k-N} |x|^{-N+k_0}$, which means $\Big| \sum\limits_{k=N}^{\infty} p_n^{(k)+}(x) \big| \ll \varepsilon^{-n} (1-\tilde{\varepsilon})^{-1} |x|^{-N+k_0}$ in $\Sigma_0(x_{\infty}, \delta)$. Substitution of this and $p^{(k)+} _n(x)$ with $k\le N-1$ into $p^{\infty+}_n(x)$ yields the asymptotic representation of $p^{\infty+}_n(x)$. Thus we obtain the proposition.
\end{proof}

\begin{rem}\label{rem3.1} Under the supposition of Proposition \ref{prop3.2}, for $p^{(k)\pm}_n (x) \sim \sum\limits_{m=0}^{\infty} p^{(k)\pm}_{nm}x^{-m}$ and $p^{(k)}_0 (x) \sim \sum\limits_{m=0}^{\infty} p^{(k)}_{0m}x^{-m}$, the asymptotic representations of $p^{\infty\pm}_n(x)$ and $p^{\infty}_0(x)$ are written in the form $p^{\infty\pm}_n(x) \sim \sum\limits_{m=0}^{\infty} p^{\infty\pm}_{nm}x^{-m}$ and $p^{\infty}_0(x) \sim \sum\limits_{m=0}^{\infty} p^{\infty}_{0m}x^{-m}$ with coefficients $p^{\infty\pm}_{nm}=\sum\limits_{k=1}^{m+k_0} p^{(k)\pm}_{nm}$ and $p^{\infty}_{0m} =\sum\limits^{m+k_0}_{k=1} p^{(k)}_{0m}$, respectively.
\end{rem}

The following sums of the form $\sum\limits_{k=1}^{\infty}\phi^k$ belong to $\mathfrak{A}(B_0, B_x, B_*, \Sigma_0(x_{\infty},\delta),\varepsilon)$, although $\big\|\phi^k\big\|$ does not necessarily satisfy the condition of Proposition~\ref{prop3.2}.

\begin{exa}\label{exa3.2} For $\phi=p^+(x) {\rm e}^xx^{\sigma-1} +p^-(x) {\rm e}^{-x} x^{-\sigma-1} \in \mathfrak{A}(B_0,B_x,B_*,\Sigma_0(x_{\infty},\delta), \varepsilon)$, if $\varepsilon'<\varepsilon$ is sufficiently small, then $\sum\limits_{k=1}^{\infty} \phi^k \in \mathfrak{A}(B_0,B_x,B_*,\Sigma_0(x_{\infty},\delta), \varepsilon')$.

{\bf Verification}. Suppose that $p^{\pm}(x) \sim p_0^{\pm}+ \sum\limits_{m=1} ^{\infty}p_m^{\pm} x^{-m}$ and that $\big|p^{\pm}(x)-p_0^{\pm}\big|<r_0$ for $(\mathbf{c},\sigma, x) \in B_0\times B_x \times B_* \times \Sigma_0(x_{\infty}, \delta)$. We may choose $\varepsilon'$ in such a way that $\sum\limits_{k=1}^{\infty}\big( (|p^+_0|+|\xi_+|)|\eta_+| + (|p^-_0|+|\xi_-|)|\eta_-|\big)^k$ is convergent for $(\mathbf{c},\sigma, \xi_+,\xi_-, \eta_+,\eta_-)$ satisfying $(\mathbf{c},\sigma)\in B_0\times B_x\times B_*$, $|\xi_{\pm}|<2r_0$, $|\eta_{\pm}|<\varepsilon'$. Then $\sum\limits_{k=1}^{\infty}\big( (p^+_0+\xi_+)\eta x^{-1}+ (p^-_0+\xi_-)\eta^{-1}x^{-1}\big)^k$ converges absolutely for $(\mathbf{c},\sigma, \xi_+, \xi_-)$ as above and for $(\eta, x)$ satisfying $\big|\eta x^{-1}\big|, \big|\eta^{-1} x^{-1}\big|<\varepsilon'$, $|x|>1/\varepsilon'$, and is written in the form
\begin{gather*}
\sum_{n=1}^{\infty} \big(\pi^+_n\big(\xi_+,\xi_-, x^{-1}\big)\big(\eta x^{-1}\big)^n + \pi^-_n\big(\xi_+,\xi_-, x^{-1}\big)\big(\eta^{-1} x^{-1}\big)^n \big) +\pi_0\big(\xi_+,\xi_-, x^{-1}\big) x^{-1},
\end{gather*}
where $\pi^{\pm}_n\big(\xi_+,\xi_-, x^{-1}\big) =\sum\limits_{m=0}^{\infty}\pi_{nm}^{\pm} (\xi_+,\xi_-) x^{-m}$, $\pi_0\big(\xi_+,\xi_-, x^{-1}\big) =\sum\limits_{m=0}^{\infty}\pi_{0m}(\xi_+,\xi_-) x^{-m}$ with $\pi^{\pm}_{nm}(\xi_+,\xi_-) =\sum\limits_{k_1,k_2}\pi_{nmk_1k_2}^{\pm}
\xi_+^{k_1} \xi_-^{k_2}$, $\pi_{0m}(\xi_+,\xi_-) =\sum\limits_{k_1,k_2}\pi_{0mk_1k_2} \xi_+^{k_1} \xi_-^{k_2}$ converge for $|\xi_{\pm}|<2r_0$, $|x|>1/\varepsilon'$. Inserting $\eta={\rm e}^x x^{\sigma}$, $\xi_{\pm}=p^{\pm}(x)-p_0^{\pm} \sim \sum\limits^{\infty}_{m=1} p_m^{\pm}x^{-m}$ into this we have the conclusion.
\end{exa}

\begin{exa}\label{exa3.3} For $\phi=p_1(x){\rm e}^x x^{\sigma-1} + p_2(x) \big({\rm e}^{x} x^{\sigma-1}\big)^2 \in \mathfrak{A}(B_0,B_x,B_*,\Sigma_0(x_{\infty},\delta), \varepsilon)$, we have $\sum\limits_{k=1}^{\infty}\phi^k \in
\mathfrak{A}(B_0,B_x,B_*,\Sigma_0(x_{\infty},\delta), \varepsilon)$, because, for each $n\ge 1$, the coefficient of $\big({\rm e}^x x^{\sigma-1}\big)^n$ is a~finite sum of asymptotic series in $x^{-1}$.
\end{exa}

The following lemma is used in evaluating the primitive function of $x^{-1} \phi$.

\begin{lem}\label{lem3.3} For every $(\sigma ,x )\in B_* \times \Sigma_0(\tilde{x} _{\infty}, \delta)$,
\begin{gather*}
\big|{\rm e}^x x^{\sigma-1}\big|^{-n} \int_{\gamma(x)} \big|{\rm e}^{\xi}\xi^{\sigma-1}\big|^n\frac{|{\rm d}\xi|}
{|\xi|}, \
\big|{\rm e}^{-x} x^{-\sigma-1}\big|^{-n} \int_{\gamma(x)} \big|{\rm e}^{-\xi}\xi^{-\sigma-1}\big|^n
\frac{|{\rm d}\xi|}{|\xi|} \le \frac 1n\left( 1+\frac{2} {\sin\delta} \right),
\end{gather*}
$n=1,2,3,\ldots$. Here $\tilde{x}_{\infty}=\tilde{x}_{\infty}(B_*,\delta)$ is a sufficiently large number depending on $(B_*,\delta)$, and~$\gamma(x)$ is a path with the properties:{\samepage
\begin{enumerate}\itemsep=0pt
\item[$(i)$] $\gamma(x)$ starts from $x$ and tends to $\infty {\rm e}^{\pi {\rm i}/2}$, and $\xi \in \gamma(x)$ is given by $\xi={\rm i} \tau \exp(-{\rm i}\theta(\tau))$ with $|x|\le \tau <\infty$, where $\theta(\tau)$ satisfies $\theta(|x|)=\pi/2- \arg x$, $\theta(\tau) \to 0$ as $\tau\to \infty$;
\item[$(ii)$] for every $\xi \in \gamma(x)$, $|{\rm e}^x x^{\sigma}|= \big|{\rm e}^{\xi} \xi^{\sigma}\big|$;
\item[$(iii)$] for every $\xi \in \gamma(x)$, $|\xi|\ge |x|$, and $|\pi/2 -\arg \xi|< \pi/2-\delta$.
\end{enumerate}}
\end{lem}

\begin{proof} We substitute $\xi= {\rm i} \tau \exp(-{\rm i}\theta(\tau))=\tau \exp({\rm i}(\pi/2-\theta(\tau)))$ into $|{\rm e}^x x^{\sigma}|= \big|{\rm e}^{\xi}\xi^{\sigma}\big|$ to obtain
\begin{gather}\label{3.1}
\re x +\re \sigma \cdot \log|x| - \im \sigma \cdot \arg x =\tau\sin\theta(\tau) + \re\sigma \cdot \log\tau -\im\sigma \cdot (\pi/2- \theta(\tau)).
\end{gather}
Choose $\tilde{x}_{\infty}=\tilde{x}_{\infty}(B_*,\delta)$ in such a way that, for every $\sigma\in B_*$,
\begin{gather}\label{3.2}
\tilde{x}_{\infty} > 100 (|\sigma|+1) (\tan\delta + 1/\sin\delta) > 200(|\sigma| +1).
\end{gather}
Then, for any $x \in \Sigma_0( \tilde{x}_{\infty}, \delta)$, the function $\theta(\tau)$ satisfying \eqref{3.1} and $\theta(|x|)=\pi/2-\arg x$ is uniquely determined near $\tau=|x|$ by the implicit function theorem. Note that
\begin{gather}\label{3.3}
 \theta'(\tau) = - \frac{\sin\theta(\tau) + \re\sigma \cdot \tau^{-1} }{\tau\cos\theta(\tau) + \im \sigma },\\
\label{3.4}
(\im\xi(\tau))' =\cos\theta(\tau) -\tau\theta'(\tau) \sin\theta(\tau)= \frac{ \tau + \im\sigma \cdot\cos\theta(\tau) + \re\sigma\cdot \sin\theta(\tau) }{\tau\cos\theta(\tau) + \im \sigma },\\
\label{3.5}
 (\re\xi(\tau))' =\sin\theta(\tau) +\tau\theta'(\tau) \cos\theta(\tau)= \frac{ \im\sigma \cdot\sin\theta(\tau) - \re\sigma\cdot \cos\theta(\tau) }{\tau\cos\theta(\tau) + \im \sigma }.
\end{gather}
As long as $\tau \ge |x| > \tilde{x}_{\infty}$, $|\theta(\tau)|< \pi/2 -\delta$, from \eqref{3.2}, \eqref{3.4} and \eqref{3.5} it follows that
\begin{gather*}
(\im\xi(\tau))' \ge \frac{ \tau (1-1/100) }{\tau(1+1/100)\sin\delta } \ge \frac 1{2\sin\delta} >0,\\
| (\re\xi(\tau))'| \le \frac{ 2|\sigma| } {\tau(1-1/100) \sin\delta } < \frac{3(|\sigma| +1)}{\tau\sin\delta},
\end{gather*}
and hence $(\im\xi(\tau))'/|(\re\xi(\tau))'| > \tau(|\sigma|+1)^{-1}/6 > 10 \tan\delta$. This fact implies that $\theta(\tau)$ may be prolonged for $\tau \ge |x| >\tilde{x}_{\infty}$ and that (ii) and (iii) are fulfilled. Then, by~\eqref{3.3} with~\eqref{3.2},
\begin{gather*}
|\theta'(\tau)| \le \frac{1+1/3 } {\tau(1-1/3)\cos\theta(\tau) } \le \frac{2\tau^{-1}}{\sin\delta},
\end{gather*}
and hence $|{\rm d}\xi| =|{\rm d}\xi/{\rm d}\tau| {\rm d}\tau \le \big|{\rm i} {\rm e}^{-{\rm i}\theta(\tau)} + \theta'(\tau) \tau {\rm e}^{-{\rm i}\theta(\tau)} \big|{\rm d}\tau \le (1+|\tau \theta'(\tau)|){\rm d}\tau \le (1+2/\sin\delta) {\rm d}\tau$. Using this and (ii) we obtain
\begin{gather*}
\big|{\rm e}^x x^{\sigma-1}\big|^{-n} \int_{\gamma(x)} \big|{\rm e}^{\xi}\xi^{\sigma-1}\big|^n\frac{|{\rm d}\xi|}
{|\xi|}= |x|^{n} \int_{|x|}^{\infty} \tau^{-n-1} \left(1+\frac{2}{\sin\delta}\right) {\rm d}\tau \le \frac 1n\left( 1+\frac{2} {\sin\delta} \right).
\end{gather*}
This completes the proof.
\end{proof}

\begin{rem}\label{rem3.2}If $|\arg x-3\pi/2| < \pi/2-\delta$, $|x|> \tilde{x}_{\infty}'(B_*,\delta)$, then along the path $\gamma_{3\pi/2}(x)$ defined by $\xi= {\rm i}\tau \exp(-{\rm i}\theta_{3\pi/2} (\tau) )$ with $\theta_{3\pi/2}(\tau) =-\pi+\theta(\tau)$ (cf.~\eqref{3.1}) the same estimates for the integrals are obtained. If $|\arg x +\pi/2| < \pi/2 -\delta$, then $\gamma_{-\pi/2}(x)$ given by $\xi= {\rm i} \tau \exp(-{\rm i} \theta_{-\pi/2}(\tau) )$ with $\theta_{-\pi/2}(\tau) = \pi +\theta(\tau)$ has the same property. These paths are obtained by making the substitutions $(x,\xi, \sigma)\mapsto \big(x{\rm e}^{-\pi {\rm i}}, \xi {\rm e}^{-\pi {\rm i}}, -\sigma\big)$ and $(x,\xi, \sigma)\mapsto \big(x{\rm e}^{\pi {\rm i}}, \xi {\rm e}^{\pi {\rm i}}, -\sigma\big)$, respectively, in the definition of $\gamma(x)$.
\end{rem}

Furthermore we have
\begin{lem}\label{lem3.4} For every $\sigma\in B_*$
\begin{gather*}
\big|{\rm e}^x x^{\sigma}\big|^{-n} \int_{\gamma_{\pi}(x)} \big|{\rm e}^{\xi}\xi^{\sigma} \big|^n{|{\rm d}\xi|} \le \frac 2{n|\cos(\arg x)|} , \qquad n=1,2,3,\ldots,
\end{gather*}
in the sector $|\arg x-\pi|< \pi/2 -\delta$, $|x|> \tilde{x}_{\infty}''(B_*,\delta)$, where $\tilde{x}_{\infty}''(B_*,\delta)$ is sufficiently large and~$\gamma_{\pi}(x)$ is a ray starting from~$x$ and tending to $\infty {\rm e}^{{\rm i} \arg x}$.
\end{lem}

\begin{proof} The path $\gamma_{\pi}(x)$ is given by $\xi=\xi(t)= x+t{\rm e}^{{\rm i}\arg x}$, $0\le t <\infty$. Then
\begin{gather*}
 \big|{\rm e}^{\xi} \xi^{\sigma}\big|^n \big| {\rm e}^x x^{\sigma} \big|^{-n} = \exp( -r(t)), \qquad
- r(t) = n( t \cos(\arg x) +\re\sigma \cdot \log (1+ t/|x|) ).
\end{gather*}
Since $-{\rm d}r/{\rm d}t =n(\cos(\arg x) +\re\sigma/(|x|+t) )$, we have ${\rm d}t/{\rm d}r \le 2/(n|\cos(\arg x)|)$, if $\tilde{x}''_{\infty}(B_*,\delta)$ is sufficiently large. Hence
\begin{gather*}
\big|{\rm e}^x x^{\sigma}\big|^{-n} \int_{\gamma_{\pi}(x)} \big|{\rm e}^{\xi}\xi^{\sigma} \big|^n{|{\rm d}\xi|}
=\int^{\infty}_0 {\rm e}^{-r(t)} \frac{{\rm d}t}{{\rm d}r} {\rm d}r \le \frac 2{n|\cos(\arg x)|} \int^{\infty}_0 {\rm e}^{-r}{\rm d}r
\le \frac 2{n|\cos(\arg x)|},
\end{gather*}
which completes the proof.
\end{proof}

\begin{rem}\label{rem3.3} Similarly, for $\sigma\in B_*$, in the sector $|\arg x|< \pi/2-\delta$, $|x| >\tilde{x}''_{\infty}(B_*,\delta)$, we have
\begin{gather*}
\big|{\rm e}^{-x} x^{-\sigma}\big|^{-n} \int_{\gamma_{0}(x)} \big|{\rm e}^{-\xi}\xi^{-\sigma} \big|^n{|{\rm d}\xi|} \le \frac 2{n\cos(\arg x)},
\end{gather*}
where $\gamma_0(x)$ is a ray starting from $x$ and tending to $\infty {\rm e}^{{\rm i}\arg x}$.
\end{rem}

\begin{lem}\label{lem3.5} Let $B_*$, $\tilde{x}_{\infty}$ and $\gamma(x)$ be as in Lemma~{\rm \ref{lem3.3}} and let $x_{\infty}> \tilde{x}_{\infty}$. Suppose that $p(x)$ is holomorphic in $(\mathbf{c},\sigma, x) \in B_0\times B_x \times B_*\times\Sigma_0(x_{\infty},\delta)$ and admits the asymptotic representation $p(x) \sim \sum\limits_{m=0}^{\infty} p_m x^{-m}$, $p_m \in \Q_*$ uniformly in $(\mathbf{c},\sigma)\in B_0\times B_x \times B_*$ as $x\to\infty$ through $\Sigma_0(x_{\infty},\delta)$. Then, for any $n\in \N$,
\begin{gather*}
I^{+n}_{\gamma(x)}(p(x)) := - \big({\rm e}^x x^{\sigma-1}\big)^{-n} \int_{\gamma(x)}\big({\rm e}^{\xi}\xi^{\sigma-1}\big)^n p(\xi) {\rm d}\xi,\\
I^{-n}_{\gamma(x)}(p(x)) := - \big({\rm e}^{-x} x^{-\sigma-1}\big)^{-n} \int_{\gamma(x)}\big({\rm e}^{-\xi}\xi^{-\sigma-1}\big)^n p(\xi) {\rm d}\xi
\end{gather*}
are holomorphic in $(\mathbf{c},\sigma, x)$ and admit the asymptotic representations
\begin{gather*}
 I^{+n}_{\gamma(x)}(p(x)) \sim \sum_{m=0}^{\infty} P^+_{nm} x^{-m}, \qquad I^{-n}_{\gamma(x)}(p(x)) \sim \sum_{m=0}^{\infty} P^-_{nm} x^{-m}
\end{gather*}
with $P^{\pm}_{nm} \in \Q_*$ uniformly in $(\mathbf{c},\sigma)\in B_0\times B_x \times B_*$ as $x\to\infty$ through $\Sigma_0(x_{\infty},\delta)$. Furthermore, if $p(x)=O\big(x^{-1}\big)$,
\begin{gather*}
 I^{0}_{\gamma(x)}(p(x)) := - x \int_{\gamma(x)}{\xi}^{-1} p(\xi) {\rm d}\xi \sim \sum_{m=0}^{\infty} P^0_{nm} x^{-m}
\end{gather*}
with $P^0_{nm} \in \Q_*$.
\end{lem}

\begin{rem}\label{rem3.4}The integrals $I^{\pm 1}_{\gamma(x)}(p(x))$ are not necessarily absolutely convergent.
\end{rem}

\begin{rem}\label{rem3.5} If $p(x) \sim \sum\limits_{m=0}^{\infty} p_m x^{-m}$ in the sector $|\arg x -3\pi/2 |< \pi/2-\delta$, then
\begin{gather*}
 I^{\pm n}_{\gamma_{3\pi/2}(x)}(p(x)) := - \big({\rm e}^{\pm x} x^{\pm\sigma-1}\big)^{-n} \int_{\gamma_{3\pi/2}(x)}({\rm e}^{\pm\xi} \xi^{\pm \sigma-1})^n p(\xi) {\rm d}\xi
\end{gather*}
with $\gamma_{3\pi/2}(x)$ in Remark \ref{rem3.2} admit asymptotic representations of the same form as above. Furthermore, for $p(x)$ in the sector $|\arg x + \pi/2|<\pi/2-\delta$, we may define $I^{\pm n}_{\gamma_{-\pi/2}(x)}(p(x))$ with $\gamma_{-\pi/2}(x)$ as in Remark~\ref{rem3.2} having the same property.
\end{rem}

\begin{proof}[Proof of Lemma \ref{lem3.5}] For every $(k, n)\in (\N\cup \{ 0\}) \times \N$, integrating by parts and using (ii) of Lemma~\ref{lem3.3}, we have
\begin{gather*}
I_{n,k}(x) := \big({\rm e}^x x^{\sigma-1}\big)^{-n} \int_{\gamma(x)} \big({\rm e}^{\xi} \xi^{\sigma-1}\big)^n \xi^{-k} {\rm d}\xi
= \left[ {\rm e}^{-nx}x^{-n(\sigma-1)} \frac{{\rm e}^{n\xi}}n \xi^{n(\sigma-1)-k} \right] _{\gamma(x)}\\
\hphantom{I_{n,k}(x):=}{} -\frac{n(\sigma-1)-k}{n} \big({\rm e}^x x^{\sigma-1}\big)^{-n} \int_{\gamma(x)}
\big({\rm e}^{\xi} \xi^{\sigma-1}\big)^n \xi^{-k-1} {\rm d}\xi\\
\hphantom{I_{n,k}(x)}{} = -\frac{x^{-k}}n -\frac{n(\sigma-1)-k}n I_{n, k+1}(x)
\end{gather*}
in $\Sigma_0(x_{\infty},\delta)$, because $\big(\big|{\rm e}^{\xi}\xi^{\sigma}\big|/ |{\rm e}^x x^{\sigma}|\big)^n |x|^n |\xi|^{-n-k} \to 0$ as $\xi\to \infty$ along $\gamma(x)$. Furthermore, $I_{n,k+1}(x)$ converges absolutely. If $k \ge 1$, then, by (iii) of Lemma~\ref{lem3.3},
\begin{gather*}
|I_{n,k}(x)|\le \int_{\gamma(x)} \big|{\rm e}^x x^{\sigma-1}\big|^{-n} |{\rm e}^{\xi}\xi^{\sigma-1}|^n |x|^{-(k-1)} \frac{|{\rm d}\xi|}{|\xi|} \ll |x|^{-(k-1)}.
\end{gather*}
Similarly, if $g(\xi) \ll |\xi|^{-k}$,
\begin{gather*}
\left| \int_{\gamma(x)} \big({\rm e}^x x^{\sigma-1}\big)^{-n} \big({\rm e}^{\xi}\xi^{\sigma-1}\big)^n g(\xi) {\rm d}\xi \right| \ll |x|^{-(k-1)}.
\end{gather*}
Combining these facts suitably, we may show that $I^{+n}_{\gamma(x)} (p(x))$ is holomorphic in $(\mathbf{c},\sigma, x)$ and get the asymptotic expansion of $I^{+n}_{\gamma(x)} (p(x))$ as in the lemma.
\end{proof}

Now we are ready to define the primitive function of $\phi \in \hat{\mathfrak{A}}$ or $\mathfrak{A}$. Let $x_{\infty} \ge \tilde{x}_{\infty}$ with $\tilde{x}_{\infty}$ as in Lemma~\ref{lem3.3}. Suppose that
\begin{gather*}
\phi = \sum_{n=1}^{\infty} p_n^+(x) \big({\rm e}^x x^{\sigma-1}\big)^n + \sum_{n=1}^{\infty} p_n^-(x) \big({\rm e}^{-x} x^{-\sigma-1}\big)^n + p_0(x) x^{-1}\\
\hphantom{\phi =}{} \in \hat{\mathfrak{A}} = \hat{\mathfrak{A}}( B_0, B_x, B_*, \Sigma_0(x_{\infty},\delta), \varepsilon)
\end{gather*}
and that $p_0(x)=O\big(x^{-1}\big)$. Let $\gamma(x)$ be as in Lemma~\ref{lem3.3}, and $I^{\pm n}_{\gamma(x)} (\,\cdot \,)$, $I^{0}_{\gamma(x)} (\,\cdot \,)$ as in Lemma~\ref{lem3.5}. Then we define
\begin{gather*}
\mathcal{I}[\phi] := \sum_{n=1}^{\infty} P^+_n(x) \big({\rm e}^x x^{\sigma-1}\big)^n + \sum_{n=1}^{\infty} P^-_n(x) \big({\rm e}^{-x} x^{-\sigma-1}\big)^n +P_0(x) x^{-1}
\end{gather*}
with
\begin{gather}\label{3.6}
P^+_n(x) = I^{+n}_{\gamma(x)} (p^{+}_n(x) ), \qquad P^-_n(x) = I^{-n}_{\gamma(x)} (p^{-}_n(x) ), \qquad P_0(x) = I^{0}_{\gamma(x)} (p_0(x) ).
\end{gather}
By Lemma \ref{lem3.5}, $P^{\pm}_n(x)$ and $P_0(x)$ are represented by asymptotic series of the form
\begin{gather}\label{3.7}
P^{\pm}_n(x) \sim \sum_{m=0}^{\infty} P^{\pm}_{nm} x^{-m}, \qquad P_0(x) \sim \sum_{m=0}^{\infty} P_{0m} x^{-m}
\end{gather}
with $P^{\pm}_{nm},P_{0m} \in \Q_*$ uniformly in $(\mathbf{c},\sigma) \in B_0\times B_x \times B_*$ as $x\to\infty$ through $\Sigma_0(x_{\infty},\delta)$, and hence $\mathcal{I}[\phi] \in \hat{\mathfrak{A}}$. The series $\mathcal{I}[\phi]$ is a~formal primitive function of~$\phi$.

\begin{prop}\label{prop3.6} Suppose that $\phi \in \hat{\mathfrak{A}}$ and $p_0(x)=O\big(x^{-1}\big)$. Then $({\rm d}/{\rm d}x) \mathcal{I}[\phi] =\phi$ as a formal series.
\end{prop}

\begin{proof} For $n\ge 1$, let $p^+_{n0}=\lim\limits_{x\to\infty} p^+_n(x)$. Then
\begin{gather*}
I^{+n}_{\gamma(x)} \big(p^+_n(x)\big) =I^{+n}_{\gamma(x)} \big(p^+_n(x) -p^+_{n0}\big) + p^+_{n0} I^{+n}_{\gamma(x)} (1).
\end{gather*}
Since $\big({\rm e}^x x^{\sigma-1}\big)^{-n} \big({\rm e}^{\xi}\xi^{\sigma-1}\big)^n \big(p^{+}_{n} (\xi) -p^+_{n0}\big) \ll |x||\xi|^{-2}$ as $\xi \to\infty$ along $\gamma(x)$, $I^{+n}_{\gamma (x)}\big(p^+_n(x) -p^+_{n0}\big) $ converges absolutely, and hence
\begin{gather*}
({\rm d}/{\rm d}x)\big( \big({\rm e}^x x^{\sigma-1}\big)^n I^{+n}_{\gamma(x)}\big(p^+_n(x) -p^+_{n0}\big) \big)= \big({\rm e}^x x^{\sigma-1}\big)^n \big(p^+_n(x) -p^+_{n0}\big).
\end{gather*}
As shown in the proof of Lemma \ref{lem3.5}, $I^{+n}_{\gamma(x)} (1) = 1/n -(\sigma-1) I^{+n}_{\gamma(x)}\big(x^{-1}\big)$, the last integral being absolutely convergent. This implies
\begin{gather*}
({\rm d}/{\rm d}x)\big( \big({\rm e}^x x^{\sigma-1}\big)^n I^{+n}_{\gamma(x)}(1) p^+_{n0} \big)= \big({\rm e}^x x^{\sigma-1}\big)^n p^+_{n0}.
\end{gather*}
Thus we obtain the conclusion.
\end{proof}

Furthermore we have
\begin{prop}\label{prop3.7} If $\phi \in \mathfrak{A} =\mathfrak{A}(B_0,B_x, B_*, \Sigma_0(x_{\infty},\delta), \varepsilon)$, then $\mathcal{I}\big[x^{-1}\phi\big] \in \mathfrak{A}$ and $\big\| \mathcal{I} \big[x^{-1} \phi\big] \big\| \le (1+2/\sin\delta) \|\phi\|$.
\end{prop}

\begin{proof} By definition
\begin{gather*}
\mathcal{I}\big[x^{-1} \phi\big] =\sum_{n=1}^{\infty} \hat{P}^+_n(x) \big({\rm e}^x x^{\sigma-1}\big)^n +\sum_{n=1}^{\infty} \hat{P}^-_n(x) \big({\rm e}^{-x} x^{-\sigma-1}\big)^n + \hat{P}_0(x)x^{-1} \in \hat{\mathfrak{A}},
\end{gather*}
where
\begin{gather*}
\hat{P}^{\pm}_n(x) = I^{\pm n}_{\gamma(x)}\big(x^{-1} p_n^{\pm}(x)\big)\\
\hphantom{\hat{P}^{\pm}_n(x)}{} = -\big({\rm e}^{\pm x} x^{\pm \sigma-1} \big)^{-n} \int_{\gamma(x)}
\big({\rm e}^{\pm \xi} \xi^{\pm \sigma-1} \big)^{n} \xi^{-1} p_{n}^{\pm}(\xi) {\rm d}\xi \sim \hat{P}^{\pm}_{n1} x^{-1} +\cdots,\\
\hat{P}_0(x) = I^0_{\gamma(x)} \big(x^{-1} p_0(x)\big) = - x \int_{\gamma(x)}\xi^{-2} p_{0}(\xi) {\rm d}\xi \sim \hat{P}_{00} +\cdots.
\end{gather*}
By Lemma~\ref{lem3.3}, for any $x$, $\tilde{x} \in \Sigma_0(x_{\infty}, \delta)$ such that $|\tilde{x}|\ge |x|$,
\begin{gather*}
\big|\hat{P}^{\pm}_n(\tilde{x})\big| \le \int_{\gamma(\tilde{x})} \big|{\rm e}^{\pm \tilde{x}} \tilde{x}^{\pm \sigma-1} \big|^{-n}
\big|{\rm e}^{\pm \xi} \xi^{\pm \sigma-1} \big|^{n} |\xi|^{-1}M\big( p_{n}^{\pm}, |\tilde{x}|\big) |{\rm d}\xi |\\
\hphantom{\big|\hat{P}^{\pm}_n(\tilde{x})\big|}{} \le (1+2/\sin\delta) M\big(p^{\pm}_n, |\tilde{x}|\big) \le (1+2/\sin\delta) M\big(p^{\pm}_n, |x|\big),
\end{gather*}
which implies $M\big(\hat{P}^{\pm}_n, |x|\big) \le (1+2/\sin\delta) M(p^{\pm}_n, |x|)$. Similarly, we have $M\big(\hat{P}_0, |x|\big) \le (1+2/\sin\delta) M(p_0, |x|)$. From these inequalities the proposition immediately follows.
\end{proof}

\begin{prop}\label{prop3.8} If $\phi \in \mathfrak{A} =\mathfrak{A}(B_0, B_x, B_*, \Sigma_0(x_{\infty}, \delta),\varepsilon)$, then we have ${\rm e}^{-x} x^{-\sigma} \mathcal{I}\big[{\rm e}^x x^{\sigma-1}\phi\big]$,
${\rm e}^{x} x^{\sigma} \mathcal{I}\big[{\rm e}^{-x} x^{-\sigma-1}\phi\big] \in \mathfrak{A}$, and $\big\|{\rm e}^{-x} x^{-\sigma} \mathcal{I}\big[{\rm e}^x x^{\sigma-1}\phi\big] \big\|$, $\big\|{\rm e}^{x} x^{\sigma} \mathcal{I}\big[{\rm e}^{-x} x^{-\sigma-1}\phi\big] \big\| \le (1+2/\sin\delta) \|\phi \|$.
\end{prop}

\begin{proof}For $ \phi=\sum\limits_{n=1}^{\infty} {p}^+_n(x) \big({\rm e}^x x^{\sigma-1}\big)^n + \sum\limits_{n=1}^{\infty} {p}^-_n(x) \big({\rm e}^{-x} x^{-\sigma-1}\big)^n + {p}_0(x)x^{-1}\in {\mathfrak{A}}$, we have
\begin{gather*}
{\rm e}^x x^{\sigma-1} \phi=\sum_{n=1}^{\infty} \tilde{p}^+_n(x) \big({\rm e}^x x^{\sigma-1}\big)^n + \sum_{n=1}^{\infty} \tilde{p}^-_n(x) \big({\rm e}^{-x} x^{-\sigma-1}\big)^n + \tilde{p}_0(x)x^{-1}
\end{gather*}
with
\begin{gather*}
\tilde{p}^+_n(x) =p^+_{n-1}(x) \qquad \text{for $n\ge 2$}, \qquad \tilde{p}^+_1(x) = x^{-1} p_{0}(x),\\
 \tilde{p}_0(x) = x^{-1} p^-_{1}(x), \qquad \tilde{p}^-_n(x) =x^{-2} p^-_{n+1}(x) \qquad \text{for $n\ge 1$}.
\end{gather*}
By the definition of the primitive function
\begin{gather*}
\mathcal{I}\big[{\rm e}^x x^{\sigma-1} \phi\big] =\sum_{n=1}^{\infty} \tilde{P}^+_n(x)\big({\rm e}^x x^{\sigma-1}\big)^n +
\sum_{n=1}^{\infty} \tilde{P}^-_n(x) \big({\rm e}^{-x} x^{-\sigma-1}\big)^n + \tilde{P}_0(x)x^{-1} \in \hat{\mathfrak{A}}
\end{gather*}
with $\tilde{P}^{\pm}_n(x) = I^{\pm n}_{\gamma(x)} (\tilde{p}^{\pm}_n(x))$, $\tilde{P}_0(x) = I^{0}_{\gamma(x)} (\tilde{p}_0(x))$, and then
\begin{gather*}
{\rm e}^{-x}x^{-\sigma} \mathcal{I}\big[{\rm e}^x x^{\sigma-1} \phi\big] =\sum_{n=1}^{\infty} {P}^{*+}_n(x) \big({\rm e}^x x^{\sigma-1}\big)^n +\sum_{n=1}^{\infty} {P}^{*-}_n(x) \big({\rm e}^{-x} x^{-\sigma-1}\big)^n + {P}^*_0(x)x^{-1},
\end{gather*}
where
\begin{gather*}
{P}^{*+}_n(x) =x^{-1}\tilde{P}^+_{n+1}(x) \qquad \text{for $n\ge 1$}, \qquad {P}^*_0(x) = x^{-1} \tilde{P}^+_{1}(x),\\
{P}^{*-}_1(x) = \tilde{P}_{0}(x), \qquad {P}^{*-}_n(x) = x \tilde{P}^-_{n-1}(x) =O\big(x^{-1}\big) \qquad \text{for $n\ge 2$.}
\end{gather*}
This implies ${\rm e}^{-x} x^{-\sigma} \mathcal{I}\big[{\rm e}^x x^{\sigma-1}\phi \big] \in \hat{\mathfrak{A}}$. Since, for $n\ge 1$,
\begin{gather*}
{P}^{*+}_n(x) = x^{-1}\tilde{P}^+_{n+1}(x)= -x^{-1} \int_{\gamma(x)} \big({\rm e}^{x} x^{\sigma-1}\big)^{-n-1}\big({\rm e}^{\xi} \xi^{\sigma-1}\big)^{n+1} \tilde{p}_{n+1}^{+}(\xi) {\rm d}\xi\\
\hphantom{{P}^{*+}_n(x)}{} = - \int_{\gamma(x)} \big({\rm e}^{x} x^{\sigma}\big)^{-n-1} x^n \big({\rm e}^{\xi} \xi^{\sigma}\big)^{n+1} \xi^{-n} {p}_{n}^{+}(\xi)\frac{ {\rm d}\xi}{\xi},
\end{gather*}
we have
$|{P}^{*+}_n(\tilde{x})| \le (1+2/\sin\delta) M(p^{+}_n, |\tilde{x}|) \le (1+2/\sin\delta) M(p^{+}_n, |x|)$ for any $x, \tilde{x} \in \Sigma_0(x_{\infty},\delta)$ such that $|\tilde{x}|\ge |x|$, which implies
$ M({P}^{*+}_n, |x|) \le (1+2/\sin\delta) M(p^{+}_n, |x|)$. Similarly, we can verify $M({P}^{*}_0, |x|) \le (1+2/\sin\delta) M(p_0, |x|)$, $M({P}^{*-}_n, |x|) \le (1+2/\sin\delta) M(p^{-}_n, |x|)$. From these estimates the proposition immediately follows.
\end{proof}

\subsection[Families $\mathfrak{A}_+$, $\mathfrak{A}_-$]{Families $\boldsymbol{\mathfrak{A}_+}$, $\boldsymbol{\mathfrak{A}_-}$}\label{ssc3.2}

In addition to $B_0$, let $\tilde{B} \subset \C$ be a given domain, and let $\sigma_0= -2\theta_x-\theta_{\infty}$. For given numbers~$\Theta_1$ and~$\Theta_2$ such that $\pi/2 <\Theta_1 < \Theta_2 <3\pi/2$, denote by $\Sigma_{\pi}(\Theta_1, \Theta_2; x_{\infty})$ the sector defined by $\Theta_1 < \arg x <\Theta_2$, $|x|>x_{\infty}$.

Let $\hat{\mathfrak{A}}_+ =\hat{\mathfrak{A}}_+\big(B_0,\tilde{B}, \Sigma_{\pi}(\Theta_1, \Theta_2;x_{\infty})\big)$ be the family of formal series of the form
\begin{gather*}
\phi=\phi(\mathbf{c}, x) = \sum_{n=1}^{\infty} p_n^+(x) \big({\rm e}^x x^{\sigma_0}\big)^n + p_0(x) x^{-1},
\end{gather*}
strictly the family of pairs $\big(\phi, \{p^+_n(x), p_0(x)\}_{n\in\N} \big)$, where $p_n^+(x)$ and $p_0(x)$ are holomorphic in $(\mathbf{c}, x) \in B_0 \times\tilde{B} \times \Sigma_{\pi}(\Theta_1, \Theta_2;x_{\infty})$ and admit asymptotic representations
\begin{gather*}
p^+_n(x) \sim \sum_{m=0}^{\infty} p^+_{nm} x^{-m}, \qquad p_0(x) \sim \sum_{m=0}^{\infty} p_{0m} x^{-m}
\end{gather*}
with $p_{0m},p^+_{nm} \in \Q\big[\theta_0,\theta_x,\theta_{\infty}, c_0, c_0^{-1}, c_x\big] \subset \Q_*$ uniformly in $\mathbf{c} \in B_0 \times \tilde{B}$ as $x\to \infty$ through $\Sigma_{\pi}(\Theta_1,\Theta_2;x_{\infty})$. Furthermore, for $\phi$ above set
\begin{gather*}
\| \phi \|_+ =\|\phi \|_+(x,\eta)={\|\phi \|_+}_{ \mathbf{c}}(x,\eta) =\sum_{n=1}^{\infty} M_+(p^+_n, |x|) |\eta |^n +M_+(p_0, |x|) |x|^{-1}
\end{gather*}
with
\begin{gather*}
M_+(p, |x|) ={M_+}_{\mathbf{c}}(p,|x|) =\sup \big\{ |p(\mathbf{c}, \xi)|; \,|\xi| \ge |x|, \, \xi \in \Sigma_{\pi}(\Theta_1, \Theta_2;x_{\infty}) \big\}.
\end{gather*}
Let $ \mathfrak{A}_+ =\mathfrak{A}_+\big(B_0,\tilde{B}, \Sigma_{\pi}(\Theta_1,\Theta_2; x_{\infty}),\varepsilon\big)$ $(\subset \hat{\mathfrak{A}}_+)$ be the family of series $ \phi \in \hat{\mathfrak{A}}_+ $ such that ${\|\phi \|_+}_{\mathbf{c}}(x,\eta)$ converges uniformly in $(\mathbf{c}, x, \eta) \in B_0\times \tilde{B} \times \Xi_+(\Sigma_{\pi}(\Theta_1, \Theta_2;x_{\infty}),\varepsilon) $, where
\begin{gather*}
\Xi_+(\Sigma_{\pi}(\Theta_1,\Theta_2;x_{\infty}), \varepsilon) = \bigcup_{x\in \Sigma_{\pi}(\Theta_1,\Theta_2;x_{\infty})} \{x\} \times \{\eta; \, |\eta |<\varepsilon \} .
\end{gather*}
Then every $\phi \in \mathfrak{A}_+(B_0,\tilde{B}, \Sigma_{\pi}(\Theta_1, \Theta_2;x_{\infty}),\varepsilon)$ is holomorphic in $(\mathbf{c}, x)$ in the domain $ B_0 \times \tilde{B} \times \big(\Sigma_{\pi}(\Theta_1,\Theta_2;x_{\infty}) \cap \big\{x; \, |{\rm e}^x x^{\sigma_0}| <\varepsilon \big\}\big)$, and satisfies $|\phi(\mathbf{c}, x) | \le {\|\phi \|_+}_{\mathbf{c}} \big(x, {\rm e}^x x^{\sigma_0}\big)$;
and we may similarly verify properties corresponding to those in Propositions~\ref{prop3.1} and~\ref{prop3.2}. The primitive function of $\phi \in \hat{\mathfrak{A}}_+$ or $\mathfrak{A}_+$ is also defined by replacing $\big(\gamma(x), {\rm e}^x x^{\sigma-1}\big)$ by $\big(\gamma_{\pi}(x), {\rm e}^x x^{\sigma_0}\big)$ (cf.~Lemma~\ref{lem3.4}). Then we obtain the same conclusions as in Propositions~\ref{prop3.6} and~\ref{prop3.7}. Note that the constant $1+2/\sin\delta$ in Proposition~\ref{prop3.7} may be replaced by $2/\min\{|\cos\Theta_1|,|\cos \Theta_2|\}$. Instead of Proposition~\ref{prop3.8} we have

\begin{prop}\label{prop3.9}If $\phi \in \mathfrak{A}_+ =\mathfrak{A}_+\big(B_0,\tilde{B}, \Sigma_{\pi} (\Theta_1,\Theta_2;x_{\infty}),\varepsilon\big)$, then $\big({\rm e}^{x} x^{\sigma_0}\big)^{-n} \mathcal{I} \big[\big({\rm e}^x x^{\sigma_0}\big)^n \phi\big] \in \mathfrak{A}_+$ and $\big\| \big({\rm e}^{x} x^{\sigma_0}\big)^{-n} \mathcal{I} \big[\big({\rm e}^x x^{\sigma_0}\big)^n \phi\big] \big\| \le 2 \|\phi\|/\min\{|\cos\Theta_1|,|\cos\Theta_2|\}$ for every $n\ge 1$.
\end{prop}

\begin{rem}\label{rem3.6} If $\phi\in \mathfrak{A} (B_0, B_x, B_*, \Sigma_0(x_{\infty},\delta),\varepsilon) \cap \hat{\mathfrak{A}}_+ \big(B_0,\tilde{B}, \Sigma_{\pi}(\pi/2+\delta, \pi-\delta;x_{\infty}),\varepsilon\big)$, then $\phi \in \mathfrak{A}_+ \big(B_0,\tilde{B}, \Sigma_{\pi}(\pi/2+\delta, \pi-\delta; x_{\infty}),\varepsilon\big)$, and $\|\phi\|_+ \le \|\phi\|$.
\end{rem}

In the sector $\Sigma_0(\Theta_1',\Theta_2',x_{\infty})$: $-\pi/2 <\Theta'_1 < \arg x <\Theta'_2 <\pi/2$, we may similarly define the family $\hat{\mathfrak{A}}_- =\hat{\mathfrak{A}}_- (\tilde{B}, B_x, \Sigma_0(\Theta'_1,\Theta'_2;x_{\infty}) )$ and ${\mathfrak{A}}_- ={\mathfrak{A}}_- (\tilde{B}, B_x, \Sigma_0(\Theta'_1,\Theta'_2;x_{\infty}), \varepsilon )$ with $\sigma_0' =2\theta_0 +\theta_{\infty}$, which have similar properties.

\section{Equation (\ref{1.3}) and a system of integral equations}\label{sc4}

We would like to construct a general solution of \eqref{1.3} under the restrictions~(a) and~(b). A~generic form of a pair of matrices $\Lambda_0$, $\Lambda_x$ satisfying $(\Lambda_0 +\Lambda_x)_{11} = -(\Lambda_0 +\Lambda_x)_{22} = -\theta_{\infty}/2$ and having the eigenvalues $\pm \theta_0/2$, $\pm \theta_x/2$, respectively, may be given by
\begin{gather}
 \Lambda_0 = \frac 14 (\sigma -\theta_{\infty}) J + \gamma^0_+ \Delta_++ \gamma^0_- \Delta_-,\qquad
 \Lambda_x = - \frac 14 (\sigma + \theta_{\infty}) J + \gamma^{x}_+ \Delta_++ \gamma^{x}_- \Delta_-\label{4.1}
\end{gather}
with
\begin{gather*}
\gamma^0_{\pm} = \pm \frac {c_0^{\pm 1}}4 (\sigma \pm 2\theta_0 -\theta_{\infty}),\qquad \gamma^x_{\pm} = \mp \frac {c_x^{\pm 1}}4(\sigma \mp 2\theta_x +\theta_{\infty})
\end{gather*}
(cf.~Theorem~\ref{thm2.1}), where $\sigma$, $c_0$, $c_x$ are arbitrary. For a solution $(A_0(x), A_x(x))$ of~\eqref{1.3}, let us set
\begin{gather}
A_0(x) = x^{-(1/4)(\sigma +\theta_{\infty})J} (\Lambda_0 +\Phi_0(x)) x^{(1/4)(\sigma +\theta_{\infty}) J},\nonumber\\
A_x(x) = {\rm e}^{(x/2)J}x^{(1/4)(\sigma-\theta_{\infty}) J} (\Lambda_x +\Phi_x(x)) x^{-(1/4)(\sigma-\theta_{\infty}) J} {\rm e}^{-(x/2)J}\label{4.2}
\end{gather}
(cf.~\eqref{4.6} and \eqref{4.7}). {\it If $\Phi_0(x), \Phi_x(x)\to 0$ along some curve tending to $\infty$, then $(A_0(x), A_x(x))$ satisfies $(\mathrm{a})$ and $(\mathrm{b})$, and \eqref{1.1} has the isomonodromy property}. In checking (a) and (b) we use $({\rm d}/{\rm d}x)\operatorname{tr} A_0(x)= ({\rm d}/{\rm d}x)\operatorname{tr} A_x(x) \equiv 0$, $({\rm d}/{\rm d}x) \det A_0(x) =({\rm d}/{\rm d}x) \det A_x(x)\equiv 0$. Indeed, if $\det A_0(x)\not\equiv 0$, then $({\rm d}/{\rm d}x)A_0(x) = x^{-1} \big(A_x(x) -A_0(x) A_x(x) A_0(x)^{-1}\big) A_0(x)$, and hence $({\rm d}/{\rm d}x)\det A_0(x) = x^{-1} \operatorname{tr}\big(A_x(x) -A_0(x) A_x(x) A_0(x)^{-1}\big) \det A_0(x)\equiv 0$.

In what follows we change \eqref{1.3} into a suitable nonlinear system, and construct a solution of it as mentioned above. Let
\begin{gather*}
A_0= f_0 J + f_+\Delta_+ + f_-\Delta_-, \qquad A_x= g_0 J + g_+\Delta_+ + g_-\Delta_-
\end{gather*}
with $g_0= -f_0 -\theta_{\infty}/2$. Then, \eqref{1.3} is equivalent to a system of nonlinear equations
\begin{alignat}{3}
& x f'_0 = f_-g_+ - f_+g_-, \qquad && g_0=-f_0 -\theta_{\infty}/2,&\nonumber\\
& x f_+' = 2(f_+g_0 - f_0g_+), \qquad && x g'_+ = 2(f_0g_+ - f_+g_0) + x g_+,& \nonumber\\
& x f_-' = 2(f_0g_- - f_-g_0), \qquad && x g'_- = 2(f_-g_0 - f_0g_-) - x g_-.& \label{4.3}
\end{alignat}
We set
\begin{gather*}
f_0= (\sigma-\theta_{\infty})/4 + \varphi, \qquad g_0= -(\sigma +\theta_{\infty})/4 -\varphi
\end{gather*}
to write \eqref{4.3} in the form
\begin{gather}
x\varphi' = f_- g_+ - f_+ g_-,\nonumber\\
xf'_+ = -(1/2)((\sigma+\theta_{\infty}) f_+ + (\sigma-\theta_{\infty}) g_+) - 2( \varphi f_+ +\varphi g_+),\nonumber\\
xf'_- = (1/2) ((\sigma+\theta_{\infty}) f_- + (\sigma -\theta_{\infty})g_-) + 2( \varphi f_- +\varphi g_-),\nonumber\\
xg'_+ = x g_+ + (1/2) ((\sigma-\theta_{\infty}) g_+ +(\sigma+\theta_{\infty})f_+) + 2(\varphi g_++\varphi f_+),\nonumber\\
xg'_- =-x g_- - (1/2)((\sigma-\theta_{\infty}) g_- +(\sigma +\theta_{\infty})f_-) - 2(\varphi g_- +\varphi f_-).\label{4.4}
\end{gather}
The following fact may be verified by an argument analogous to that of \cite[Section~10]{Si} (see also \cite[Chapter~4]{W}).

\begin{lem}\label{lem4.1} By the change of variables $\mathbf{y} = (I+ p(x) \Delta_+) \mathbf{z}$ the linear system
\begin{gather*}
x\frac {{\rm d}\mathbf{y}}{{\rm d}x} = \left( \begin{pmatrix} 0 & 0 \\ 0 & 1 \end{pmatrix}
x + \frac 12 \begin{pmatrix} -(\sigma +\theta_{\infty} ) & - (\sigma-
\theta_{\infty}) \\
\sigma +\theta_{\infty} & \sigma-\theta_{\infty} \\
\end{pmatrix} \right) \mathbf{y}
\end{gather*}
is taken to
\begin{gather*}
x\frac {{\rm d}\mathbf{z}}{{\rm d}x} = \left( \begin{pmatrix} 0 & 0 \\ 0 & 1 \end{pmatrix}
x + \frac 12 \begin{pmatrix} -(\sigma +\theta_{\infty} )(1+p(x)) & 0 \\
\sigma +\theta_{\infty} & \sigma-\theta_{\infty} + (\sigma
+ \theta_{\infty})p(x) \end{pmatrix}\right) \mathbf{z},
\end{gather*}
and by $\mathbf{z} =(I+ q(x) \Delta_-) \mathbf{w}$ the last system is reduced to
\begin{gather*}
x\frac {{\rm d}\mathbf{w}}{{\rm d}x} = \left( \begin{pmatrix} 0 & 0 \\ 0 & 1 \end{pmatrix}
x + \frac 12 \begin{pmatrix} -(\sigma +\theta_{\infty} )(1+p(x)) & 0 \\
0 & \sigma-\theta_{\infty} + (\sigma + \theta_{\infty})p(x) \end{pmatrix}\right) \mathbf{w}.
\end{gather*}
Here $p(x)$ and $q(x)$ satisfy
\begin{gather*}
xp'(x) + xp(x) + (1/2)(1+p(x))( \sigma-\theta_{\infty} +(\sigma+\theta_{\infty}) p(x) ) =0,\\
xq'(x) - xq(x) - (1/2)( (\sigma +\theta_{\infty})(1+ (1+2p(x))q(x)) +(\sigma -\theta_{\infty}) q(x) )=0,
\end{gather*}
and admit the asymptotic representations
\begin{gather*}
p(x) = -(1/2)(\sigma-\theta_{\infty})\big( x^{-1} + (1-\sigma)x^{-2} +\big[x^{-3}\big]\big)\qquad \text{for $|\arg x -\pi| <3\pi/2-\delta$},\\
q(x) = -(1/2)(\sigma +\theta_{\infty})\big(x^{-1} -(1+\sigma) x^{-2}+\big[x^{-3}\big]\big) \qquad \text{for $|\arg x - \pi /2| < \pi - \delta$},
\end{gather*}
whose coefficients are in $\Q[\theta_{\infty},\sigma]$.
\end{lem}

\begin{rem}\label{rem4.1}\quad\samepage \begin{enumerate}\itemsep=0pt
\item[(1)] In Lemma \ref{lem4.1}, $p(x)$ may be replaced by $\tilde{p}(x)$ having the same asymptotic representation in the sector $|\arg x +\pi|< 3\pi/2 -\delta$. The diagonalisation is possible by $(\tilde{p}(x),q(x))$ for $|\arg x+ \pi/2| <\pi-\delta$ as well as $(p(x), q(x))$ for $|\arg x- \pi/2|<\pi-\delta$.
\item[(2)] By the substitution $(\sigma-\theta_{\infty}, \sigma +\theta_{\infty} , x) \mapsto \big({-}(\sigma-\theta_{\infty}), -(\sigma +\theta_{\infty}) , {\rm e}^{\pi {\rm i}}x\big) $ or $ \mapsto
\big({-}(\sigma-\theta_{\infty}), -(\sigma +\theta_{\infty}) , {\rm e}^{-\pi {\rm i}}x\big) $, we obtain
\begin{gather*}
p^*(x) = -(1/2)(\sigma-\theta_{\infty})\big( x^{-1} - (1+\sigma)x^{-2} +\big[x^{-3}\big]\big),\\
q^*(x) = -(1/2)(\sigma +\theta_{\infty})\big(x^{-1} +(1-\sigma) x^{-2}+\big[x^{-3}\big]\big)
\end{gather*}
such that
\begin{gather*}
x\frac {{\rm d}\mathbf{y}}{{\rm d}x} = \left( \begin{pmatrix} 0 & 0 \\ 0 & -1 \end{pmatrix}
x + \frac 12 \begin{pmatrix} \sigma +\theta_{\infty} & \sigma-\theta_{\infty} \\
-(\sigma +\theta_{\infty} ) & -(\sigma-\theta_{\infty})
\end{pmatrix} \right) \mathbf{y}
\end{gather*}
is changed into
\begin{gather*}
x\frac {{\rm d}\mathbf{w}}{{\rm d}x} = \left( \begin{pmatrix} 0 & 0 \\ 0 & -1 \end{pmatrix}
x + \frac 12 \begin{pmatrix} (\sigma +\theta_{\infty})(1+p^*(x))
 & 0 \\
0 & -(\sigma-\theta_{\infty})-(\sigma+\theta_{\infty})p^*(x)
\end{pmatrix} \right) \mathbf{w}.
\end{gather*}
\end{enumerate}
\end{rem}

\begin{rem}\label{rem4.2}The proof of Lemma~\ref{lem4.1} depends on the following fact: from every point in the sector $|\arg x-\pi| < 3\pi/2 -\delta$ (respectively, $|\arg x|< 3\pi/2-\delta$) one may draw a line such that
it is contained in the sector and that $\re x \to -\infty$ (respectively, $\to \infty$) along it.
\end{rem}

By the facts above the linear parts of \eqref{4.4} may be diagonalised, that is, there exists a~transformation of the form
\begin{gather*}
f_+ =(1+pq)u_++pv_+= \big(1+\big[x^{-2}\big]\big) u_+ - x^{-1} \big((\sigma-\theta_{\infty})/2 +\big[x^{-1}\big]\big) v_+,\\
g_+ = qu_+ +v_+ = x^{-1}\big({-}(\sigma +\theta_{\infty})/2 +\big[x^{-1}\big]\big)u_+ + v_+,\\
f_- = (1+p^*q^*)u_- + p^* v_-= \big(1+\big[x^{-2}\big]\big) u_- - x^{-1} \big((\sigma-\theta_{\infty})/2 +\big[x^{-1}\big]\big) v_-,\\
g_- = q^* u_- +v_-= x^{-1}\big({-}(\sigma +\theta_{\infty})/2 +\big[x^{-1}\big]\big)u_- + v_-
\end{gather*}
that takes \eqref{4.4} to
\begin{gather}
 x\varphi' =\big(1+\big[x^{-2}\big]\big) u_-v_+ - \big(1+\big[x^{-2}\big]\big) u_+v_-\nonumber\\
\hphantom{x\varphi' =}{} + \big(\sigma+\theta_{\infty} +\big[x^{-1}\big]\big)x^{-2}u_+u_- + \big(\sigma-\theta_{\infty} +\big[x^{-1}\big]\big) x^{-2}v_+v_-,\nonumber\\
 xu'_+= \big({-}(\sigma +\theta_{\infty})/2 +\big[x^{-1}\big]\big)u_+ -2\varphi\big( \big(1+\big[x^{-1}\big]\big)u_+ +\big(1+\big[x^{-1}\big]\big)v_+\big),\nonumber\\
 xv'_+=xv_+ + \big((\sigma-\theta_{\infty})/2 +\big[x^{-1}\big]\big)v_+ +2\varphi\big( \big(1+\big[x^{-1}\big]\big)u_+ +\big(1+\big[x^{-1}\big]\big)v_+\big),\nonumber\\
 xu'_-= \big((\sigma +\theta_{\infty})/2 +\big[x^{-1}\big]\big)u_- +2\varphi\big( \big(1+\big[x^{-1}\big]\big)u_- +\big(1+\big[x^{-1}\big]\big)v_-\big),\nonumber\\
 xv'_-=-xv_- - \big((\sigma-\theta_{\infty})/2 +\big[x^{-1}\big]\big)v_- -2\varphi\big( \big(1+\big[x^{-1}\big]\big)u_- +\big(1+\big[x^{-1}\big]\big)v_-\big),\label{4.4.1}
\end{gather}
where $\big[x^{-1}\big], \big[x^{-2}\big], \ldots $ are valid in the sector $|\arg x- \pi/2|<\pi -\delta$. Rewriting, e.g., the last equation in the form
\begin{gather*}
\big({\rm e}^x x^{(\sigma-\theta_{\infty})/2} \big(1+\big[x^{-1}\big]\big)v_-\big)'= -2{\rm e}^xx^{(\sigma-\theta_{\infty})/2 -1} \varphi\big( \big(1+\big[x^{-1}\big]\big)u_- +\big(1+\big[x^{-1}\big]\big)v_-\big),
\end{gather*}
and setting
\begin{gather*}
 x^{(\sigma +\theta_{\infty})/2} \big(1+\big[x^{-1}\big]\big)u_+ = \gamma^0_+ +\varphi_+,\qquad {\rm e}^{-x} x^{-(\sigma-\theta_{\infty})/2}\big(1+\big[x^{-1}\big]\big)v_+ = \gamma^x_+ +\psi_+,\\
 x^{-(\sigma+\theta_{\infty})/2} \big(1+\big[x^{-1}\big]\big)u_- = \gamma^0_- +\varphi_-,\qquad {\rm e}^{x} x^{(\sigma-\theta_{\infty})/2 } \big(1+\big[x^{-1}\big]\big)v_- = \gamma^x_- +\psi_-
\end{gather*}
with $\gamma^0_{\pm}$ and $\gamma^x_{\pm}$ as in \eqref{4.1}, we arrive at a system of integral equations of the form
\begin{gather}
\varphi = \mathcal{I} \bigl[{\rm e}^x x^{\sigma-1}(1)_x \big(\gamma_-^0 +\varphi_-\big)\big(\gamma^x_+ +\psi_+\big) - {\rm e}^{-x} x^{-\sigma-1} (1)_x \big(\gamma^0_+ +\varphi_+\big)\big(\gamma^x_- +\psi_-\big)\bigr]\nonumber\\
\hphantom{\varphi =}{} + \mathcal{I} \bigl[ \big(\sigma +\theta_{\infty}+\big[x^{-1}\big]\big) x^{-3}\big(\gamma_+^0 +\varphi_+\big)\big(\gamma^0_- +\varphi_-\big)\nonumber\\
\hphantom{\varphi =}{} + \big(\sigma -\theta_{\infty}+\big[x^{-1}\big]\big) x^{-3} \big(\gamma^x_+ +\psi_+\big)\big(\gamma^x_-+ \psi_-\big)\bigr],\nonumber\\
 \varphi_+ = -2 \mathcal{I} \bigl[\varphi \bigl( x^{-1}(1)_x \big(\gamma_+^0 +\varphi_+\big)+ {\rm e}^{x} x^{\sigma-1} (1)_x \big(\gamma^x_+ +\psi_+\big) \bigr) \bigr],\nonumber\\
 \psi_+ = 2 \mathcal{I} \bigl[\varphi \bigl( x^{-1}(1)_x \big(\gamma_+^x +\psi_+\big)+ {\rm e}^{-x} x^{-\sigma-1} (1)_x \big(\gamma^0_+ +\varphi_+\big) \bigr) \bigr],\nonumber\\
\varphi_- = 2 \mathcal{I} \bigl[\varphi \bigl( x^{-1}(1)_x \big(\gamma_-^0 +\varphi_-\big)+ {\rm e}^{-x} x^{-\sigma-1} (1)_x \big(\gamma^x_- +\psi_-\big)\bigr) \bigr],\nonumber\\
\psi_- = -2 \mathcal{I} \bigl[\varphi \bigl( x^{-1}(1)_x \big(\gamma_-^x +\psi_-\big)+ {\rm e}^{x} x^{\sigma-1} (1)_x \big(\gamma^0_- +\varphi_-\big) \bigr) \bigr].\label{4.5}
\end{gather}
Here $(1)_x$ denotes $1+\big[x^{-1}\big]$, and the functions $\varphi$, $\varphi_{\pm}$, $\psi_{\pm}$ and the products $\varphi_-\psi_+$, $\varphi_+ \psi_-$, $\ldots$ are supposed to be at least in $\hat{\mathfrak{A}}$. If we succeed in constructing $\varphi, \varphi_{\pm}, \psi_{\pm} \in \mathfrak{A}$, then, by Propositions~\ref{prop3.6} through~\ref{prop3.8}, $(A_0(x), A_x(x))$ with
\begin{gather}
f_0= (\sigma-\theta_{\infty})/4 + \varphi, \qquad g_0= -(\sigma+\theta_{\infty})/4 -\varphi,\nonumber\\
f_+ = x^{-(\sigma+\theta_{\infty})/2} \bigl((1)_x \big(\gamma^0_+ +\varphi_+\big)- \big((\sigma-\theta_{\infty})/2+\big[x^{-1}\big]\big) {\rm e}^x x^{\sigma -1}\big(\gamma^x_+ +\psi_+\big) \bigr),\nonumber\\
g_+ = {\rm e}^x x^{(\sigma-\theta_{\infty})/2}\bigl((1)_x \big(\gamma^x_+ +\psi_+\big)- \big((\sigma +\theta_{\infty})/2 +\big[x^{-1}\big]\big) {\rm e}^{-x} x^{-\sigma -1}\big(\gamma^0_+ +\varphi_+\big)\bigr),\nonumber\\
f_- = x^{(\sigma+\theta_{\infty})/2} \bigl((1)_x \big(\gamma^0_- +\varphi_-\big)- \big((\sigma-\theta_{\infty})/2+\big[x^{-1}\big]\big) {\rm e}^{-x} x^{-\sigma -1}\big(\gamma^x_- +\psi_-\big) \bigr),\nonumber\\
g_- = {\rm e}^{-x} x^{-(\sigma-\theta_{\infty})/2} \bigl((1)_x \big(\gamma^x_- +\psi_-\big)- \big((\sigma+\theta_{\infty})/2 +\big[x^{-1}\big]\big) {\rm e}^{x} x^{\sigma -1}\big(\gamma^0_- +\varphi_-\big) \bigr)\label{4.6}
\end{gather}
is a solution of \eqref{1.3} written as \eqref{4.2}. Then, $\Phi_0(x)$ and $\Phi_x(x)$ in \eqref{4.2} are given by
\begin{gather}
\Phi_0(x)_{11}= -\Phi_0(x)_{22} =\varphi,\nonumber\\
\Phi_0(x)_{12}= \gamma^0_+\big[x^{-1}\big] +(1)_x \varphi_+ -\big((\sigma-\theta_{\infty})/2 +\big[x^{-1}\big]\big) {\rm e}^x x^{\sigma-1}\big(\gamma^x_+ +\psi_+\big),\nonumber\\
\Phi_0(x)_{21}= \gamma^0_-\big[x^{-1}\big] +(1)_x \varphi_- -\big((\sigma-\theta_{\infty})/2 +\big[x^{-1}\big]\big) {\rm e}^{-x}x^{-\sigma-1}\big(\gamma^x_- +\psi_-\big),\nonumber\\
\Phi_x(x)_{11}= -\Phi_x(x)_{22} =-\varphi,\nonumber\\
\Phi_x(x)_{12}= \gamma^x_+\big[x^{-1}\big] +(1)_x \psi_+ -\big((\sigma +\theta_{\infty})/2 +\big[x^{-1}\big]\big) {\rm e}^{-x} x^{-\sigma-1}\big(\gamma^0_+ +\varphi_+\big),\nonumber\\
\Phi_x(x)_{21}= \gamma^x_-\big[x^{-1}\big] +(1)_x \psi_- -\big((\sigma +\theta_{\infty})/2 +\big[x^{-1}\big]\big) {\rm e}^{ x} x^{\sigma-1}\big(\gamma^0_- +\varphi_-\big).\label{4.7}
\end{gather}
Moreover, if $\varphi$, $\varphi_{\pm}$, $\psi_{\pm}$ are such that $\Phi_0(x), \Phi_x (x) \to 0$ as $x\to \infty$ along some curve, then $(A_0(x), A_x(x)) $ satisfies (a) and (b), and hence it is a desired solution of~\eqref{1.3}.
\begin{rem}\label{rem4.3}By $p(x)$ and ${p}^*(x)$ in Lemma \ref{lem4.1} and Remark \ref{rem4.1}, the linear parts of \eqref{4.4.1} are written in the more detailed form
\begin{gather*}
 xu'_+= (-(\sigma +\theta_{\infty})/2 +\kappa(x))u_+ + \cdots, \\
 xv'_+=xv_+ + ((\sigma-\theta_{\infty})/2 -\kappa(x) )v_+ + \cdots,\\
 xu'_-= ((\sigma +\theta_{\infty})/2 -\kappa(x) )u_- + \cdots,\\
 xv'_-=-xv_- + (-(\sigma-\theta_{\infty})/2 + \kappa(x) )v_- +\cdots
\end{gather*}
with $\kappa(x)= \big(\sigma^2 -\theta_{\infty}^2\big) x^{-1}/4 + \big[x^{-2}\big]$, in each appearance $\big[x^{-2}\big]$ not necessarily denoting the same function. Then the expressions of $f_{\pm}$, $g_{\pm}$ in \eqref{4.6} become
\begin{gather*}
f_+ = x^{-(\sigma+\theta_{\infty})/2} \bigl((1-\kappa(x)) \big(\gamma^0_+ +\varphi_+\big)- \cdots \bigr),\\
g_+ = {\rm e}^x x^{(\sigma-\theta_{\infty})/2}\bigl((1+\kappa(x)) \big(\gamma^x_+ +\psi_+\big)- \cdots \bigr),\\
f_- = x^{(\sigma+\theta_{\infty})/2} \bigl((1+\kappa(x)) \big(\gamma^0_- +\varphi_-\big)- \cdots \bigr),\\
g_- = {\rm e}^{-x} x^{-(\sigma-\theta_{\infty})/2} \bigl((1-\kappa(x)) \big(\gamma^x_- +\psi_-\big)- \cdots \bigr),
\end{gather*}
and the first equation in \eqref{4.5} is
\begin{gather*}
\varphi = \mathcal{I} \bigl[{\rm e}^x x^{\sigma-1}(1+2\kappa(x)) \big(\gamma_-^0 +\varphi_-\big)\big(\gamma^x_+ +\psi_+\big)\\
\hphantom{\varphi =}{} - {\rm e}^{-x} x^{-\sigma-1} (1-2\kappa(x)) \big(\gamma^0_+ +\varphi_+\big)\big(\gamma^x_- +\psi_-\big)\bigr] +\cdots.
\end{gather*}
\end{rem}

\section{Proofs of Theorems \ref{thm2.1} and \ref{thm2.2}}\label{sc5}

\subsection{System of integral equations}\label{ssc5.1}
Instead of system \eqref{4.5} we deal with
\begin{gather}
\varphi = F_0(x,\varphi, \varphi_+, \psi_+, \varphi_-, \psi_-) := {\rm e}^{x}x^{\sigma-1}(1)_x \big(\gamma^0_- +\varphi_-\big)\big(\gamma^x_+ + \psi_+\big)\nonumber\\
\hphantom{\varphi =}{} + {\rm e}^{-x}x^{-\sigma-1}(1)_x \big(\gamma^0_+ +\varphi_+\big)\big(\gamma^x_-+ \psi_-\big)\nonumber\\
\hphantom{\varphi =}{} -4 \mathcal{I}\bigl[{\rm e}^{x}x^{\sigma-2}(1)_x \varphi\big(\gamma^0_- +\varphi_-\big)\big(\gamma^x_+ + \psi_+\big)- {\rm e}^{-x}x^{-\sigma-2}(1)_x \varphi\big(\gamma^0_+ +\varphi_+\big)\big(\gamma^x_- +\psi_-\big) \bigr]\nonumber\\
\hphantom{\varphi =}{} + \mathcal{I} \bigl[ x^{-3}\big(\sigma+ \theta_{\infty}+\big[x^{-1}\big]+[1]\varphi\big)\big(\gamma^0_+ +\varphi_+\big)\big(\gamma^0_- +\varphi_-\big)\nonumber\\
\hphantom{\varphi =}{} + x^{-3}\big(\sigma- \theta_{\infty}+\big[x^{-1}\big]+[1]\varphi\big)\big(\gamma^x_+ +\psi_+\big)\big(\gamma^x_- +\psi_-\big) \bigr],\nonumber\\
 \varphi_+ = F_+(x, \varphi, \varphi_+, \psi_+):= -2 \mathcal{I} \bigl[\varphi \bigl( x^{-1}(1)_x \big(\gamma_+^0 +\varphi_+\big)+ {\rm e}^{x} x^{\sigma-1} (1)_x \big(\gamma^x_+ +\psi_+\big)\bigr) \bigr],\nonumber\\
 \psi_+ = G_+(x, \varphi, \varphi_+, \psi_+):= 2 \mathcal{I} \bigl[\varphi \bigl( x^{-1}(1)_x \big(\gamma_+^x +\psi_+\big)+ {\rm e}^{-x} x^{-\sigma-1} (1)_x \big(\gamma^0_+ +\varphi_+\big)\bigr) \bigr],\nonumber\\
 \varphi_- = F_-(x, \varphi, \varphi_-, \psi_-):= 2 \mathcal{I} \bigl[\varphi \bigl( x^{-1}(1)_x \big(\gamma_-^0 +\varphi_-\big)+ {\rm e}^{-x} x^{-\sigma-1} (1)_x \big(\gamma^x_- +\psi_-\big)\bigr) \bigr],\nonumber\\
 \psi_- = G_-(x, \varphi, \varphi_-, \psi_-):= -2 \mathcal{I} \bigl[\varphi\bigl( x^{-1}(1)_x \big(\gamma_-^x +\psi_-\big)+ {\rm e}^{x} x^{\sigma-1} (1)_x \big(\gamma^0_- +\varphi_-\big) \bigr) \bigr],\label{5.1}
\end{gather}
which is equivalent to \eqref{4.5}. Indeed, by Proposition \ref{prop3.6} and integration by parts, we may write the first equation of \eqref{4.5} in the form
\begin{gather*}
\varphi= I_0 + \mathcal{I} \bigl[\big(\sigma+\theta_{\infty} +\big[x^{-1}\big]\big) x^{-3}\big(\gamma^0_+ +\varphi_+\big)\big(\gamma^0_- +\varphi_-\big) + \cdots \bigr]
\end{gather*}
with
\begin{gather*}
I_0 = \mathcal{I} \bigl[{\rm e}^x x^{\sigma-1}(1)_x \big(\gamma_-^0 +\varphi_-\big) \big(\gamma^x_+ +\psi_+\big) - {\rm e}^{-x} x^{-\sigma-1} (1)_x \big(\gamma^0_+ +\varphi_+\big)\big(\gamma^x_- +\psi_-\big)\bigr]\\
\hphantom{I_0}{} = {\rm e}^x x^{\sigma-1}(1)_x \big(\gamma_-^0 +\varphi_-\big)\big(\gamma^x_+ +\psi_+\big) + {\rm e}^{-x} x^{-\sigma-1} (1)_x \big(\gamma^0_+ +\varphi_+\big)\big(\gamma^x_- +\psi_-\big)\\
\hphantom{I_0=}{} - \mathcal{I} \bigl[{\rm e}^x x^{\sigma-1}(1)_x \big(\psi'_+\big(\gamma_-^0 +\varphi_-\big)+\varphi'_-\big(\gamma^x_+ +\psi_+\big) \big)\\
\hphantom{I_0=}{} + {\rm e}^{-x} x^{-\sigma-1} (1)_x \big(\psi'_-\big(\gamma^0_+ +\varphi_+\big) +\varphi'_+\big(\gamma^x_- +\psi_-\big) \big) \bigr].
\end{gather*}
The substitution of
\begin{gather*}
\varphi'_+ = -2 \varphi \bigl( x^{-1}(1)_x \big(\gamma_+^0 +\varphi_+\big)+ {\rm e}^{x} x^{\sigma-1} (1)_x \big(\gamma^x_+ +\psi_+\big) \bigr),
\end{gather*}
$\varphi'_- = \cdots$, $\psi'_{\pm} =\cdots$ into the last integral yields the first equation of \eqref{5.1}.

\subsection{Sequences}\label{ssc5.2}
To construct a solution of \eqref{5.1} we define $\big\{ \big(\varphi^j, \varphi_+^j, \psi_+^j, \varphi_-^j, \psi_-^j\big) \big\}_{j \ge 0}$ by
\begin{gather}
\varphi^0=\varphi^0_{\pm} =\psi^0_{\pm} \equiv 0,\qquad
\varphi^{j+1} = F_0\big(x, \varphi^j, \varphi_+^j, \psi_+^j, \varphi_-^j, \psi_-^j\big),\nonumber\\
\varphi_+^{j+1} = F_+\big(x, \varphi^{j+1}, \varphi_+^j, \psi_+^j \big),\qquad \psi_+^{j+1} = G_+\big(x, \varphi^{j+1}, \varphi_+^j, \psi_+^j \big),\nonumber\\
\varphi_-^{j+1} = F_-\big(x, \varphi^{j+1}, \varphi_-^j, \psi_-^j \big),\qquad \psi_-^{j+1} = G_-\big(x, \varphi^{j+1}, \varphi_-^j, \psi_-^j \big)\label{5.2}
\end{gather}
for $j \ge 0$. It is shown by induction on $j$ that $\varphi^j$, $\varphi_{\pm}^j$, $\psi_{\pm}^j$ are finite sums of $\big({\rm e}^x x^{\sigma -1}\big)^n[1]$, $\big({\rm e}^{-x} x^{-\sigma-1}\big) ^n [1]$ and $\big[x^{-1}\big]$, and hence $\varphi^j, \varphi_{\pm}^j, \psi_{\pm}^j \in \mathfrak{A}$ for $j\ge 0$.

\begin{rem}\label{rem5.1}As long as $|\arg x -\pi/2|<\pi-\delta$, the path of integration for $\mathcal {I}[\, \cdot \,]$ may also be taken to be a line on which $\big|{\rm e}^{\xi} \xi^{\sigma-1}\big|$ or $\big|{\rm e}^{-\xi}\xi^{-\sigma-1}\big|$ decays exponentially, and hence the asymptotic expansions $\big[x^{-1}\big]$, $\ldots$ in the expressions of $\varphi^j$, $\varphi^j_ {\pm}$, $\psi^j_{\pm}$ are valid in the sector $|\arg x-\pi/2|<\pi -\delta$ (cf.\ Example~\ref{exa3.1}).
\end{rem}

In this section, to simplify the description, for a sequence $\{\phi^j \}$ we write $\tri \phi^j:=\phi^j-\phi^{j-1}$.

\subsubsection[$\varphi^j$, $\varphi_{\pm}^j$, $\psi_{\pm}^j$ for $j=1,2$]{$\boldsymbol{\varphi^j}$, $\boldsymbol{\varphi_{\pm}^j}$, $\boldsymbol{\psi_{\pm}^j}$ for $\boldsymbol{j=1,2}$}\label{sssc5.2.1}

By definition $\varphi^1= F_0(x, 0,0, 0,0, 0)$, $\varphi_{\pm}^1 =F_{\pm}\big(x, \varphi^1,0, 0\big)$, $\psi_{\pm}^1 =G_{\pm}\big(x, \varphi^1, 0, 0\big)$, that is
\begin{gather}
\varphi^1=\gamma^0_-\gamma^x_+(1)_x {\rm e}^x x^{\sigma-1}+\gamma^0_+\gamma^x_-(1)_x {\rm e}^{-x} x^{-\sigma-1} + \big(\gamma^0_+\gamma^0_-[1]+ \gamma^x_+\gamma^x_-[1]\big)x^{-2},\nonumber\\
\varphi_+^1 =-2\gamma^0_+ X^1_0 -\gamma^x_+ X^1_+,\qquad \psi_+^1 =2\gamma^x_+ X^1_0 +\gamma^0_+ X^1_-,\nonumber\\
\varphi_-^1 =2\gamma^0_- X^1_0 +\gamma^x_- X^1_-,\qquad \psi_-^1 =-2\gamma^x_- X^1_0 -\gamma^0_- X^1_+\label{5.3}
\end{gather}
with
\begin{gather*}
X^1_0=\gamma^0_-\gamma^x_+(1)_x {\rm e}^x x^{\sigma-2}-\gamma^0_+\gamma^x_-(1)_x {\rm e}^{-x} x^{-\sigma-2} + \big( \gamma^0_+\gamma^0_-[1]+ \gamma^x_+\gamma^x_-[1]\big)x^{-2},\\
X^1_+=\gamma^0_-\gamma^x_+(1)_x {\rm e}^{2x} x^{2\sigma-2}-2\gamma^0_+\gamma^x_-(1)_x x^{-1} + \big(\gamma^0_+\gamma^0_-[1]+ \gamma^x_+\gamma^x_-[1]\big){\rm e}^x x^{\sigma -3},\\
X^1_-=-2 \gamma^0_-\gamma^x_+(1)_x x^{-1}-\gamma^0_+\gamma^x_-(1)_x {\rm e}^{-2x}x^{-2\sigma-2} + \big(\gamma^0_+\gamma^0_-[1]+ \gamma^x_+\gamma^x_-[1]\big){\rm e}^{-x} x^{-\sigma -3},
\end{gather*}
which belong to $\mathfrak{A}=\mathfrak{A}(B_0, B_x, B_*, \Sigma_0(x_{\infty},\delta), \varepsilon)$ (cf.\ Section~\ref{ssc3.1}). We may choose $x^1_{\infty}=x^1_{\infty}(B_0,B_x, B_*,\delta)> x_{\infty}$ depending on $\big(\gamma^0_{\pm},\gamma^x_{\pm}\big)$ or on $(\mathbf{c},\sigma)\in B_0\times B_x\times B_*$ in such a way that the following estimates with absolute implied constants are valid for $(\mathbf{c},\sigma, x)\in B_0\times B_x \times B_* \times \Sigma_0\big(x^1_{\infty}, \delta\big)$:
\begin{gather*}
\big\| \varphi^1 \big\| \ll \big|\gamma^0_-\gamma^x_+ \big| \big|\eta x^{-1}\big| + \big|\gamma^0_+\gamma^x_- \big| \big|\eta^{-1} x^{-1}\big| + |x|^{-1},\\
\big\| X^1_0 \big\| \ll |x|^{-1}\big( \big|\gamma^0_-\gamma^x_+ \big| \big|\eta x^{-1}\big| + \big|\gamma^0_+\gamma^x_- \big| \big|\eta^{-1} x^{-1}\big| + 1\big),\\
\big\| X^1_+ \big\| \ll \big|\gamma^0_-\gamma^x_+ \big| \big|\eta x^{-1}\big|^2 + \big|\gamma^0_+\gamma^x_- \big| |x|^{-1} +\big|\eta x^{-1}\big| |x|^{-1},\\
\big\| X^1_- \big\| \ll \big|\gamma^0_-\gamma^x_+ \big| |x|^{-1} + \big|\gamma^0_+\gamma^x_- \big| \big|\eta^{-1} x^{-1}\big|^2 +\big|\eta^{-1} x^{-1}\big| |x|^{-1}.
\end{gather*}
Then under the condition that
\begin{gather}\label{5.4}
(\gamma_0^* +\gamma_1^* +1)(\gamma^*_1 +1) \varepsilon \le r_0 <1
\end{gather}
with
\begin{gather*}
\gamma^*_0:= \big|\gamma_-^0\gamma_+^x\big| +\big|\gamma_+^0\gamma_-^x\big|, \qquad\gamma^*_1:= \big|\gamma_+^0\big| +\big|\gamma_-^0\big| +\big|\gamma_+^x\big| +\big|\gamma_-^x\big|
\end{gather*}
for every $(\mathbf{c}, \sigma) \in B_0\times B_x \times B_*$ (cf.\ Remark~\ref{rem2.1}), we have, for $\big|\eta x^{-1}\big|, \big|\eta^{-1} x^{-1}\big| <\varepsilon$, $x\in\Sigma_0\big(x^1_{\infty},\delta\big)$,
\begin{gather}\label{5.5}
\big \|\varphi^1 \big\| \ll \big(\gamma^*_0 +1\big)\varepsilon, \qquad \big\|\varphi^1_{\pm}\big\|, \big\|\psi^1_{\pm}\big\| \ll \big(\gamma^*_0+1\big)\big(\gamma^*_1+1\big) \varepsilon,
\end{gather}
where $r_0$ will be chosen later. By \eqref{5.3} we have ${\rm e}^{-x}x^{-\sigma} \varphi^1_+ \!\in\! \mathfrak{A}$, and, under~\eqref{5.4}, \smash{$\big\| {\rm e}^{-x}x^{-\sigma} \varphi^1_+ \big\| \!\ll\! 1$} for $\big|\eta x^{-1}\big|, \big|\eta^{-1} x^{-1}\big|<\varepsilon$, $x\in \Sigma_0\big(x^1_{\infty},\delta\big)$. Similarly, ${\rm e}^x x^{\sigma}\varphi_-^1, \big({\rm e}^x x^{\sigma}\big)^{\pm 1}\psi^1_{\pm} \in \mathfrak{A}$, and $\big\| {\rm e}^x x^{\sigma}\varphi_-^1\big\|$, $\big\|\big({\rm e}^x x^{\sigma}\big)^{\pm 1}\psi^1_{\pm}\big\| \ll 1$. By~\eqref{5.2},
\begin{gather*}
\tri \varphi^2 =F_0\big(x,\varphi^1, \varphi^1_+,\psi^1_+,\varphi^1_-,\psi^1_-\big)-F_0(x,0,0,0,0,0)=X_0 + I_1 +I_2
\end{gather*}
with
\begin{gather*}
X_0 ={\rm e}^x x^{\sigma-1} (1)_x \big(\gamma^x_+\varphi^1_- +\gamma^0_-\psi_+^1 +\varphi^1_-\psi_+^1\big)+{\rm e}^{-x} x^{-\sigma-1} (1)_x \big(\gamma^x_-\varphi^1_+ +\gamma^0_+\psi_-^1 +\varphi^1_+\psi_-^1\big),\\
I_1 = -4 \mathcal{I}\bigl[ {\rm e}^x x^{\sigma-2} (1)_x\varphi^1 \big(\gamma_-^0+\varphi_-^1\big)\big(\gamma_+^x +\psi_+^1\big) - {\rm e}^{-x} x^{-\sigma-2} (1)_x\varphi^1 \big(\gamma_+^0+\varphi_+^1\big)\big(\gamma_-^x +\psi_-^1\big)\bigr],\\
I_2 = \mathcal{I}\bigl[ \big[x^{-3}\big]\big(\gamma_-^0 \varphi_+^1 +\gamma_+^0\varphi_-^1 +\varphi_+^1\varphi_-^1\big) +\big[x^{-3}\big] \varphi^1\big(\gamma_+^0 +\varphi_+^1\big)\big(\gamma_-^0 +\varphi_-^1\big)\\
\hphantom{I_2 =}{} + \big[x^{-3}\big]\big(\gamma_-^x \psi_+^1 +\gamma_+^x\psi_-^1 +\psi_+^1\psi_-^1\big) +\big[x^{-3}\big] \varphi^1\big(\gamma_+^x +\psi_+^1\big)\big(\gamma_-^x +\psi_-^1\big) \bigr].
\end{gather*}
It is easy to see $ \tri \varphi^2 \in \mathfrak{A}$. By Proposition~\ref{prop3.7}
\begin{gather*}
\| X_0\| \ll \big(\big\| {\rm e}^x x^{\sigma-1}\big\| +\big\|{\rm e}^{-x}x^{-\sigma-1}\big\|\big) \Upsilon_0,\\
\| I_1\| \ll \big(\big\| {\rm e}^x x^{\sigma-1}\big\| +\big\|{\rm e}^{-x}x^{-\sigma-1}\big\|\big) \big\|\varphi^1\big\|\big(\gamma^*_0 + \Upsilon_0\big),\qquad \| I_2 \| \ll |x|^{-2} \big(\big\| \varphi^1 \big\| +1\big) \big(\gamma^*_0 + \Upsilon_0)
\end{gather*}
with $\Upsilon_0 =\big(\gamma^*_1+ \big\| \varphi_+^1\big\| +\big\|\psi_+^1 \big\| \big)\big( \big\| \varphi_+^1\big\| +\big\|\psi_+^1 \big\|+ \big\| \varphi_-^1\big\| +\big\|\psi_-^1 \big\| \big)$, the implied constants depending on~$\delta$ only. Substituting~\eqref{5.5}, we have, under \eqref{5.4}, $\big\|\tri \varphi^2 \big\| \ll \big(\gamma^*_1 +1\big)\varepsilon$ for $\big|\eta x^{-1} \big|, \big|\eta^{-1} x^{-1} \big|<\varepsilon$. It is easy to verify $\big({\rm e}^x x^{\sigma}\big)^{\pm 1} \tri \varphi^2 \in\mathfrak{A}$, and we have
\begin{gather*}
\big\| {\rm e}^x x^{\sigma} \tri \varphi^2 \big\| \ll \big\| {\rm e}^x x^{\sigma-1} \big\|\big(\gamma_1^* +\big\| \psi^1_+ \big\|\big) \big(\big\| {\rm e}^x x^{\sigma} \varphi_-^1 \big\| +\big\| {\rm e}^x x^{\sigma} \psi_+^1 \big\|\big)\\
\hphantom{\big\| {\rm e}^x x^{\sigma} \tri \varphi^2 \big\| \ll}{} + |x|^{-1}\big(\gamma_1^* +\big\| \varphi^1_+ \big\|\big) \big(\big\| \varphi_+^1 \big\| +\big\| \psi_-^1 \big\|\big) + \| I_3 \| + \| I_4 \| + \big\| {\rm e}^x x^{\sigma} I_2 \big\|
\end{gather*}
with
\begin{gather*}\begin{split}&
I_3 = {\rm e}^x x^{\sigma} \mathcal{I}\bigl[ {\rm e}^x x^{\sigma-2} (1)_x\varphi^1 \big(\gamma_-^0+\varphi_-^1\big)\big(\gamma_+^x +\psi_+^1\big) \bigr],\\
& I_4 = {\rm e}^x x^{\sigma} \mathcal{I} \bigl[{\rm e}^{-x} x^{-\sigma-2} (1)_x\varphi^1 \big(\gamma_+^0+\varphi_+^1\big)\big(\gamma_-^x +\psi_-^1\big)\bigr]\end{split}
\end{gather*}
and $I_2$ given above. By Proposition~\ref{prop3.8}
\begin{gather*}
 \bigl\| I_3- \gamma_-^0 \gamma_+^x {\rm e}^x x^{\sigma} \mathcal{I} \big[{\rm e}^x x^{\sigma-2}(1)_x \varphi^1\big] \bigr\|\\
\qquad{} = \bigl \| {\rm e}^xx^{\sigma} \mathcal{I} \bigl[ {\rm e}^{-x}x^{-\sigma-1} (1)_x {\rm e}^xx^{\sigma-1} \varphi^1 \big(\gamma_-^0 {\rm e}^x x^{\sigma} \psi_+^1 +\gamma_+^x {\rm e}^xx^{\sigma}\varphi_-^1 +\psi_+^1 {\rm e}^x x^{\sigma} \varphi_-^1\big) \bigr] \bigr\|\\
\qquad{} \ll \big\| {\rm e}^x x^{\sigma-1} \big\| \big\|\varphi^1 \big\| \big(\gamma_1^* + \big\|\psi_+^1 \big\|\big)\big(\big\| {\rm e}^x x^{\sigma} \psi_+^1 \big\| + \big\| {\rm e}^x x^{\sigma}\varphi_-^1 \big\|\big),\\
\| I_4\| \ll |x|^{-1} \big\|\varphi^1\big\| \big(\gamma^*_0 + \big(\gamma^*_1 +\big\|\varphi_+^1 \big\|\big)\big(\big\|\varphi_+^1 \big\| +\big\|\psi_-^1 \big\|\big) \big),\\
\big\|{\rm e}^x x^{\sigma} I_2 \big\| \ll |x|^{-1} \big\| {\rm e}^x x^{\sigma-1} \big\| \big(\|\varphi^1\|+1\big)\big(\gamma^*_0+ \Upsilon_0\big).
\end{gather*}
By \eqref{5.3}
\begin{gather*}
\big\| {\rm e}^x x^{\sigma} \mathcal{I} \big[ {\rm e}^x x^{\sigma-2}(1)_x \varphi^1 \big] \big\|\ll \gamma^*_0 \big(\big\| {\rm e}^x x^{\sigma-1} \big\|^3 + |x|^{-1}\big\| {\rm e}^x x^{\sigma-1}\big\|\big) +\big(\gamma^*_1\big)^2|x|^{-2} \big\| {\rm e}^x x^{\sigma -1}\big\|^2 .
\end{gather*}
Summing up these estimates we get $\big\| {\rm e}^x x^{\sigma} \tri \varphi^2 \big\| \ll \big(\gamma^*_0 +\gamma^*_1 +1\big)\varepsilon$ under~\eqref{5.4}. Similarly for $\big\| {\rm e}^{-x} x^{-\sigma} \tri \varphi^2 \big\|$ we have the same inequality. We may verify that ${\rm e}^{-x}x^{-\sigma} \tri \varphi_+^2 \in \mathfrak{A}$ as well, and by analogous arguments we have
\begin{gather*}
\big\| \tri \varphi_+^2 \big\| \ll \big|\gamma_+^0\big| \big\|\tri \varphi^2\big\| + \big\|\varphi^2 \big\| \big\|\varphi_+^1 \big\| + |\gamma_+^x| \big\| {\rm e}^x x^{\sigma} \tri \varphi^2 \big\| + \big\|\varphi^2 \big\| \big\|{\rm e}^x x^{\sigma}\psi_+^1 \big\|, \\
\big\| {\rm e}^{-x}x^{-\sigma} \tri \varphi_+^2 \big\| \ll \big|\gamma_+^0\big| \big\| {\rm e}^{-x}x^{-\sigma} \tri \varphi^2 \big\| + \big\|\varphi^2 \big\| \big\| {\rm e}^{-x} x^{-\sigma} \varphi_+^1 \big\| + \big|\gamma_+^x\big| \big\|\tri \varphi^2\big\| + \big\|\varphi^2 \big\| \big\|\psi_+^1 \big\|,
\end{gather*}
where $\big\|\varphi^2 \big\| \le \big\|\varphi^1\big\| + \big\|\tri\varphi^2\big\|$. Substitution of \eqref{5.5} and the estimates for $\big\|\tri \varphi^2 \big\|$, $\ldots$ obtained above yields $\big\|\tri \varphi_+^2 \big\|,\big\|{\rm e}^{-x} x^{-\sigma} \tri \varphi_+^2 \big\| \ll \big(\gamma_0^*+\gamma_1^*+1\big)\big(\gamma_1^* +1\big)\varepsilon $. Furthermore ${\rm e}^xx^{\sigma}\tri\varphi_-^2$,
$\big({\rm e}^xx^{\sigma}\big)^{\pm 1}\tri \psi_{\pm}^2 \in \mathfrak{A}$, and for $\big\|\tri \varphi_-^2 \big\|$, $\big\|{\rm e}^{x}x^{\sigma}\tri \varphi_-^2\big\|$, $\big\|\tri\psi_{\pm}^2\big\|$, $\big\|\big({\rm e}^xx^{\sigma}\big)^{\pm}\tri\psi_{\pm}^2\big\|$, we have the same estimates. As will be shown later $\big({\rm e}^x x^{\sigma}\big)^{\pm 1}\tri \varphi^{j+1},\big({\rm e}^x x^{\sigma}\big)^{\mp 1} \tri \varphi_{\pm}^{j+1},\big({\rm e}^x x^{\sigma}\big)^{\pm 1} \tri \psi_{\pm}^{j+1} \in \mathfrak{A}$ for $j \ge 2$ as well. Let us set
\begin{gather*}
\Psi_{j} : = \big\|\tri \varphi^{j+1}\big\| +\big\|{\rm e}^x x^{\sigma} \tri \varphi^{j+1}\big\| + \big\|{\rm e}^{-x} x^{-\sigma} \tri \varphi^{j+1} \big\|\\
\hphantom{\Psi_{j} : =}{} +\big\|\tri \varphi_+^{j+1} \big\| + \big\|{\rm e}^{-x} x^{-\sigma} \tri \varphi_+^{j+1} \big\| + \big\|\tri \psi_+^{j+1} \big\| + \big\|{\rm e}^{x} x^{\sigma} \tri \psi_+^{j+1} \big\|
\\
\hphantom{\Psi_{j} : =}{} + \big\| \tri \varphi_-^{j+1} \big\| + \big\|{\rm e}^{x} x^{\sigma} \tri \varphi_-^{j+1} \big\| + \big\| \tri \psi_-^{j+1} \big\| + \big\|{\rm e}^{-x} x^{-\sigma} \tri \psi_-^{j+1} \big\|.
\end{gather*}
For $j=1$, as shown above, we have

\begin{lem}\label{lem5.1} If $x^1_{\infty}=x^1_{\infty}(B_0,B_x, B_*,\delta)$ is sufficiently large, then $\big({\rm e}^x x^{\sigma}\big)^{\pm 1} \tri \varphi^2$, $\big({\rm e}^x x^{\sigma}\big)^{\mp 1} \tri \varphi_{\pm}^2$, $\big({\rm e}^x x^{\sigma}\big)^{\pm 1} \tri \psi_{\pm}^2$ also belong to~$\mathfrak{A}$, and, under the condition~\eqref{5.4} for every $(\mathbf{c},\sigma) \in B_0\times B_x \times B_*$, we have
\begin{gather*}
\big\|\varphi^1\big\| \le K_0\big(\gamma^*_0 +1\big)\varepsilon, \qquad \big\|\varphi_{\pm}^1\big\|,\big\|\psi_{\pm}^1\big\| \le K_0\big(\gamma^*_0 +1\big)\big(\gamma^*_1 +1\big)\varepsilon,\\
 \Psi_1 \le K_0\big(\gamma^*_0 +\gamma^*_1 +1\big)\big(\gamma^*_1+1\big)\varepsilon
\end{gather*}
for $\big|\eta x^{-1}\big|,\big|\eta^{-1} x^{-1}\big|<\varepsilon$, $x\in \Sigma_0\big(x_{\infty}^1,\delta\big)$, where $K_0 \ge 1$ is some positive number depending on $\delta$ only.
\end{lem}

\subsubsection[$\Psi_j$ for $j\ge 2$]{$\boldsymbol{\Psi_j}$ for $\boldsymbol{j\ge 2}$}\label{sssc5.2.2}

In addition to \eqref{5.4} suppose that
\begin{gather}\label{5.6}
\|\varphi^{\nu}\| \le 3K_0\big(\gamma^*_0 +\gamma^*_1 +1\big)\big(\gamma^*_1 +1\big)\varepsilon \le 1, \qquad \big\|\varphi_{\pm}^{\nu} \big\|, \big\|\psi_{\pm}^{\nu} \big\| \le 1,\\
\label{5.7}
 \big({\rm e}^x x^{\sigma}\big)^{\pm 1}\tri\varphi^{\nu},\big({\rm e}^x x^{\sigma}\big)^{\mp 1}\tri\varphi_{\pm}^{\nu},
\big({\rm e}^x x^{\sigma}\big)^{\pm 1}\tri\psi_{\pm}^{\nu} \in \mathfrak{A},\\
\label{5.7a}
\Psi_{\nu-1} \le (1/2) \Psi_{\nu-2} \qquad \text{if $\nu \ge 3$}
\end{gather}
for $2\le \nu \le j$. Lemma \ref{lem5.1} implies that~\eqref{5.6}, \eqref{5.7} and~\eqref{5.7a} are valid for $j=2$ if $3K_0r_0 \le 1$, since $\big\| \varphi^2\big\| \le \big\|\varphi^1\big\|+ \Psi_1$, $\big\| \varphi_{\pm}^2\big\| \le \big\|\varphi_{\pm}^1\big\|+ \Psi_1$ and $\big\| \psi_{\pm}^2\big\| \le \big\|\psi_{\pm}^1\big\|+ \Psi_1$.

From \eqref{5.2} it follows that, for $j\ge 2$,
\begin{gather*}
\big\|\tri \varphi^{j+1} \big\| \le \big\| {\rm e}^x x^{\sigma-1} (1)_x \tri \omega_1^{j} +{\rm e}^{-x} x^{-\sigma-1} (1)_x \tri \omega_2^{j} \big\| \\
\hphantom{\big\|\tri \varphi^{j+1} \big\| \le}{} +\big\| 4\mathcal{I} \big[ {\rm e}^x x^{\sigma-2} (1)_x\tri\big( \varphi^{j}\omega_1^{j}\big)
- {\rm e}^{-x}x^{-\sigma-2}(1)_x \tri\big(\varphi^{j}\omega_2^{j}\big) \big] \big\|\\
\hphantom{\big\|\tri \varphi^{j+1} \big\| \le}{}
 + \bigl\| \mathcal{I} \bigl[ x^{-3}\bigl([1]\tri\chi_1^{j} + [1]\tri\chi_2^{j} + [1]\tri\big(\varphi^{j}\chi_1^{j}\big)
 + [1]\tri\big(\varphi^{j}\chi_2^{j}\big) \bigr) \bigr] \bigr\|
\end{gather*}
with $ \omega_1^j=\big(\gamma_-^0 +\varphi_-^j\big)\big(\gamma_+^x + \psi_+^j\big)$, $\omega_2^j=\big(\gamma_+^0 +\varphi_+^j\big)\big(\gamma_-^x + \psi_-^j\big)$, $ \chi_1^j=\big(\gamma_+^0 +\varphi_+^j\big)\big(\gamma_-^0 + \varphi_-^j\big)$, $\chi_2^j=\big(\gamma_+^x +\psi_+^j\big) \big(\gamma_-^x + \psi_-^j\big)$. Then, by~\eqref{5.4} and~\eqref{5.6}, the first two parts on the right-hand side are, respectively,
\begin{gather*}
 {} \le \big\|{\rm e}^x x^{\sigma-1}\big\| \big(\big|\gamma^0_-\big|\big\| \tri\psi_+^j \big\| + \big|\gamma^x_+\big| \big\|\tri \varphi_-^j \big\| + \big\| \tri \big(\varphi_-^j\psi_+^j\big) \big\| \big)\\
\hphantom{\le}{} + \big\|{\rm e}^{-x} x^{-\sigma-1}\big\| \big(\big|\gamma^0_+\big|\big\|\tri \psi_-^j \big\| + \big|\gamma^x_-\big|
\big\|\tri \varphi_+^j \big\| + \big\| \tri\big( \varphi_+^j\psi_-^j \big) \big\| \big)\\
{} \le \big(1+\gamma^*_1\big)\varepsilon \big(\big\| \tri \varphi_+^j \big\|+\big\| \tri \psi_+^j \big\| +\big\| \tri \varphi_-^j \big\|+
\big\|\tri \psi_-^j \big\| \big) \le \big(1+\gamma^*_1\big)\varepsilon \Psi_{j-1},
\end{gather*}
and
\begin{gather*}
{}\le L_0 \big\|{\rm e}^x x^{\sigma-1}\big\| \bigl(\big|\gamma^0_-\gamma^x_+\big|\big\|\tri \varphi^j \big\| +\big|\gamma_0^-\big|\big\|\tri\big( \varphi^j \psi_+^j\big) \big\| + \big|\gamma^x_+\big|\big\|\tri\big(\varphi^j \varphi_-^j\big) \big\|
 + \big\|\tri\big(\varphi^j \varphi_-^j\psi_+^j\big) \big\|\bigr)\\
\hphantom{\le}{} +L_0 \big\|{\rm e}^{-x} x^{-\sigma-1}\big\|\bigl(\big|\gamma^0_+\gamma^x_-\big|\big\|\tri \varphi^j\big\| +\big|\gamma_+^0\big| \big\|\tri\big( \varphi^j \psi_-^j\big)\big\| +\big| \gamma^x_-\big|\big\|\tri\big( \varphi^j \varphi_+^j\big)\big\| \\
\hphantom{\le}{} + \big\| \tri\big(\varphi^j \varphi_+^j\psi_-^j\big) \big\| \bigr) \\
{} \le 2 L_0\big(\gamma^*_0+\gamma^*_1 +1\big) \varepsilon \big(\big\|\tri \varphi^j \big\| +\big\|\tri \varphi_+^j \big\|+ \big\|\tri \psi_+^j \big\| +\big\| \tri \varphi_-^j \big\|+\big\|\tri \psi_-^j \big\| \big)\\
{}\le 2 L_0\big(\gamma^*_0+\gamma^*_1 +1\big) \varepsilon \Psi_{j-1}
\end{gather*}
for $\big|\eta x^{-1}\big|, \big|\eta^{-1}x^{-1} \big| <\varepsilon$, $x\in \Sigma_0\big(x^1_{\infty}, \delta\big)$, since
\begin{gather*}
\big\|\tri\big( \varphi^j \varphi_-^j \psi_+^j\big) \big\| \le \big\|\varphi_-^j\big\| \big\|\psi_+^j\big\| \big\|\tri \varphi^j \big\|
+ \big\|\varphi^{j-1}\big\| \big\|\psi_+^j\big\| \big\|\tri \varphi_-^j \big\| + \big\|\varphi^{j-1}\big\| \big\| \varphi_-^{j-1}\big\| \big\|\tri \psi_+^j \big\|.
\end{gather*}
Here $L_0 \ge K_0$ is some number depending on $\delta$ only, which may be retaken larger, if necessary, in each appearance below. Similarly the remaining part is $\le 2 L_0\big(\gamma^*_0 +\gamma^*_1 +1\big)\varepsilon^2 \Psi_{j-1}$, and hence
\begin{gather}\label{5.8}
\big\|\tri \varphi^{j+1} \big\| \le 5L_0\big(\gamma^*_0+\gamma^*_1 +1\big)\varepsilon \Psi_{j-1}.
\end{gather}
Observe that ${\rm e}^x x^{\sigma} \mathcal{I}\big[{\rm e}^x x^{\sigma-2} (1)_x \tri \big(\varphi^j \omega_1^j\big) \big]={\rm e}^x x^{\sigma} \mathcal{I}\big[{\rm e}^{-x} x^{-\sigma-1} (1)_x {\rm e}^x x^{\sigma-1} (\cdots)\big]$, where
\begin{gather*}
(\cdots) ={\rm e}^xx^{\sigma} \tri\big(\varphi^j \omega_1^j\big) =\big(\gamma_-^0\gamma_+^x +\gamma_-^0\psi_+^j + \gamma_+^x\varphi_-^j
+ \psi_+^j \varphi_-^j\big)\cdot {\rm e}^x x^{\sigma}\tri \varphi^j\\
\hphantom{(\cdots) =}{} +\big(\gamma_-^0 + \varphi_-^j\big) \varphi^{j-1}\cdot {\rm e}^xx^{\sigma} \tri \psi_+^j+\big(\gamma_+^x + \psi_+^{j-1}\big) \varphi^{j-1}\cdot {\rm e}^x x^{\sigma} \tri \varphi_-^j.
\end{gather*}
By Proposition \ref{prop3.8}, this and analogous facts combined with~\eqref{5.7} imply that ${\rm e}^x x^{\sigma}\tri\varphi^{j+1} \in \mathfrak{A}$. Then, dividing $\big\| {\rm e}^x x^{\sigma} \tri \varphi^{j+1} \big\| $ into three parts corresponding to those of $\big\| \tri \varphi^{j+1} \big\|$ above, we derive
\begin{gather}
\notag
\big\| {\rm e}^x x^{\sigma}\tri \varphi^{j+1} \big\| \le 4L_0 \big(\gamma^*_0 +\gamma^*_1 +1\big) \varepsilon \big(\big\|\tri \varphi^j \big\|+ \big\|{\rm e}^x x^{\sigma}\tri \varphi^j \big\|\\
\notag
\hphantom{\big\| {\rm e}^x x^{\sigma}\tri \varphi^{j+1} \big\| \le}{} + \big\| \tri \varphi_+^j \big\|+ \big\|{\rm e}^x x^{\sigma} \tri \psi_+^j \big\|+ \big\|{\rm e}^x x^{\sigma} \tri \varphi_-^j \big\|+ \big\|\tri \psi_-^j \big\|\big)\\
\label{5.9}
\hphantom{\big\| {\rm e}^x x^{\sigma}\tri \varphi^{j+1} \big\|}{} \le 4 L_0\big(\gamma^*_0 +\gamma^*_1 +1\big)\varepsilon \Psi_{j-1}.
\end{gather}
Similarly we may show that ${\rm e}^{-x}x^{-\sigma}\tri \varphi^{j+1} \in \mathfrak{A}$ and
\begin{gather}\label{5.10}
 \big\| {\rm e}^{-x} x^{-\sigma}\tri \varphi^{j+1} \big\| \le 4 L_0\big(\gamma^*_0 +\gamma^*_1 +1\big)\varepsilon \Psi_{j-1}.
\end{gather}

Since ${\rm e}^{-x}x^{-\sigma} \mathcal{I} \big[ x^{-1} (1)_x (\varphi^{j+1} \big(\gamma^0_+ +\varphi_+^j\big) -\varphi^j \big(\gamma^0_+ +\varphi_+^{j-1}\big) \big) \big] ={\rm e}^{-x}x^{-\sigma} \mathcal{I} \big[ {\rm e}^x x^{\sigma-1} (1)_x (\cdots) \big]$ with $(\cdots)= \big(\gamma^0_+ +\varphi_+^j\big) {\rm e}^{-x}x^{-\sigma}\tri \varphi^{j+1}+\varphi^j\cdot {\rm e}^{-x}x^{-\sigma}\tri \varphi_+^{j}$, we have $ {\rm e}^{-x}x^{-\sigma} \tri \varphi_+^{j+1} \in \mathfrak{A}$ as well. Furthermore by~\eqref{5.6}
\begin{gather*}
 \big\|\tri \varphi_+^{j+1} \big\| \le L_0 \bigl( \gamma^*_1 \big\| \tri \varphi^{j+1} \big\| + \big\| \tri\big( \varphi^{j+1} \varphi_+^{j} \big) \big\| +\gamma^*_1 \big\| {\rm e}^x x^{\sigma} \tri \varphi^{j+1} \big\| +\big\| {\rm e}^x x^{\sigma}\tri\big(\varphi^{j+1}\psi_+^j \big)\big\| \bigr) \\
 \hphantom{\big\|\tri \varphi_+^{j+1} \big\|}{} \le L_0 \big(\gamma^*_1+1\big) \bigl( \big\| \tri\varphi^{j+1} \big\|
+ \big\| {\rm e}^x x^{\sigma} \tri \varphi^{j+1} \big\| \\
\hphantom{\big\|\tri \varphi_+^{j+1} \big\|}{} +3K_0\big(\gamma^*_0 +\gamma^*_1 +1\big)\varepsilon \big( \big\| \tri \varphi_+^{j} \big\|
 +\big\| {\rm e}^x x^{\sigma}\tri \psi_+^j \big\| \big) \bigr)
\end{gather*}
since
\begin{gather*}
\big\|\tri\big( \varphi^{j+1} \varphi_+^{j} \big) \big\|\le \big\|\varphi_+^j \big\| \big\| \tri \varphi^{j+1} \big\|+\big\|\varphi^j \big\| \big\| \tri \varphi_+^{j} \big\|,\\
\big\| {\rm e}^xx^{\sigma} \tri\big(\varphi^{j+1}\psi_+^j\big) \big\| \le \big\|\psi_+^j \big\| \big\|{\rm e}^xx^{\sigma} \tri\varphi^{j+1}\big\| +\big\|\varphi^j\big\| \big\|{\rm e}^xx^{\sigma} \tri\psi^j_+\big\|,
\end{gather*}
and similarly
\begin{gather*}
\big\| {\rm e}^{-x} x^{-\sigma} \tri \varphi_+^{j+1} \big\| \le L_0\big(\gamma^*_1+1\big) \bigl( \big\| \tri \varphi^{j+1} \big\| + \big\| {\rm e}^{-x} x^{-\sigma} \tri \varphi^{j+1} \big\|\\
\hphantom{\big\| {\rm e}^{-x} x^{-\sigma} \tri \varphi_+^{j+1} \big\| \le}{} +3K_0\big(\gamma^*_0 +\gamma^*_1 +1\big) \varepsilon
 \big( \big\|{\rm e}^{-x} x^{-\sigma}\tri \varphi_+^{j} \big\| +\big\|\tri \psi_+^j \big\| \big) \bigr).
\end{gather*}
We combine these estimates with \eqref{5.8}, \eqref{5.9} and \eqref{5.10} to obtain
\begin{gather*}
\big\|\tri \varphi_+^{j+1} \big\| + \big\| {\rm e}^{-x} x^{-\sigma} \tri \varphi_+^{j+1} \big\| \le 3L_0^2\big(\gamma^*_0 +\gamma^*_1 +1\big)\big(\gamma^*_1 +1\big)\varepsilon \Psi_{j-1}\\
\qquad\quad{} + L_0\big(\gamma^*_1+1\big) \big( 2 \big\| \tri\varphi^{j+1} \big\| + \big\| {\rm e}^{x} x^{\sigma} \tri \varphi^{j+1} \big\|
+ \big\| {\rm e}^{-x} x^{-\sigma} \tri \varphi^{j+1} \big\| \big)\\
\qquad{} \le 21L_0^2\big(\gamma^*_0 +\gamma^*_1 +1\big)\big(\gamma^*_1 +1\big)\varepsilon \Psi_{j-1}.
\end{gather*}
The other differences $\tri \varphi_-^{j+1}$, ${\rm e}^xx^{\sigma}\tri \varphi_-^{j+1} $, $\tri \psi_{\pm}^{j+1}$, $\big({\rm e}^x x^{\sigma}\big)^{\pm 1} \tri \psi_{\pm}^{j+1} $ are treated in a similar manner. Thus we have shown that~\eqref{5.7} is valid for $ \nu \le j+1$, and that $\Psi_{j} \le 100 L_0^2 \big(\gamma^*_0+\gamma^*_1 +1\big)\big(\gamma^*_1 +1\big)\varepsilon \Psi_{j-1}$. Choose~$r_0$ in~\eqref{5.4} in such a way that $3K_0 r_0 \le 100 L^2_0 r_0 \le 1/2$. Then $\Psi_j \le (1/2) \Psi_{j-1}$, and hence~\eqref{5.7a} is valid for $\nu \le j+1$. By Lemma~\ref{lem5.1}
\begin{gather*}
\big\|\varphi^{j+1} \big\| \le \big\|\varphi^1\big\| + \sum_{\nu=1}^j\big\|\tri \varphi^{\nu+1}\big\| \le \big\|\varphi^1\big\| +\sum_{\nu=1}^j \Psi_{\nu} \le \big\|\varphi^1\big\| + 2\Psi_1\\
\hphantom{\big\|\varphi^{j+1} \big\|}{} \le K_0\big(\gamma^*_0+1\big)\varepsilon + 2K_0 \big(\gamma^*_0 +\gamma^*_1 +1\big)\big(\gamma^*_1
+1\big)\varepsilon \le 3K_0\big(\gamma^*_0 +\gamma^*_1+1\big)\big(\gamma^*_1 +1\big)\varepsilon,\\
\big\|\varphi_+^{j+1} \big\| \le \big\|\varphi^1_+ \big\| + 2\Psi_1 \le 3K_0\big(\gamma^*_0 +\gamma^*_1+1\big)\big(\gamma^*_1 +1\big)\varepsilon, \quad \ldots,
\end{gather*}
that is, \eqref{5.6} is also valid for $\nu\le j+1$. Thus we have shown that~\eqref{5.6}, \eqref{5.7} and \eqref{5.7a} are valid for every $\nu$ if $r_0$ is as above.
\begin{prop}\label{prop5.2}For $j \ge 2$ we have $\big({\rm e}^x x^{\sigma}\big)^{\pm 1} \tri\varphi^{j},\big({\rm e}^x x^{\sigma}\big)^{\mp 1} \tri \varphi_{\pm}^{j},\big({\rm e}^x x^{\sigma}\big)^{\pm 1} \tri \psi_{\pm}^{j} \in\mathfrak{A}$, and $\Psi_{j} \le (1/2)\Psi_{j-1} $ for $\big|\eta x^{-1}\big|, \big|\eta^{-1} x^{-1}\big|< \varepsilon$, $x\in \Sigma_0\big(x^1_{\infty}, \delta\big)$ under~\eqref{5.4} with $r_0=r_0(\delta)$ such that $100 L_0^2r_0 \le 1/2$.
\end{prop}

\subsection{Asymptotic coefficients}\label{ssc5.3}

Let $\phi \in \hat{\mathfrak{A}}$ be given by
\begin{gather*}
\phi=\sum_{n=1}^{\infty} p_n^+(x)\big({\rm e}^x x^{\sigma-1}\big)^n+\sum_{n=1}^{\infty} p_n^-(x)\big({\rm e}^{-x} x^{-\sigma-1}\big)^n +p_0(x) x^{-1}.
\end{gather*}
For every $p^+_n(x) \not\sim 0$, let $d_+(n) \in \N\cup\{0\}$ be such that $p^+_n(x)=x^{-d_+(n)}\big(a^+_n+ O\big(x^{-1}\big) \big)$ with $a^+_n \not= 0$, and assign the lattice point $(n, -d_+(n)) \in \Z^2$ to $p_n^+(x)$. For the other asymptotic coefficients $p_0(x), p^-_n(x)\not\sim 0$, the degrees $d(0)$ and $d_-(n)$ are similarly defined, and the lattice points $(0, -d(0))$ and $(-n, -d_-(n))$ are assigned to $p_0(x)\not\sim 0$ and to $p_n^-(x)\not\sim 0$, respectively. Then denote by $\varpi(\phi)$ the set of such lattice points for all asymptotic coefficients $\not\sim 0$ of $\phi$. For $d, m_-,m_+ \in \Z$ satisfying $m_- \le m_+$, $d \ge 0$, set
\begin{gather*}
 [m_-, m_+; -d] := \big\{ (x_1, x_2) \in \Z^2; \, x_2\le -d,\, x_2\le x_1 - m_- - d, \, x_2 \le -x_1 +m_+ -d \big\}.
\end{gather*}
Then, for $\varphi^j$, $\varphi_{\pm}^j$, $\psi_{\pm}^j$ given by~\eqref{5.2}, we have

\begin{prop}\label{prop5.3}For every $j \ge 2$, the lattice sets $\varpi\big(\varphi^{j-1}\big)$, $\varpi\big(\varphi_{\pm}^{j-1}\big)$, $\varpi\big(\psi_{\pm}^{j-1}\big)$ consist of finite numbers of lattice points, and have the properties:
\begin{gather*}
\varpi\big( \varphi^{j} \big) \subset [-1, 1; 0],\qquad \varpi\big( \tri \varphi^{j} \big) \subset [-j, j; -j+1];\\
\varpi\big( \varphi_+^{j} \big),\varpi\big( \psi_-^{j}\big ) \subset [0, 2; 0], \qquad \varpi\big( \tri \varphi_+^{j} \big),\varpi\big( \tri\psi_-^{j} \big) \subset [-j, -1; -j] \cup [0, j+1; -j+1];\\
\varpi\big( \varphi_-^{j} \big),\varpi\big(\psi_+^{j} \big) \subset [-2, 0; 0], \qquad \varpi\big( \tri \varphi_-^{j} \big),\varpi\big( \tri \psi_+^{j} \big) \subset [-j-1, 0; -j+1]\cup [1,j;-j].
\end{gather*}
\end{prop}

The polygons packing the lattice sets on the right-hand sides are described in Fig.~\ref{convex}.

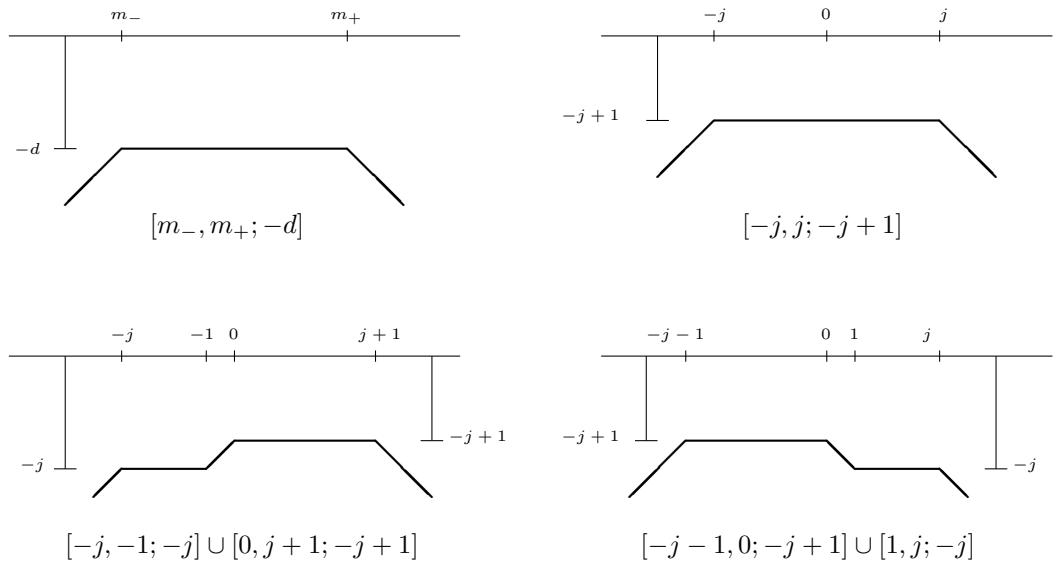
\begin{figure}[htb]\small
\begin{center}
\unitlength=0.75mm
\begin{picture}(80,55)(-10,-20)
\put(-5,30){\line(1,0){80}}
\put(15,30){\line(0,1){1}}
\put(15,30){\line(0,-1){1}}
\put(55,30){\line(0,1){1}}
\put(55,30){\line(0,-1){1}}
\put(13,33){\makebox{\tiny$m_-$}}
\put(52,33){\makebox{\tiny$m_+$}}
\put(3,10){\line(1,0){4}}
\put(5,30){\line(0,-1){20}}
\put(-4,9){\makebox{\tiny$-d$}}
\thicklines
\put(15,10){\line(-1,-1){10}}
\put(15,10){\line(1,0){40}}
\put(55,10){\line(1,-1){10}}
\put(20,-5){\makebox{$[m_-,m_+;-d]$}}
\end{picture}
\qquad\qquad\quad
\begin{picture}(80,55)(-10,-20)
\put(-5,30){\line(1,0){80}}
\put(15,30){\line(0,1){1}}
\put(15,30){\line(0,-1){1}}
\put(55,30){\line(0,1){1}}
\put(55,30){\line(0,-1){1}}
\put(35,30){\line(0,1){1}}
\put(35,30){\line(0,-1){1}}
\put(13,33){\makebox{\tiny$-j$}}
\put(55,33){\makebox{\tiny$j$}}
\put(34,33){\makebox{\tiny$0$}}
\put(3,15){\line(1,0){4}}
\put(5,30){\line(0,-1){15}}
\put(-12,15){\makebox{\tiny$-j+1$}}
\thicklines
\put(15,15){\line(-1,-1){10}}
\put(15,15){\line(1,0){40}}
\put(55,15){\line(1,-1){10}}
\put(20,-5){\makebox{$[-j,j;-j+1]$}}
\end{picture}
\vskip0.1cm\noindent
\begin{picture}(80,40)(-10,-5)
\put(-5,30){\line(1,0){80}}
\put(15,30){\line(0,1){1}}
\put(15,30){\line(0,-1){1}}
\put(60,30){\line(0,1){1}}
\put(60,30){\line(0,-1){1}}
\put(35,30){\line(0,1){1}}
\put(35,30){\line(0,-1){1}}
\put(30,30){\line(0,1){1}}
\put(30,30){\line(0,-1){1}}
\put(13,33){\makebox{\tiny$-j$}}
\put(57,33){\makebox{\tiny$j+1$}}
\put(34,33){\makebox{\tiny$0$}}
\put(27,33){\makebox{\tiny$-1$}}
\put(3,10){\line(1,0){4}}
\put(5,30){\line(0,-1){20}}
\put(-3,10){\makebox{\tiny$-j$}}
\put(68,15){\line(1,0){4}}
\put(70,30){\line(0,-1){15}}
\put(73,15){\makebox{\tiny$-j+1$}}
\thicklines
\put(15,10){\line(-1,-1){5}}
\put(15,10){\line(1,0){15}}
\put(30,10){\line(1,1){5}}
\put(35,15){\line(1,0){25}}
\put(60,15){\line(1,-1){10}}
\put(5,-5){\makebox{$[-j,-1;-j]\cup[0,j+1;-j+1]$}}
\end{picture}
\qquad\qquad\quad
\begin{picture}(80,40)(-10,-5)
\put(-5,30){\line(1,0){80}}
\put(10,30){\line(0,1){1}}
\put(10,30){\line(0,-1){1}}
\put(55,30){\line(0,1){1}}
\put(55,30){\line(0,-1){1}}
\put(35,30){\line(0,1){1}}
\put(35,30){\line(0,-1){1}}
\put(40,30){\line(0,1){1}}
\put(40,30){\line(0,-1){1}}
\put(3,33){\makebox{\tiny$-j-1$}}
\put(52,33){\makebox{\tiny$j$}}
\put(34,33){\makebox{\tiny$0$}}
\put(39,33){\makebox{\tiny$1$}}
\put(1,15){\line(1,0){4}}
\put(3,30){\line(0,-1){15}}
\put(-12,15){\makebox{\tiny$-j+1$}}
\put(63,10){\line(1,0){4}}
\put(65,30){\line(0,-1){20}}
\put(68,10){\makebox{\tiny$-j$}}
\thicklines
\put(10,15){\line(-1,-1){10}}
\put(10,15){\line(1,0){25}}
\put(35,15){\line(1,-1){5}}
\put(40,10){\line(1,0){15}}
\put(55,10){\line(1,-1){5}}
\put(2,-5){\makebox{$[-j-1,0;-j+1]\cup[1,j;-j]$}}
\end{picture}
\end{center}
\caption{Polygons packing the lattice sets.}\label{convex}
\end{figure}

\begin{proof}Let us identify $\big({\rm e}^x x^{\sigma-1}\big)^n x^{-d_+(n)}$, $x^{-1} x^{-d(0)}$ and $\big({\rm e}^{-x} x^{-\sigma-1}\big)^n x^{-d_-(n)}$ with $(n, -d_+(n))$, $(0,-d(0))$ and $(n, -d_-(n))$, respectively. Then we may canonically define the product of them and write, for example, $\big({\rm e}^x x^{\sigma-1}\big)^n x^{-d_+(n)} (-n', -d_-(n')) =(n, -d_+(n)) (-n', -d_-(n'))$ $=(n-n', -d_+(n)-d_-(n') -2n')$ if $n>n'$, $ \big({\rm e}^x x^{\sigma-1}\big)^n x^{-d_+(n)} (-n, -d_-(n)) =(0, -d_+(n)-d_-(n) -2n+1)$. The following formulas are easily obtained:
\begin{gather}\label{5.11}
 {\rm e}^x x^{\sigma-1} [-j,j; -j+1]=[-j+1,-1;-j-1] \cup [0,1; -j] \cup[2, j+1; -j+1],\\
\label{5.12}
 {\rm e}^x x^{\sigma-1} ( [-j-1,0; -j+1] \cup [1,j; -j] )=[-j,-1;-j-1] \cup [0,j+1; -j],
\end{gather}
in particular,
\begin{gather}\label{5.13}
 {\rm e}^x x^{\sigma-1} [-1,1; 0]\subset [1,2;0],\qquad {\rm e}^x x^{\sigma-1} [-2,0; 0]\subset [0,2;-1].
\end{gather}
Indeed, for example, \eqref{5.11} is verified by using
\begin{gather*}
[-j, j;-j+1] = [-j,-2;-j+1] \bigl|_{x_1 \le -2} \cup [-1, 0; -j+1]\bigl|_{-1\le x_1 \le 0}\cup [1,j; -j+1] \bigl|_{ x_1 \ge 1}.
\end{gather*}

We show the relations by induction on $j$. By virtue of the symmetric property of \eqref{5.2}, it is sufficient to focus on $\varphi^j$ and~$\varphi_+^j$. By~\eqref{5.3}, $\varpi\big(\varphi^1\big) \subset [-1,1; 0]$, $\varpi\big(\varphi^1_+\big), \varpi\big(\psi^1_-\big) \subset [0,2;0]$, $\varpi\big(\varphi^1_-\big), \varpi\big(\psi^1_+\big) \subset [-2,0;0]$. Note that $\mathcal{I}\big[\big({\rm e}^x x^{\sigma-1}\big)^m[x^{-l}\big] \big] =\big({\rm e}^x x^{\sigma-1}\big)^m [x^{-l}]$, $\mathcal{I}\big[x^{-1} \big[x^{-l-1}\big] \big] = \big[x^{-l-1}\big]$, $\mathcal{I}\big[\big({\rm e}^{-x} x^{-\sigma-1}\big)^m \big[x^{-l}\big] \big] =\big({\rm e}^{-x} x^{-\sigma-1}\big)^m \big[x^{-l}\big]$, where $m\in \N$, $l \in \N \cup \{0\}$. By $\tri \varphi^2=F_0\big(x, \varphi^1, \varphi^1_+, \psi^1_+, \varphi_-^1,\psi_-^1\big) -F_0(x, 0, 0, 0, 0, 0)$ and~\eqref{5.13} we have $\varpi\big(\tri\varphi^2\big) \subset [-2, 2;-1]$, and hence $\varpi\big(\varphi^2\big) \!\subset\! [-1, 1;0] \cup [-2,2; -1]= [-1,1;0]$. In obtaining this, we have used $\varpi\big({\rm e}^x x^{\sigma-2} \varphi_-^1 \varphi^1\big)\!$ $\subset x^{-1} {\rm e}^x x^{\sigma-1} [-2,0;0] [-1,1;0] \subset [0,2; -2] [-1,1;0]\subset [-1, 3; -2]\subset [0,2;-1]$. From
\begin{gather*}
\tri \varphi_+^2 = \mathcal{I}\big[ x^{-1} [1]\big( \big([1]+ \varphi_+^1\big)\tri \varphi^2 +\varphi^1 \varphi_+^1\big) + {\rm e}^x x^{\sigma-1} [1] \big( \big([1]+ \psi_+^1\big) \tri\varphi^2 + \varphi^1 \psi_+^1\big) \big],
\end{gather*}
we derive $\varpi\big(\tri \varphi^2_+\big) \subset [-2,-1; -2] \cup [0, 3; -1]$ and $\varpi\big(\varphi_+^2\big) \subset [0,2; 0]$ by using
\begin{gather*}
\varpi\big({\rm e}^x x^{\sigma-1} \tri \varphi^2\big) \subset {\rm e}^xx^{\sigma-1}[-2,2;-1] \subset {\rm e}^xx^{\sigma-1}([-2,-2;-1] \cup [-1,0;-1]\cup [1,2;-1])\\
\hphantom{\varpi\big({\rm e}^x x^{\sigma-1} \tri \varphi^2\big)}{} \subset [-1,-1; -3]\cup [0,1;-2]\cup [2,3;-1] \subset [0,1;-2]\cup[2,3;-1]\subset [0,3;-1],\\
\varpi\big({\rm e}^x x^{\sigma-1}\psi^1_+ \tri \varphi^2\big) \subset \psi^1_+\varpi\big({\rm e}^xx^{\sigma-1}\tri\varphi^2\big) \subset [-2,0;0] ([0,1;-2] \cup [2,3;-1])\\
\hphantom{\varpi\big({\rm e}^x x^{\sigma-1}\psi^1_+ \tri \varphi^2\big)}{} \subset [-2,1;-2]\cup [0,3;-1] \subset [-2,-1;-2] \cup [0,3;-1],\\
 \varpi\big({\rm e}^x x ^{\sigma-1} \varphi^1\psi_+^1\big) \subset {\rm e}^x x^{\sigma-1} [-2,0;0] [-1,1; 0]\subset [0,2; -1] ( [-1, -1; 0] \cup [0, 1; 0])\\
\hphantom{\varpi\big({\rm e}^x x ^{\sigma-1} \varphi^1\psi_+^1\big)}{} \subset [-1, 1; -2] \cup [0,3; -1]
\end{gather*}
and so on. Hence the assertion is valid for $j=2$.

Suppose that the assertion is valid for every integer $\le j$. From \eqref{5.2} it follows that
\begin{gather*}
\tri \varphi^{j+1} = {\rm e}^x x^{\sigma-1} [1] \big( \big([1]+\varphi_-^j\big) \tri \psi_+^j +\big([1] +\psi_+^{j-1}\big)\tri \varphi_-^j \big) + {\rm e}^{-x} x^{-\sigma-1} (\cdots)\\
\hphantom{\tri \varphi^{j+1} =}{}+ \mathcal{I} \bigl[ {\rm e}^x x^{\sigma-2} [1] \bigl( \big([1]+\varphi_-^j\big)\big([1]+\psi_+^j\big)\tri \varphi^j + \big([1]+\varphi_-^j\big) \varphi^{j-1}\tri\psi_+^j\\
\hphantom{\tri \varphi^{j+1} =}{} +\big([1] +\psi_+^{j-1}\big) \varphi^{j-1} \tri \varphi_-^j \bigr) + {\rm e}^{-x} x^{-\sigma-2}(\cdots) \bigr]\\
\hphantom{\tri \varphi^{j+1} =}{} + \mathcal{I} \bigl[ x^{-3} \bigl( [1] \big([1]+\varphi_+^j\big)\big([1]+\varphi_-^j\big)\tri \varphi^j + \big([1]+[1]\varphi^{j-1}\big) \big([1]+ \varphi_-^{j}\big)\tri\varphi_+^j\\
\hphantom{\tri \varphi^{j+1} =}{} + \big([1]+[1]\varphi^{j-1}\big) \big([1]+ \varphi_+^{j-1}\big)\tri\varphi_-^j \bigr)+ x^{-3}(\cdots) \bigr].
\end{gather*}
By \eqref{5.11} and \eqref{5.12} we have
\begin{gather*}
\varpi\big({\rm e}^xx^{\sigma-2} \tri\varphi^j\big) \subset [-j+1,-1;-j-2] \cup [0,1; -j-1] \cup [2,j+1;-j]\\
\hphantom{\varpi\big({\rm e}^xx^{\sigma-2} \tri\varphi^j\big)}{} \subset [-j-1,j+1;-1],\\
\varpi\big({\rm e}^xx^{\sigma-2} \varphi^{j-1}\tri\varphi_-^j\big) \subset [-1,1;0]([-j,-1;-j-2]\cup [0,j+1;-j-1])\\
\hphantom{\varpi\big({\rm e}^xx^{\sigma-2} \varphi^{j-1}\tri\varphi_-^j\big)}{} \subset [-j-1,0;-j-2]\cup [-1,j+2;-j-1] \subset [-j-1,j+1;-j],\\
\varpi\big({\rm e}^xx^{\sigma-2} \varphi^{j-1}\psi^j_+\tri\varphi_-^j\big) \subset [-2,0;0] ([-j-1,0;-j-2]\cup [-1,j+2;-j-1])\\
\hphantom{\varpi\big({\rm e}^xx^{\sigma-2} \varphi^{j-1}\psi^j_+\tri\varphi_-^j\big)}{} \subset [-j-1,j+1;-j]
\end{gather*}
and so on. From these it follows that $\varpi\big(\tri \varphi^{j+1}\big) \subset [-j-1, j+1; -j]$ and consequently $\varpi\big(\varphi^{j+1}\big) \subset [-1,1;0]$. Furthermore,
\begin{gather*}
\tri \varphi_+^{j+1} =\mathcal{I} \bigl[ x^{-1} [1] \big( \big([1] +\varphi_+^j\big) \tri\varphi^{j+1} +\varphi^j \tri \varphi_+^j\big)\\
\hphantom{\tri \varphi_+^{j+1} =}{} + {\rm e}^x x^{\sigma-1} [1] \big(\big([1] +\psi_+^j\big) \tri \varphi^{j+1}+ \varphi^j \tri \psi_+^j\big)\bigr].
\end{gather*}
Observing that, by \eqref{5.12},
\begin{gather*}
\varpi\big({\rm e}^x x ^{\sigma-1} \varphi^j\tri \psi_+^j\big) \subset [-1,1; 0] ( [-j, -1; -j-1] \cup [0, j+1; -j])\\
 \qquad{} \subset [-j-1, 0; -j-1] \cup [-1, 1; 0] [0,j+1; -j] \subset [-j-1, j; -j-1] \cup [0, j+2; -j],
\end{gather*}
where $[-1,1;0] [0,j+1;-j]=([-1,-1;0] \cup [0, 1;0]) [0,j+1; -j] \subset [-1, j; -j-1] \cup [0,j+2;-j]$, and so on, we have $\varpi\big(\tri \varphi_+^{j+1}\big) \subset [-j-1, -1; -j-1] \cup [0, j+2; -j]$ and $\varpi\big(\varphi_+^{j+1}\big) \subset [0,2;0]$. Thus we obtain the proposition.
\end{proof}

\begin{prop}\label{prop5.4}The summands of $\varphi^j$, $\varphi_{\pm}^j$ and $\psi_{\pm}^j$ satisfy the following: if $n \ge 1$,
\begin{enumerate}\itemsep=0pt
\item[$(i)$] $p^+_n(x) =\big(\gamma^0_- \gamma_+^x\big)^n [1]$, $p^-_n(x) =\big(\gamma^0_+ \gamma_-^x\big)^n [1]$ for $\varphi^j$;
\item[$(ii)$] $p^+_n(x) =\gamma_+^x\big(\gamma^0_- \gamma_+^x\big)^{n-1} [1]$, $p^-_n(x) =\gamma_+^0\big(\gamma^0_+ \gamma_-^x\big)^n [1]$ for $\varphi_+^j$;
\item[$(iii)$] $p^+_n(x) =\gamma_+^x\big(\gamma^0_- \gamma_+^x\big)^{n} [1]$, $p^-_n(x) =\gamma_+^0\big(\gamma^0_+ \gamma_-^x\big)^{n-1} [1]$ for $\psi_+^j$;
\item[$(iv)$] $p^+_n(x) =\gamma_-^0\big(\gamma^0_- \gamma_+^x\big)^{n} [1]$, $p^-_n(x) =\gamma_-^x\big(\gamma^0_+ \gamma_-^x\big)^{n-1} [1]$ for $\varphi_-^j$;
\item[$(v)$] $p^+_n(x) =\gamma_-^0\big(\gamma^0_- \gamma_+^x\big)^{n-1} [1]$, $p^-_n(x) =\gamma_-^x\big(\gamma^0_+ \gamma_-^x\big)^{n} [1]$ for $\psi_-^j$.
\end{enumerate}
Furthermore $p_0(x)=\gamma_{\pm}^0 [1]$ for $\varphi_{\pm}$, and $p_0(x)=\gamma_{\pm}^x [1]$ for $\psi_{\pm}$.
\end{prop}

\begin{proof}Note that the relations for $\varphi^{j+1}$, $\varphi_+^{j+1}$ and $\psi_-^{j+1}$ in~\eqref{5.2} are rewritten in the form
\begin{gather*}
\varphi^{j+1} = F_0\big(x, \varphi^j, {\rm e}^xx^{\sigma}\varphi_{+*}^j, \psi_+^j, \varphi_-^j,{\rm e}^x x^{\sigma}\psi_{-*}^j\big),\\
\varphi_{+*}^{j+1} = {\rm e}^{-x}x^{-\sigma} F_+\big(x, \varphi^{j+1}, {\rm e}^xx^{\sigma}\varphi_{+*}^j, \psi_+^j\big),\\
\psi_{-*}^{j+1} = {\rm e}^{-x}x^{-\sigma} G_-\big(x, \varphi^{j+1}, \varphi_-^j,{\rm e}^xx^{\sigma}\psi_{-*}^j\big),
\end{gather*}
where $\varphi_{+*}^j= {\rm e}^{-x}x^{-\sigma} \varphi_+^j$, $\psi_{-*}^j= {\rm e}^{-x}x^{-\sigma} \psi_-^j$. Combining these with the relations for~$\varphi_-^{j+1}$ and~$\psi_+^{j+1}$ in~\eqref{5.2}, by induction on~$j$ we may verify the facts: if $n\ge 0$,
\begin{enumerate}\itemsep=0pt
\item[(a)] $p_n^+(x)=\big(\gamma_-^0 \gamma_+^x\big)^n [1]$ for $\varphi^j $;
\item[(b)] $p_n^+(x)=\gamma_+^x \big(\gamma_-^0 \gamma_+^x\big)^n [1]$ for $\varphi_{+*}^j $, $\psi_+^j$;
\item[(c)] $p_n^+(x)=\gamma_-^0 \big(\gamma_-^0 \gamma_+^x\big)^n [1]$ for $\varphi_{-}^j $, $\psi_{-*}^j$.
\end{enumerate}
Similarly, if $n\ge 0$,
\begin{enumerate}\itemsep=0pt
\item[(a$'$)] $p_n^-(x)=\big(\gamma_+^0 \gamma_-^x\big)^n [1]$ for $\varphi^j $;
\item[(b$'$)] $p_n^-(x)=\gamma_+^0\big(\gamma_+^0 \gamma_-^x\big)^n [1]$ for $\varphi_{+}^j$, $\psi_{+*}^j$;
\item[(c$'$)] $p_n^-(x)=\gamma_-^x\big(\gamma_+^0 \gamma_-^x\big)^n [1]$ for $\varphi_{-*}^j$, $\psi_{-}^j$,
\end{enumerate}
where $\psi_{+*}^j= {\rm e}^x x^{\sigma}\psi_+^j$, $\varphi_{-*}^j= {\rm e}^x x^{\sigma}\varphi_-^j$. Then the proposition immediately follows.
\end{proof}

\subsection{Completion of the proof of Theorem \ref{thm2.1}}\label{ssc5.4}

By Proposition \ref{prop5.3}, for $j\ge 2$, $ \big\| \tri \varphi^{j} \big\| , \big\| \tri \varphi_{\pm}^{j} \big\|, \big\| \tri \psi_{\pm}^{j} \big\| \ll |x|^{-j+1}$, and hence $ \big\|\big({\rm e}^x x^{\sigma}\big)^{\pm 1} \tri \varphi^{j}\big\|$, $ \big\|\big({\rm e}^x x^{\sigma}\big)^{\mp 1} \tri \varphi_{\pm}^{j} \big\| , \big\|\big({\rm e}^x x^{\sigma}\big)^{\pm 1} \tri \psi_{\pm}^{j} \big\| \ll |x|^{-j+2}$ if $\big|\eta x^{-1}\big|, \big|\eta^{-1} x^{-1}\big|<\varepsilon$, $x\in \Sigma_0\big(x^1_{\infty}, \delta\big)$, the implied constants possibly depending on~$j$. Let~$N$ be a given positive integer. Combining this fact with Proposition~\ref{prop5.2} we derive that, for every $j\ge N+1$,
\begin{gather*}
\big\| \tri \varphi^{j} \big\| , \big\| \tri \varphi_{\pm}^{j} \big\| , \big\| \tri \psi_{\pm}^{j} \big\| \le \Psi_{j-1} \le 2^{-j+1+N}\Psi_N \ll 2^{-j+N}|x|^{-N+1},
\end{gather*}
the implied constant not depending on $j$ but possibly on $N$. By Proposition~\ref{prop3.2} we conclude that
\begin{gather*}
\varphi^{\infty} =\sum_{j=2}^{\infty} \tri \varphi^j +\varphi^1, \qquad \varphi_{\pm}^{\infty} =\sum_{j=2}^{\infty} \tri \varphi_{\pm}^j +\varphi_{\pm}^1, \qquad \psi_{\pm}^{\infty} =\sum_{j=2}^{\infty} \tri \psi_{\pm}^j +\psi_{\pm}^1
\end{gather*}
belong to $\mathfrak{A}=\mathfrak{A}\big(B_0, B_x, B_*, \Sigma_0\big(x^1_{\infty},\delta\big), \varepsilon\big)$ if $\varepsilon$ fulfils \eqref{5.4} with $r_0$ chosen as in Proposition~\ref{prop5.2}. Thus we have constructed a solution $\big(\varphi, \varphi_{\pm},\psi_{\pm}\big)= \big(\varphi^{\infty}, \varphi^{\infty}_{\pm},\psi^{\infty}_{\pm}\big)$ of~\eqref{5.1} and of~\eqref{4.5}. By~\eqref{5.3}, Propositions~\ref{prop3.2} and~\ref{prop5.4},
$\varphi^{\infty}$, $\varphi_{\pm}^{\infty}$, $\psi_{\pm}^{\infty}$ are written in the form
\begin{gather*}
\varphi^{\infty} = -\frac 12 \big((\sigma+\theta_{\infty})\gamma^0_+\gamma^0_-+ (\sigma-\theta_{\infty}) \gamma^x_+\gamma^x_- +\big[x^{-1}\big]\big) x^{-2}\\
\hphantom{\varphi^{\infty} =}{} + \gamma_-^0\gamma_+^x (1)_x {\rm e}^x x^{\sigma-1}+ \sum_{n=2}^{\infty}\big( \gamma_-^0\gamma_+^x\big)^n \big[x^{-n+1}\big]\big( {\rm e}^x x^{\sigma-1}\big)^n\\
\hphantom{\varphi^{\infty} =}{} + \gamma_+^0\gamma_-^x (1)_x {\rm e}^{-x} x^{-\sigma-1} + \sum_{n=2}^{\infty}\big(\gamma_+^0\gamma_-^x\big)^n \big[x^{-n+1}\big]\big( {\rm e}^{-x} x^{-\sigma-1}\big)^n,\\
\varphi_+^{\infty} = \gamma_+^0\big[x^{-1}\big] - \gamma_+^x\big(2\gamma_+^0\gamma_-^0 +\big[x^{-1}\big]\big) {\rm e}^x x^{\sigma-2}-\gamma_-^0\big(\gamma_+^x\big)^2(1)_x \big({\rm e}^x x^{\sigma-1}\big)^2\\
\hphantom{\varphi_+^{\infty} =}{} + \sum_{n=3}^{\infty} \gamma_+^x \big( \gamma_-^0\gamma_+^x\big)^{n-1} \big[x^{-n+2}\big]\big( {\rm e}^x x^{\sigma-1}\big)^n+ 2\big(\gamma_+^0\big)^2\gamma_-^x (1)_x {\rm e}^{-x} x^{-\sigma-2}\\
\hphantom{\varphi_+^{\infty} =}{} + \sum_{n=2}^{\infty} \gamma_+^0\big( \gamma_+^0\gamma_-^x\big)^n \big[x^{-n}\big]\big( {\rm e}^{-x} x^{-\sigma-1}\big)^n,\\
\psi_+^{\infty} = \gamma_+^x \big[x^{-1}\big] +2 \gamma_-^0\big(\gamma_+^x\big)^2 (1)_x {\rm e}^x x^{\sigma-2} + \sum_{n=2}^{\infty} \gamma_+^x\big( \gamma_-^0\gamma_+^x\big)^n \big[x^{-n}\big]\big( {\rm e}^x x^{\sigma-1}\big)^n\\
\hphantom{\psi_+^{\infty} =}{} - \gamma_+^0\big(2\gamma_+^x \gamma_-^x + \big[x^{-1}\big]\big) {\rm e}^{-x} x^{-\sigma-2}- \big(\gamma_+^0\big)^2\gamma_-^x (1)_x \big({\rm e}^{-x}x^{-\sigma-1}\big)^2\\
\hphantom{\psi_+^{\infty} =}{} + \sum_{n=3}^{\infty}\gamma_+^0\big( \gamma_+^0\gamma_-^x\big)^{n-1} \big[x^{-n+2}\big]\big( {\rm e}^{-x} x^{-\sigma-1}\big)^n,\\
\varphi_-^{\infty} = \gamma_-^0 \big[x^{-1}\big]+2 \big(\gamma_-^0\big)^2\gamma_+^x (1)_x {\rm e}^x x^{\sigma-2}+ \sum_{n=2}^{\infty}\gamma_-^0\big( \gamma_-^0\gamma_+^x\big)^n \big[x^{-n}\big]\big( {\rm e}^x x^{\sigma-1}\big)^n\\
\hphantom{\varphi_-^{\infty} =}{} - \gamma_-^x\big(2\gamma_+^0 \gamma_-^0 + \big[x^{-1}\big]\big) {\rm e}^{-x} x^{-\sigma-2} - \gamma_+^0\big(\gamma_-^x\big)^2 (1)_x \big({\rm e}^{-x} x^{-\sigma-1}\big)^2\\
\hphantom{\varphi_-^{\infty} =}{} + \sum_{n=3}^{\infty}\gamma_-^x\big( \gamma_+^0\gamma_-^x\big)^{n-1} \big[x^{-n+2}\big]\big( {\rm e}^{-x} x^{-\sigma-1}\big)^n,\\
\psi_-^{\infty} = \gamma_-^x\big[x^{-1}\big] - \gamma_-^0\big(2\gamma_+^x\gamma_-^x +\big[x^{-1}\big]\big) {\rm e}^x x^{\sigma-2}-\big(\gamma_-^0\big)^2 \gamma_+^x (1)_x \big({\rm e}^x x^{\sigma-1}\big)^2\\
\hphantom{\psi_-^{\infty} =}{} + \sum_{n=3}^{\infty}\gamma_-^0\big( \gamma_-^0\gamma_+^x\big)^{n-1} \big[x^{-n+2}\big]\big( {\rm e}^x x^{\sigma-1}\big)^n+ 2\gamma_+^0\big(\gamma_-^x\big)^2 (1)_x {\rm e}^{-x} x^{-\sigma-2}\\
\hphantom{\psi_-^{\infty} =}{} + \sum_{n=2}^{\infty}\gamma_-^x\big( \gamma_+^0\gamma_-^x\big)^n \big[x^{-n}\big]\big( {\rm e}^{-x} x^{-\sigma-1}\big)^n.
\end{gather*}
Then, by \eqref{4.7}, $\Phi_0(x), \Phi_x(x) \to 0$ as $x\to \infty$ along a curve on which $\big|{\rm e}^x x^{\sigma}\big|=1$. Substitution of these into~\eqref{4.6} leads us to the desired solution of Theorem~\ref{thm2.1}.

\begin{rem}\label{rem5.11}By Remark \ref{rem4.3}, in the first equation of \eqref{5.1},
\begin{gather*}
 F_0(x,\varphi, \varphi_+, \psi_+, \varphi_-, \psi_-):= {\rm e}^{x}x^{\sigma-1}\big(1-(\sigma-1)x^{-1} +2\kappa(x)\big) \big(\gamma^0_- +\varphi_-\big)\big(\gamma^x_+ + \psi_+\big)\\
\qquad{} + {\rm e}^{-x}x^{-\sigma-1}(1-(\sigma+1)x^{-1} -2\kappa(x)) \big(\gamma^0_+ +\varphi_+\big)\big(\gamma^x_-+ \psi_-\big) - \cdots.
\end{gather*}
Using this fact, \eqref{4.6} and Proposition \ref{prop5.3}, and computing $\varphi^2$, $\varphi_{\pm}^2$, $\psi_{\pm}^2$, we may write some terms of the expressions for $f_0$, $f_{\pm}$, $g_{\pm}$ in Theorem \ref{thm2.1}
in more detail:
\begin{gather*}
 f_0 = \cdots + \gamma^0_-\gamma^x_+ \bigl(1-(\sigma-1 + 2\big(\gamma_+^0\gamma_-^0 + \gamma_+^x\gamma_-^x\big) -\big(\sigma^2-\theta_{\infty}^2\big)/2 )x^{-1}+\big[x^{-2}\big] \bigr) {\rm e}^x x^{\sigma-1}\\
\hphantom{f_0 =}{} + \gamma^0_+\gamma^x_- \bigl(1-\big(\sigma+1 - 2\big(\gamma_+^0\gamma_-^0 + \gamma_+^x\gamma_-^x\big) +\big(\sigma^2-\theta_{\infty}^2\big)/2 \big)x^{-1}+\big[x^{-2}\big] \bigr) {\rm e}^{-x} x^{-\sigma-1} + \cdots\!,\\
 x^{(\sigma +\theta_{\infty})/2} f_+ = \gamma^0_+ \bigl(1+ \big(2\gamma^x_+\gamma^x_- - \big(\sigma^2-\theta_{\infty}^2\big)/4\big) x^{-1} +\big[x^{-2}\big] \bigr) +\cdots,\\
 {\rm e}^{-x} x^{-(\sigma -\theta_{\infty})/2} g_+ = \gamma^x_+ \bigl(1- \big(2\gamma^0_+\gamma^0_- - \big(\sigma^2-\theta_{\infty}^2\big)/4\big) x^{-1} +\big[x^{-2}\big] \bigr) +\cdots,\\
 x^{-(\sigma +\theta_{\infty})/2} f_- = \gamma^0_- \bigl(1- \big(2\gamma^x_+\gamma^x_- - \big(\sigma^2-\theta_{\infty}^2\big)/4\big) x^{-1} +\big[x^{-2}\big] \bigr) +\cdots,\\
 {\rm e}^{x} x^{(\sigma -\theta_{\infty})/2} g_- = \gamma^x_- \bigl(1+ \big(2\gamma^0_+ \gamma^0_- - \big(\sigma^2-\theta_{\infty}^2\big)/4\big) x^{-1} +\big[x^{-2}\big] \bigr) + \cdots.
\end{gather*}
These facts are used in computing the tau-function.
\end{rem}

\subsection{Proof of Theorem \ref{thm2.2}}\label{ssc5.5}
In the proof of Theorem \ref{thm2.1} described above, we put $\sigma=\sigma_0 =-2\theta_x-\theta_{\infty}$, namely $\gamma_-^x=0$, to obtain the solution of Theorem~\ref{thm2.2} in the sector $|\arg x -\pi/2|<\pi/2 -\delta$,
$\big|{\rm e}^x x^{\sigma_0 -1} \big|<\varepsilon$, since the restriction $\big|{\rm e}^{-x} x^{-\sigma_0-1}\big|< \varepsilon$ is removed. It is sufficient to show that this expression may be extended to the sector $|\arg x-\pi|<\pi/2-\delta$. In the sector $|\arg x -\pi/2| <\pi/2 -\delta$, write $\big(\varphi^{\infty}, \varphi^{\infty}_{\pm}, \psi^{\infty}_{\pm}\big)$ with $\sigma=\sigma_0$ in the form
\begin{alignat}{3}
&\varphi^{\infty}= p(x) + {\rm e}^x x^{\sigma_0 } \hat{\varphi}^{\infty}, \qquad && \varphi_{\pm}^{\infty}= p_{\pm}(x)+ {\rm e}^x x^{\sigma_0 }\hat{\varphi}_{\pm}^{\infty},&\nonumber\\
& \psi_+^{\infty}= q_+(x) {\rm e}^{-x}x^{-\sigma_0} +\hat{\psi}_+^{\infty}, \qquad && \psi_-^{\infty}=q_-(x) {\rm e}^x x^{\sigma_0} + \big({\rm e}^x x^{\sigma_0 }\big)^2\hat{\psi}_-^{\infty},& \label{5.14}
\end{alignat}
where $p(x)=\big[x^{-2}\big]$, $p_{\pm}(x)=\big[x^{-1}\big]$, $q_{\pm}(x)=\big[x^{-3}\big]$, and $\hat{\varphi}^{\infty}, \hat{\varphi}_{\pm}^{\infty},\hat{\psi}_{\pm}^{\infty} \in \mathfrak{A}\big(\Sigma_0\big(x^1_{\infty},\delta\big), \varepsilon\big)$ $\big({=}\, \mathfrak{A}\big(B_0, B_x, \{\sigma_0\}, \Sigma_0\big(x^1_{\infty},\delta\big), \varepsilon\big)\big)$. Note that $\hat{\varphi}^{\infty}$, $\hat{\varphi}_{\pm}^{\infty}$, $\hat{\psi}_{\pm}^{\infty} $ also belong to $\mathfrak{A}_+\big(\Sigma_{\pi}\big(\pi/2+\delta, \pi-\delta; x^1_{\infty}\big), \varepsilon\big)$. Recall that $\big(\varphi^{\infty}, \varphi^{\infty}_{\pm}, \psi^{\infty}_{\pm}\big)$ solves \eqref{5.1} with $\gamma_-^x=0$. Inserting \eqref{5.14} into this system and putting $\gamma_+^x=0$, we find that $(p(x), p_{\pm}(x), q_{\pm} (x))$ solves \eqref{5.1} with $\gamma_-^x=\gamma_+^x=0$, namely~\eqref{5.1} with
\begin{gather*}
 F_0=F_0^* (x, \varphi, \varphi_+, \psi_+,\varphi_-, \psi_-) =x^{-1}\big( (1)_x\big(\gamma_-^0+\varphi_-\big)\psi_+ + (1)_x \big(\gamma^0_+ +\varphi_+\big)\psi_-\big)\\
\hphantom{F_0=}{}-4 \mathcal{I}\bigl[ x^{-2}\varphi ( (1)_x\big(\gamma_-^0+\varphi_-\big)\psi_+- (1)_x \big(\gamma^0_+ +\varphi_+\big) \psi_-\big) \bigr]\\
\hphantom{F_0=}{} +\mathcal{I} \bigl[ x^{-3} \big(([1]+[1]\varphi) \big(\gamma^0_+ +\varphi_+\big)\big(\gamma^0_- +\varphi_-\big) + ([1]+[1]\varphi) \psi_+\psi_-\big) \bigr],\\
 F_+=F_+^* (x, \varphi, \varphi_+, \psi_+) = -2\mathcal{I}\bigl[ x^{-1}\varphi \big((1)_x \big(\gamma^0_+ +\varphi_+\big) +(1)_x \psi_+\big) \bigr],\\
 G_+=G_+^* (x, \varphi, \varphi_+, \psi_+) = 2 {\rm e}^xx^{\sigma_0} \mathcal{I}\bigl[{\rm e}^{-x}x^{-\sigma_0-1}\varphi \big((1)_x \big(\gamma^0_+ +\varphi_+\big) +(1)_x \psi_+\big) \bigr],\\
 F_-=F_-^* (x, \varphi, \varphi_-, \psi_-) = 2\mathcal{I}\bigl[ x^{-1}\varphi \big((1)_x \big(\gamma^0_- +\varphi_-\big) +(1)_x \psi_-\big) \bigr],\\
 G_-=G_-^* (x, \varphi, \varphi_-, \psi_-) =- 2 {\rm e}^{-x}x^{-\sigma_0} \mathcal{I}\bigl[{\rm e}^{x} x^{\sigma_0-1}\varphi \big((1)_x \big(\gamma^0_- +\varphi_-\big) +(1)_x \psi_-\big)\bigr].
\end{gather*}
Let $(\varphi^j, \varphi_{\pm}^j, \psi_{\pm}^j)$ be the sequence defined by
\begin{gather*}
 \varphi^0=\varphi_{\pm}^0=\psi_{\pm}^0\equiv 0,\\
 \varphi^{j+1}=F^*\big(x, \varphi^j, \varphi^j_+, \psi^j_+, \varphi^j_-, \psi^j_-\big),\\
 \varphi_{\pm}^{j+1}=F^*_{\pm}\big(x, \varphi^{j+1}, \varphi^j_{\pm}, \psi^j_{\pm}\big),\qquad \psi_{\pm}^{j+1}=G^*_{\pm}\big(x, \varphi^{j+1}, \varphi^j_{\pm}, \psi^j_{\pm}\big).
\end{gather*}
Here, for $|\arg x-\pi/2|<\pi-\delta$, we may replace the path $\gamma(x)$ by a ray tending to $\infty {\rm e}^{{\rm i}\vartheta}$ such that $|\vartheta|<\pi/2-\delta$ (respectively, $|\vartheta-\pi|<\pi/2-\delta$) in $G^*_+$ (respectively, $G^*_-$) and by one tending to $\infty {\rm e}^{{\rm i}\arg x}$ in the others. Then the sequence converges to $\big(p^{\infty}(x), p^{\infty}_{\pm}(x), q^{\infty}_{\pm}(x)\big)$ whose entries admit asymptotic expansions in the sector $|\arg x-\pi/2|<\pi -\delta$. Since~\eqref{5.1} with $\gamma^x_-=\gamma^x_+=0$ whose paths are replaced as above has a unique solution tending to $0$ in this sector, the asymptotic expression for $(p(x), p_{\pm}(x), q_{\pm}(x))$ is valid in the sector $|\arg x-\pi/2| <\pi-\delta$. By the fact that $(p(x), p_{\pm}(x), q_{\pm}(x))$ solves this system, $\big(\hat{\varphi},\hat{\varphi}_{\pm},\hat{\psi}_{\pm}\big) = \big(\hat{\varphi}^{\infty},\hat{\varphi}_{\pm}^{\infty}, \hat{\psi}_{\pm}^{\infty}\big) $ satisfies a~system of the form
\begin{gather*}
\hat{\varphi}=\big[x^{-1}\big]+x^{-1} \big( \big[x^{-2}\big] \hat{\varphi}_+ + (*) \hat{\varphi}_-+[1] \hat{\psi}_+ + [1] \hat{\psi}_- + (*)\hat{\varphi}_-\hat{\psi}_+ + (*) \hat{\varphi}_+\hat{\psi}_- \big)\\
\hphantom{\hat{\varphi}=}{} +\big({\rm e}^x x^{\sigma_0}\big)^{-1} \mathcal{I} \bigl[ {\rm e}^x x^{\sigma_0-1}\bigl( (*) \hat{\varphi} + \big[x^{-1}\big]\hat{\varphi}_+ +(*) \hat{\varphi}_-
+\big[x^{-1}\big] \hat{\psi}_+ + (*) \hat{\psi}_-\\
\hphantom{\hat{\varphi}=}{} +(*) \hat{\varphi}_+\hat{\varphi}_- +(*) \hat{\psi}_+\hat{\psi}_- +(*) \hat{\varphi}_-\hat{\psi}_+ +(*) \hat{\varphi}_+\hat{\psi}_-+ \hat{\varphi} \big( (*) \hat{\varphi}_+ +(*) \hat{\varphi}_-\\
\hphantom{\hat{\varphi}=}{} +(*) \hat{\psi}_+ + (*) \hat{\psi}_- +(*) \hat{\varphi}_+\hat{\varphi}_- +(*) \hat{\psi}_+\hat{\psi}_- +(*) \hat{\varphi}_-\hat{\psi}_+ +(*) \hat{\varphi}_+\hat{\psi}_- \big)\bigr) \bigr],\\
 \hat{\varphi}_+=\big[x^{-1}\big] +\big({\rm e}^x x^{\sigma_0}\big)^{-1} \\
 \hphantom{\hat{\varphi}_+=}{}\times \mathcal{I} \bigl[ {\rm e}^x x^{\sigma_0-1}
\bigl( ([1]+(*)) \hat{\varphi} + \big[x^{-2}\big]\hat{\varphi}_+ + \big[x^{-2}\big]\hat{\psi}_++ \hat{\varphi} \big( (*) \hat{\varphi}_+ +(*)\hat{\psi}_+\big) \bigr) \bigr],\\
 \hat{\psi}_+=\big[x^{-1}\big]+ \mathcal{I} \bigl[ x^{-1}\bigl( ([1]+(*)) \hat{\varphi} + \big[x^{-2}\big]\hat{\varphi}_+ + \big[x^{-2}\big]\hat{\psi}_++ \hat{\varphi} ( (*) \hat{\varphi}_+ +(*)\hat{\psi}_+) \bigr) \bigr],\\
 \hat{\varphi}_-= \big({\rm e}^x x^{\sigma_0}\big)^{-1} \mathcal{I} \bigl[ {\rm e}^x x^{\sigma_0-1}\bigl( [1] \hat{\varphi} + \big[x^{-2}\big]\hat{\varphi}_- + \big[x^{-2}\big]\hat{\psi}_-+ \hat{\varphi} \big( (*) \hat{\varphi}_- +(*)\hat{\psi}_-\big) \bigr) \bigr],\\
 \hat{\psi}_-= \big({\rm e}^x x^{\sigma_0}\big)^{-2} \mathcal{I} \bigl[ \big({\rm e}^x x^{\sigma_0}\big)^2 x^{-1}\bigl( [1] \hat{\varphi} + \big[x^{-2}\big]\hat{\varphi}_- + \big[x^{-2}\big]\hat{\psi}_-
+ \hat{\varphi} \big( (*) \hat{\varphi}_- +(*)\hat{\psi}_-\big) \bigr) \bigr],
\end{gather*}
where every asymptotic coefficient is valid in the sector $|\arg x-\pi/2| <\pi-\delta$, and each $(*)$ denotes a function of the form $\big[x^{-1}\big] +[1]{\rm e}^x x^{\sigma_0} + [1] \big({\rm e}^x x^{\sigma_0}\big)^2$. For $|\arg x-\pi|<\pi/2-\delta$, replace $\gamma(x)$ by $\gamma_{\pi}(x)$ (cf.\ Section~\ref{ssc3.2}), and define the sequence $\big(\hat{\varphi}^j,\hat{\varphi}^j_{\pm},\hat{\psi}^j_{\pm}\big)$ by the same way as in~\eqref{5.2}. Then, using the facts in Section~\ref{ssc3.2}, we may construct a solution $\big(\hat{\varphi}^*,\hat{\varphi}^*_{\pm},\hat{\psi}^*_{\pm}\big)$ whose entries are in $\mathfrak{A}_+\big(\Sigma_{\pi}\big(\pi/2 +\delta, 3\pi/2-\delta, x^1_{\infty}\big), \varepsilon\big)$. This coincides with $\big(\hat{\varphi}^{\infty},\hat{\varphi}^{\infty}_{\pm},\hat{\psi}^{\infty}_{\pm}\big)$ in the sector $\pi/2 +\delta <\arg x<\pi-\delta$, since the corresponding asymptotic coefficients of these solutions satisfy the same recursive relation. This completes the proof of Theorem~\ref{thm2.2}.

\section{Proofs of the results on (V)} \label{sc6}

Let $(f_0, f_{\pm}, g_{\pm})$ be the solution given by Theorem \ref{thm2.1} or \ref{thm2.2}, which has been obtained by constructing $(\varphi^{\infty}, \varphi^{\infty}_{\pm}, \psi^{\infty}_{\pm})$ that solves~\eqref{5.1} in Section~\ref{ssc2.3}. Then, by~\eqref{1.2},
\begin{gather}\label{6.1}
y= \frac {g_+ (f_0 +\theta_0/2)} {f_+ (g_0 +\theta_x/2 )}, \qquad g_0 =-f_0 -\theta_{\infty}/2
\end{gather}
is a solution of (V).

\subsection{Proofs of Theorems \ref{thm2.6} and \ref{thm2.7}} \label{ssc6.1}

We begin with the following:
\begin{prop}\label{prop6.1}The solution $y$ depends on the parameters $\sigma$ and $c=c_x/c_0$ $($respectively, $c'=c_0/c_x)$ only.
\end{prop}

\begin{proof}Note that the coefficients of each asymptotic series $\big[x^{-1}\big]$ in~\eqref{5.1} are in $\Q[\theta_0, \theta_x, \theta_{\infty}, \sigma]$. We may suppose that $\gamma^0_{\pm}, \gamma^x_{\pm} \not=0$. Set
$\varphi_+=\gamma_+^0 \tilde{\varphi}_+$, $\psi_+=\gamma^x_+ \tilde{\psi}_+$, $\varphi_-=\gamma_-^0 \tilde{\varphi}_-$, $\psi_-=\gamma^x_- \tilde{\psi}_-$. Then \eqref{5.1} becomes{\allowdisplaybreaks
\begin{gather}
 \varphi = \gamma_-^0\gamma_+^x (1)_x {\rm e}^x x^{\sigma-1} \big(1+\tilde{\varphi}_-\big)\big(1+\tilde{\psi}_+\big)+ \gamma_+^0\gamma_-^x (1)_x {\rm e}^{-x}x^{-\sigma-1} \big(1+\tilde{\varphi}_+\big)\big(1+\tilde{\psi}_-\big)\nonumber\\
\hphantom{\varphi =}{} -4 \mathcal{I}\bigl[ \gamma_-^0\gamma_+^x (1)_x {\rm e}^x x^{\sigma-2}\varphi \big(1+\tilde{\varphi}_-\big)\big(1+\tilde{\psi}_+\big)
 - \gamma_+^0\gamma_-^x (1)_x {\rm e}^{-x}x^{-\sigma-2} \varphi \big(1+\tilde{\varphi}_+\big)\big(1+\tilde{\psi}_-\big) \bigr]\nonumber\\
\hphantom{\varphi =}{} +\mathcal{I} \bigl[ \gamma_+^0\gamma_-^0 x^{-3} ([1]+[1]\varphi) \big(1+\tilde{\varphi}_+\big)\big(1+\tilde{\varphi}_-\big)
 + \gamma_+^x\gamma_-^x x^{-3} ([1]+[1]\varphi)\big(1+\tilde{\psi}_+\big)\big(1+\tilde{\psi}_-\big) \bigr],\nonumber\\
 \tilde{\varphi}_+ = -2 \mathcal{I} \bigl[ \varphi \bigl(x^{-1}(1)_x\big(1+\tilde{\varphi}_+\big) +\big(\gamma_+^x / \gamma_+^0\big) {\rm e}^x x^{\sigma-1} (1)_x\big(1+\tilde{\psi}_+\big) \bigr) \bigr],\nonumber\\
 \tilde{\psi}_+ = 2 \mathcal{I} \bigl[ \varphi \bigl(x^{-1}(1)_x\big(1+\tilde{\psi}_+\big) +\big(\gamma_+^0 / \gamma_+^x\big) {\rm e}^{-x} x^{-\sigma-1} (1)_x\big(1+\tilde{\varphi}_+\big) \bigr) \bigr],\nonumber\\
 \tilde{\varphi}_- = 2 \mathcal{I} \bigl[ \varphi \bigl(x^{-1}(1)_x\big(1+\tilde{\varphi}_-\big) +\big(\gamma_-^x / \gamma_-^0\big) {\rm e}^{-x} x^{-\sigma-1} (1)_x\big(1+\tilde{\psi}_-\big) \bigr) \bigr],\nonumber\\
 \tilde{\psi}_- = - 2 \mathcal{I} \bigl[ \varphi \bigl(x^{-1}(1)_x\big(1+\tilde{\psi}_-\big) +\big(\gamma_-^0 / \gamma_-^x\big) {\rm e}^{x} x^{\sigma-1} (1)_x\big(1+\tilde{\varphi}_-\big) \bigr) \bigr].\label{6.2}
\end{gather}
This} implies that, for $\big(\varphi^{\infty}, \varphi^{\infty}_{\pm},\psi^{\infty}_{\pm}\big)$, the corresponding solution $\big(\tilde{\varphi}^{\infty},\tilde{\varphi}^{\infty}_{\pm},\tilde{\psi}^{\infty}_{\pm}\big)$ of~\eqref{6.2} depends on $\sigma$ and $c=c_x/c_0$ only, since $\gamma_-^0\gamma_+^x$, $\gamma_+^0\gamma_-^x$, $\gamma_+^x/\gamma_+^0$, $\gamma_-^x/\gamma_-^0$ (respectively, $\gamma_+^0\gamma_-^0$, $\gamma_+^x\gamma_-^x$) are written in terms of $\sigma$ and $c$ (respectively, $\sigma$) only. From~\eqref{4.6} it follows that
\begin{gather*}
 f_0= -g_0 -\theta_{\infty}/2 =(\sigma-\theta_{\infty})/4 + \varphi^{\infty},\\
 \frac{g_+}{f_+} = \frac{(1)_x c {\rm e}^x x^{\sigma} \big(1+\tilde{\psi}_+^{\infty} \big) - \big( (\sigma +\theta_{\infty}
)/2 +\big[x^{-1}\big] \big) x^{-1} \big(1+\tilde{\varphi}_+^{\infty} \big) }
{(1)_x \big(1+\tilde{\varphi}_+^{\infty} \big) - \big( (\sigma -\theta_{\infty}
)/2 +\big[x^{-1}\big] \big)c {\rm e}^x x^{\sigma-1} \big(1+\tilde{\psi}_+^{\infty} \big) },
\end{gather*}
and hence the proposition follows immediately.
\end{proof}

Suppose $\operatorname{dist}(\{-2\theta_0 +\theta_{\infty}, 2\theta_x -\theta_{\infty} \}, B_*)=d_0 >0$, that is, $\big|\gamma_+^0/c_0\big|, \big|\gamma_+^{x}/c_x\big|\ge d_0/4$ for every $\sigma \in B_*$. Then, by Theorem~\ref{thm2.1} combined with the expression of $\varphi^{\infty}$ in Section~\ref{ssc5.4},
\begin{gather*}
\frac{f_0+\theta_0/2}{g_0+\theta_x/2} = \frac{\gamma^0_+ /c_0 +\varphi^{\infty}} {\gamma^x_+/c_x -\varphi^{\infty}} = \left( 1+\frac{c_x \gamma^0_+}{c_0 \gamma^x_+} \right) \left( 1- \frac {c_x}{\gamma_+^x}\varphi^{\infty}
 \right)^{-1}-1,\\
 x^{(\sigma +\theta_{\infty})/2} g_+ = \gamma_+^x {\rm e}^x x^{\sigma} (1)_x\big(1+ \big(\gamma_+^x\big)^{-1} \hat{g}_+ \big),\\
 \big(x^{(\sigma +\theta_{\infty})/2} f_+ \big)^{-1} = \big(\gamma_+^0\big)^{-1} (1)_x \big(1- \big(\gamma_+^0\big)^{-1} \hat{f}_+ \big)^{-1},
\end{gather*}
provided that $\big|{\rm e}^x x^{\sigma-1}\big|$ and $\big|{\rm e}^{-x} x^{-\sigma-1}\big|$ are sufficiently small, where
\begin{gather*}
\hat{g}_+ = \sum_{n=1}^{\infty} \gamma_+^x \big(\gamma_-^0 \gamma_+^x\big)^n \big[x^{-n}\big]\big({\rm e}^x x^{\sigma-1}\big)^n -\gamma_+^0 \big((\sigma+\theta_{\infty})/2+\big[x^{-1}\big]\big) {\rm e}^{-x}x^{-\sigma-1}\\
\hphantom{\hat{g}_+ =}{} - \big(\gamma_+^0\big)^2 \gamma_-^x (1)_x \big({\rm e}^{-x} x^{-\sigma-1}\big)^2 + \sum_{n=3}^{\infty} \gamma_+^0 \big(\gamma_+^0 \gamma_-^x\big)^{n-1} \big[x^{-n+2}\big]\big({\rm e}^{-x} x^{-\sigma-1}\big)^n,\\
\hat{f}_+ = \gamma_+^x \big((\sigma-\theta_{\infty})/2+\big[x^{-1}\big]\big) {\rm e}^{x}x^{\sigma-1} + \gamma_-^0\big(\gamma_+^x\big)^2 (1)_x \big({\rm e}^{x} x^{\sigma-1}\big)^2\\
\hphantom{\hat{f}_+ =}{} + \sum_{n=3}^{\infty} \gamma_+^x \big(\gamma_-^0 \gamma_+^x\big)^{n-1} \big[x^{-n+2}\big]\big({\rm e}^x x^{\sigma-1}\big)^n+ \sum_{n=1}^{\infty} \gamma_+^0 \big(\gamma_+^0 \gamma_-^x\big)^{n} \big[x^{-n}\big] \big({\rm e}^{-x} x^{-\sigma-1}\big)^n.
\end{gather*}
Furthermore
\begin{gather*}
\big(1-c_x \big(\gamma^x_+\big)^{-1} \varphi^{\infty}\big)^{-1} =(1-\chi_0)^{-1}\big(1- x^{-1}\varphi^{\infty}_* (1-\chi_0)^{-1}\big)^{-1},\\
\big(1- \big(\gamma^0_+\big)^{-1} \hat{f}_+\big)^{-1} =(1-\chi_1)^{-1}\big(1- x^{-1}\hat{f}_{+*} (1-\chi_1)^{-1}\big)^{-1},
\end{gather*}
where
\begin{gather*}
\chi_0 = c_x \big(\gamma_+^x\big)^{-1} \big(\gamma_-^0 \gamma_+^x(1)_x {\rm e}^x x^{\sigma-1}+\gamma_+^0 \gamma_-^x(1)_x {\rm e}^{-x} x^{-\sigma-1}\big),\\
 x^{-1} \varphi^{\infty}_* = c_x \big(\gamma_+^x\big)^{-1} \big( \varphi^{\infty} -\gamma_-^0 \gamma_+^x(1)_x {\rm e}^x x^{\sigma-1}- \gamma_+^0 \gamma_-^x(1)_x {\rm e}^{-x} x^{-\sigma-1}\big),\\
\chi_1 = \big(\gamma_+^0\big)^{-1} \big(\gamma_+^x \big((\sigma-\theta_{\infty})/2+\big[x^{-1}\big]\big) {\rm e}^x x^{\sigma-1}+\gamma_-^0 \big(\gamma_+^x\big)^2(1)_x \big({\rm e}^{x} x^{\sigma-1}\big)^2\big),\\
x^{-1}\hat{f}_{+*} = \big(\gamma_+^0\big)^{-1} \big(\hat{f}_+ - \gamma_+^x\big((\sigma-\theta_{\infty})/2+\big[x^{-1}\big]\big) {\rm e}^x x^{\sigma-1}-\gamma_-^0 \big(\gamma_+^x\big)^2(1)_x \big({\rm e}^{x} x^{\sigma-1}\big)^2\big).
\end{gather*}
By Proposition \ref{prop3.2} with Examples \ref{exa3.2} and \ref{exa3.3}, substitution of these expressions into \eqref{6.1} yields $y(c,\sigma, x)$ as in Theorem~\ref{thm2.6}.

If we set $\sigma=\sigma_0= -2\theta_x -\theta_{\infty}$, that is, $\gamma^x_- =0$, then the coefficients of the solution in Theo\-rem~\ref{thm2.6} are such that $b_n=0$ for $n\ge 2$, and we obtain the solution $y_+(c, x)$. If $\sigma=\sigma'_0=2\theta_0+\theta_{\infty}$, that is, $\gamma_-^0=0$, then
\begin{gather*}
 x^{(\sigma'_0+\theta_{\infty})/2}f_+= \gamma_+^0 (1)_x \big(1- \big(\gamma_+^0\big)^{-1}\hat{f}_+\big),\\
\big(x^{(\sigma'_0+\theta_{\infty})/2}g_+\big)^{-1}=\big( \gamma_+^x\big)^{-1} {\rm e}^{-x}x^{-\sigma'_0} (1)_x \big(1+ \big(\gamma_+^x\big)^{-1} \hat{g}_+\big)^{-1},
\end{gather*}
from which the solution $y_-(c',x)$ follows. Thus Theorem~\ref{thm2.7} is obtained.

\subsection{Proofs of Theorems \ref{thm2.8} and \ref{thm2.9}} \label{ssc6.2}

To discuss the poles and zeros of $y(c,\sigma, x)$, under the condition $\gamma^0_{\pm}, \gamma^x_{\pm} \not=0$ we write
\begin{gather*}
{\rm e}^{-x}x^{-(\sigma-\theta_{\infty})/2} g_+= \gamma_+^x -\big(\gamma_+^0/2\big) (\sigma +\theta_{\infty}){\rm e}^{-x}x^{-\sigma-1} - \big(\gamma_+^0\big)^2 \gamma_-^x \big({\rm e}^{-x}x^{-\sigma-1}\big)^2
+O\big(x^{-1}\big)\\
\hphantom{{\rm e}^{-x}x^{-(\sigma-\theta_{\infty})/2} g_+}{}
= - \big(\gamma_+^0\big)^2 \gamma_-^x \big({\rm e}^{-x}x^{-\sigma-1} -\varrho_1 \big) \big({\rm e}^{-x}x^{-\sigma-1} -\varrho_2 \big) + O\big(x^{-1}\big),\\
x^{(\sigma+\theta_{\infty})/2} f_+= \gamma_+^0 -\big(\gamma_+^x/2\big) (\sigma -\theta_{\infty}){\rm e}^{x}x^{\sigma-1} - \gamma_-^0\big( \gamma_+^x\big)^2 \big({\rm e}^{x}x^{\sigma-1}\big)^2+O\big(x^{-1}\big)\\
\hphantom{x^{(\sigma+\theta_{\infty})/2} f_+}{} = - \gamma_-^0\big( \gamma_+^x\big)^2 \big({\rm e}^{x}x^{\sigma-1} -\tilde{\varrho}_1 \big) \big({\rm e}^{x}x^{\sigma-1} -\tilde{\varrho}_2 \big) + O\big(x^{-1}\big)
\end{gather*}
with
\begin{gather*}
\varrho_1= -\frac{c_x}{\gamma_+^0} = \frac{-4c}{\sigma+2\theta_0 -\theta _{\infty}}, \qquad \varrho_2= \frac{\gamma^x_+}{c_x \gamma_+^0\gamma_-^x} = \frac{-4c(\sigma-2\theta_x +\theta_{\infty})}
{(\sigma +2\theta_x +\theta_{\infty}) (\sigma+2\theta_0 -\theta_{\infty}) },\\
\tilde{\varrho}_1= \frac{c_0}{\gamma_+^x} = \frac{-4}{c(\sigma-2\theta_x +\theta_{\infty})}, \qquad\tilde{\varrho}_2= -\frac{\gamma^0_+}{c_0 \gamma_-^0\gamma_+^x}= \frac{-4(\sigma+2\theta_0 -\theta_{\infty})}
{c(\sigma -2\theta_0 -\theta_{\infty}) (\sigma-2\theta_x +\theta_{\infty}) },
\end{gather*}
provided that $\big|{\rm e}^x x^{\sigma-1}\big|$, $\big|{\rm e}^{-x} x^{-\sigma-1}\big|<\varepsilon$. Furthermore, for $\varepsilon^{-1}|x|^{-1}< \big|{\rm e}^{-x}x^{-\sigma-1}\big| <\varepsilon$, $\big|{\rm e}^xx^{\sigma-1}\big|$ $<\varepsilon |x|^{-1}$,
\begin{gather*}
f_0+ \theta_0/2 = \gamma_+^0\gamma_-^x \big({\rm e}^{-x} x^{-\sigma-1} -\varrho_3 \big) +O\big(x^{-1}\big),\\
g_0+ \theta_x/2 = - \gamma_+^0\gamma_-^x \big({\rm e}^{-x} x^{-\sigma-1} -\varrho_2 \big) +O\big(x^{-1}\big),
\end{gather*}
and for $\varepsilon^{-1}|x|^{-1}< |{\rm e}^{x}x^{\sigma-1}| <\varepsilon$, $\big|{\rm e}^{-x}x^{-\sigma-1}\big|<\varepsilon |x|^{-1}$,
\begin{gather*}
f_0+ \theta_0/2 = \gamma_-^0\gamma_+^x \big({\rm e}^{x} x^{\sigma-1} -\tilde{\varrho}_2 \big) +O\big(x^{-1}\big),\\
g_0+ \theta_x/2 = - \gamma_-^0\gamma_+^x \big({\rm e}^{x} x^{\sigma-1} -\tilde{\varrho}_3 \big) +O\big(x^{-1}\big),
\end{gather*}
where
\begin{gather*}
\varrho_3= -\frac 1{c_0\gamma_-^x} = \frac{-4c}{\sigma+2\theta_x +\theta_{\infty}}, \qquad \tilde{\varrho}_3= \frac 1{c_x\gamma_-^0} = \frac{-4}{c(\sigma-2\theta_0 -\theta_{\infty})}.
\end{gather*}
If $\theta_x(\theta_0 \pm \theta_x -\theta_{\infty}) \not=0$ (respectively, $\theta_0 (\pm\theta_0 - \theta_x +\theta_{\infty}) \not=0$), $\varrho_1$, $ \varrho_2$, $\varrho_3$ (respectively, $\tilde{\varrho}_1$, $ \tilde{\varrho}_2$, $\tilde{\varrho}_3$) are distinct. By the expressions above, in the domain $\varepsilon^{-1}|x|^{-1} < \big|{\rm e}^{-x} x^{-\sigma -1} \big| <\varepsilon$, where $\varepsilon$ is a positive number such that
\begin{gather}
\nonumber \bigl(\big|\gamma^0_- \gamma_+^x \big| + \big|\gamma^0_+ \gamma_-^x \big| + \big|\gamma^0_+\big| + \big|\gamma^0_-\big| + \big|\gamma^x_+\big| + \big|\gamma^x_-\big| +1\bigr)\\
\label{6.3}
\qquad {}\times \bigl( \big|\gamma^0_+\big| + \big|\gamma^0_-\big| + \big|\gamma^x_+\big| + \big|\gamma^x_-\big| +1 \bigr)\varepsilon
\le r_0(\delta)
\end{gather}
(cf.\ Remark~\ref{rem2.1}), $y(c,\sigma, x)$ admits a zero $x^{(0)}$ such that ${\rm e}^{-x^{(0)}} \big(x^{(0)}\big)^{-\sigma-1} \sim \varrho_1$ (respectively, $\sim \varrho_3$), if $|\varrho_1|<\varepsilon$ (respectively, $|\varrho_3| <\varepsilon$). Let $(c_0, c_x)=(1,c)$ with $0<|c|< R_0$. If $|\gamma_-^x| =|(\sigma +2\theta_x +\theta_{\infty})/(4c) |< R_0/4$, then $|\sigma + 2\theta_x +\theta_{\infty}|< R_0^2$ and $\big|\gamma_{\pm}^0\big|$, $\big|\gamma_+^x\big|< R_*$, $R_*=R_*(\theta_0,\theta_x,\theta_{\infty}, R_0)$ being some number depending on $(\theta_0,\theta_x,\theta_{\infty}, R_0)$ only. Choose $\varepsilon=\varepsilon_0$ in such a way that~\eqref{6.3} is valid uniformly in $\big(\gamma_{\pm}^0, \gamma_{\pm}^x\big)$ satisfying $\big|\gamma_{\pm}^0\big|$, $\big|\gamma_+^x\big|< R_*$, $\big|\gamma_-^x\big|< R_0/4$. Then there exists a~domain consisting of $(c,\sigma)$ such that $|(\sigma +2\theta_x +\theta_{\infty})/c|< R_0$, $\sigma \not= \pm 2\theta_0 +\theta_{\infty}$, $2\theta_x -\theta_{\infty}$ and $|\varrho_1| =|4c/(\sigma +2\theta_0 -\theta_{\infty})|<\varepsilon_0$, since $\theta_0 -\theta_x -\theta_{\infty} \not=0$. For such $(c,\sigma)$, $y(c,\sigma, x)$ has a~sequence of zeros $\big\{x_m^{(0)} \big\}$ with $\rho_0(\sigma)=\varrho_1/c = -4/(\sigma+2\theta_0 -\theta_{\infty})$. For $(c_0,c_x)=(1/c, 1)$ with $0<|c|<R_0$, we obtain another sequence of zeros with $\rho_0(\sigma)=\varrho_3/c = -4/(\sigma +2\theta_x+\theta_{\infty})$. Thus the assertion (1) of Theorem~\ref{thm2.8} has been verified. The second assertion is shown by an analogous argument about a pole $x^{(\infty)}$ such that ${\rm e}^{x^{(\infty)}}\big(x^{(\infty)}\big)^{\sigma -1} \sim \tilde{\varrho}_1$ or $\sim \tilde{\varrho}_3$ in the domain $\varepsilon^{-1}|x|^{-1} < \big|{\rm e}^x x^{\sigma-1}\big|<\varepsilon$. By putting $\sigma=\sigma_0=-2\theta_x -\theta_{\infty}$ or $\sigma'_0 =2\theta_0+\theta_{\infty}$, and observing ${\rm e}^{-x}x^{-(\sigma-\theta_{\infty})/2} g_+ = \gamma^x_+ - \big(\gamma_+^0/2\big) (\sigma_0+\theta_{\infty}){\rm e}^{-x}x^{-\sigma-1} +O\big(x^{-1}\big)$ and so on, we deduce Theo\-rem~\ref{thm2.9}.

\section{Proofs of the results on the monodromy data}\label{sc7}

To show Theorem~\ref{thm2.3} we compute the monodromy matrices~$M_0$,~$M_x$ with respect to solution~\eqref{2.1} of linear system~\eqref{1.1} by
matching perturbed solutions as $x\to \infty$. Note that, by Theo\-rem~\ref{thm2.1},
\begin{gather*}
A_0(\mathbf{c},\sigma, x) \sim \frac 14 (\sigma-\theta_{\infty})J +\gamma_+^0 x^{-(\sigma +\theta_{\infty}) /2}\Delta_+ +\gamma_-^0 x^{ (\sigma +\theta_{\infty}) /2}\Delta_-, \\
A_x(\mathbf{c},\sigma, x) \sim - \frac 14 (\sigma+\theta_{\infty})J +\gamma_+^{x} {\rm e}^{x} x^{(\sigma -\theta_{\infty}) /2}\Delta_+
+\gamma_-^{x} {\rm e}^{-x} x^{-(\sigma -\theta_{\infty}) /2}\Delta_-,
\end{gather*}
if ${\rm e}^x x^{\sigma-1}$, ${\rm e}^{-x} x^{-\sigma-1} =o(1)$. In what follows we suppose that $\arg x \sim \pi/2$ and that
\begin{gather}\label{7.1}
\big|{\rm e}^x x^{\sigma}\big|, \big|{\rm e}^{-x} x^{-\sigma}\big| \ll 1
\end{gather}
as $x\to \infty$. By $Y= {\rm e}^{(x/4)J} x^{-(\theta_{\infty}/4)J} \hat{Y}$, system \eqref{1.1} is changed into
\begin{gather}\label{7.2}
\frac{{\rm d}\hat{Y}}{{\rm d}\lambda} = \left( \frac{\hat{A}_0}{\lambda} + \frac{\hat{A}_x}{\lambda-x} + \frac J 2 \right) \hat{Y}
\end{gather}
with
\begin{gather*}
 \hat{A}_0 = \big( (\sigma-\theta_{\infty})/4+ O\big(x^{-1}\big) \big)J\\
\hphantom{\hat{A}_0 =}{} +\big(\gamma_+^0 +O\big(x^{-1}\big) \big) {\rm e}^{-x/2} x^{-\sigma /2}\Delta_+
+\big(\gamma_-^0 +O\big(x^{-1}\big) \big) {\rm e}^{ x/2} x^{ \sigma /2}\Delta_-,\\
 \hat{A}_x = \big( {-}(\sigma+\theta_{\infty})/4+ O\big(x^{-1}\big) \big)J\\
\hphantom{\hat{A}_x =}{} +\big(\gamma_+^{x} +O\big(x^{-1}\big) \big) {\rm e}^{x/2} x^{\sigma /2}\Delta_+
+\big(\gamma_-^{x} +O\big(x^{-1}\big) \big) {\rm e}^{-x/2} x^{-\sigma /2}\Delta_-.
\end{gather*}

\subsection{Approximate equation}\label{ssc7.1}
As long as $|\lambda|> |x|^{1/2}$, $|\lambda- x|> |x|^{1/2}$, the eigenvalues of $A_*= \hat{A}_0/\lambda + \hat{A}_x/(\lambda-x) +J/2$ are $\pm \mu(x, \lambda)$ with
\begin{gather*}
\mu(x, \lambda)= \frac 12 + \frac{(\sigma-\theta_{\infty})/4 + O\big(x^{-1}\big) } {\lambda} - \frac{ (\sigma+\theta_{\infty})/4 + O\big(x^{-1}\big) } {\lambda-x} + O\big(|\lambda|^{-2} +|\lambda-x|^{-2} \big),
\end{gather*}
and by $\hat{Y} =( \mu(x,\lambda)J +A_* )Z =\big(J+O\big(|\lambda|^{-1} +|\lambda-x|^{-1}\big) \big)Z$ system \eqref{7.2} is reduced to
\begin{gather}\label{7.3}
\frac{{\rm d}Z}{{\rm d}\lambda}= (\mu(x, \lambda) J +H(x, \lambda) )Z, \qquad H(x, \lambda)=(h_{ij}(x, \lambda) ) \ll |\lambda|^{-2}+|\lambda-x|^{-2}.
\end{gather}

\begin{lem}\label{lem7.1}Let $\Sigma_{\pi/2}(x)$ be a sector given by
\begin{gather*}
|\arg \lambda -\pi/2|<\pi, \qquad |\arg (\lambda-x) -\pi/2|<\pi/4, \qquad |\lambda- x|>|x|^{1/2},
\end{gather*}
and $\Sigma_{3\pi/2}(0)$ one given by
\begin{gather*}
|\arg \lambda - 3\pi/2|<\pi/4, \qquad |\lambda|>|x|^{1/2}.
\end{gather*}
Then \eqref{7.2} admits the matrix solution
\begin{gather*}
Z^x_{\mathrm{WKB}}(x, \lambda) = \big(J+ O\big(|\lambda|^{-1} +|\lambda-x|^{-1}\big)\big) {\rm e}^{(\lambda/2) J} \lambda^{\alpha(x)J} (\lambda- x)^{\beta(x)J}
\end{gather*}
with
\begin{gather*}
\alpha(x)=(\sigma-\theta_{\infty})/4 +O\big(x^{-1}\big), \qquad \beta(x)=-(\sigma+\theta_{\infty})/4 +O\big(x^{-1}\big)
\end{gather*}
uniformly in sufficiently large $x$ as $\lambda \to \infty$ through $\Sigma_{\pi/2}(x)$, and the solution $Z^0_{\mathrm{WKB}}(x,\lambda)$ having an asymptotic representation of the same form as $\lambda\to \infty$ through $\Sigma_{3\pi/2}(0)$.
\end{lem}

\begin{figure}[htb]\small
\begin{center}
\unitlength=1mm
\begin{picture}(60,34)(-30,-16)
\put(0,-15){\circle*{1}}
\put(0,0){\circle*{1}}
\put(2.5,-1.5){\makebox{$x$}}
\put(2.5,-16.5){\makebox{$0$}}
\qbezier(-4,4)(0,8)(4,4)
\put(-4,4){\line(-1,1){10}}
\put(4,4){\line(1,1){7}}
\put(7,14){\makebox{$\Sigma_{\pi/2}(x)$}}
\end{picture}
\qquad\quad
\begin{picture}(60,34)(-30,-16)
\put(0,15){\circle*{1}}
\put(0,0){\circle*{1}}
\put(2.5,14.5){\makebox{$x$}}
\put(2.5,0.5){\makebox{$0$}}
\qbezier(-4,-4)(0,-8)(4,-4)
\put(-4,-4){\line(-1,-1){10}}
\put(4,-4){\line(1,-1){7}}
\put(5,-15){\makebox{$\Sigma_{3\pi/2}(0)$}}
\end{picture}
\end{center}
\caption{Sectors $\Sigma_{\pi/2}(x)$ and $\Sigma_{3\pi/2}(0)$.}\label{sectors}
\end{figure}
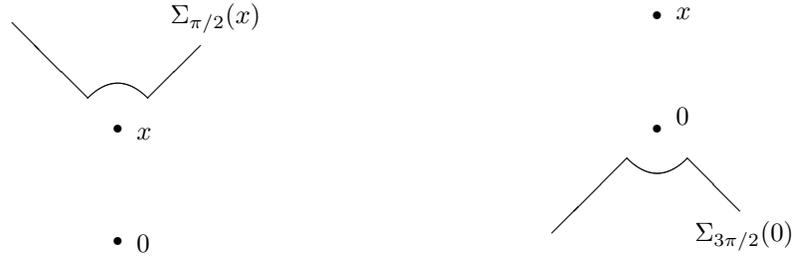

\begin{proof}By $Z=(I+ p\Delta_+)Z_*$ system \eqref{7.2} is taken to
\begin{gather*}
\frac{{\rm d}Z_*}{{\rm d}\lambda} = (\mu(x,\lambda) J + H(x,\lambda) -h_{21} pJ + h_* \Delta_+ ) Z_*
\end{gather*}
with
\begin{gather*}
h_*= 2\mu(x,\lambda) p -{\rm d}p/{\rm d}\lambda +(h_{11}-h_{22})p -h_{21}p^2.
\end{gather*}
From every point in $\Sigma_{\pi/2}(x)$ one may draw a line in $\Sigma_{\pi/2}(x)$ in such a way that $\re\lambda \to \infty $, and hence there exists $p=p(x,\lambda)$ such that $h_{12}+h_*=0$ and that $p(x,\lambda) \ll |\lambda|^{-2} +|\lambda-x|^{-2}$ (cf.\ Lemma~\ref{lem4.1} and Remark
\ref{rem4.2}). As a result the coefficient matrix becomes of lower-triangular form. We apply a suitable further transformation of the form $Z_*=(I + q\Delta_-)Z_{**}$ with $q=q(x,\lambda)\ll |\lambda|^{-2}+ |\lambda-x|^{-2}$ to get the diagonal system
\begin{gather*}
\frac{{\rm d}Z_{**}}{{\rm d}\lambda} = \big(\mu(x,\lambda)J +\operatorname{diag} \big[\tilde{h}_1 (x,\lambda) , \tilde{h}_2(x,\lambda) \big] \big)Z_{**}
\end{gather*}
with $\tilde{h}_1(x,\lambda), \tilde{h}_2(x,\lambda) \ll |\lambda|^{-2} +|\lambda-x|^{-2}$, from which the desired solution immediately follows.
\end{proof}

\begin{rem}\label{rem7.1} In the sector $\Sigma_{-\pi/2}(0)$: $|\arg \lambda +\pi/2 |<\pi/4$, $|\lambda| >|x|^{1/2}$ as well, \eqref{7.2} admits the solution $\hat{Z}^0_{\mathrm{WKB}} (x,\lambda)$ with an asymptotic representation of the same form.
\end{rem}

\begin{rem}\label{rem7.2} The matrix function $Z^x_{\mathrm{WKB}}(x,\lambda)$ or $W^0_{\mathrm{WKB}}(x,\lambda)$, which corresponds to $\Psi_q(\lambda)$ given by \cite[equation~(7.10)]{Andreev-Kitaev}, is essentially a WKB solution. The representation for it remains valid also in a suitably extended sector with opening angle $\pi-\delta$. Furthermore, e.g., the first column of $W^x_{\mathrm{WKB}}(x,\lambda)$ is a vector solution in a domain with the properties:
\begin{enumerate}\itemsep=0pt
\item[(i)] $|\arg(\lambda-x)|, |\arg \lambda|<3\pi/2 -\delta$, $|\lambda|, |\lambda-x| > |x|^{1/2}$;
\item[(ii)] from every point in the domain one may draw a line contained in it and satisfying \smash{$\re \lambda \to \infty $}.
\end{enumerate}
\end{rem}

\subsection{Local equation}\label{ssc7.2}
If $|\lambda|< 2|x|^{1/2}$, system \eqref{7.2} is written as
\begin{gather*}
\frac{{\rm d}\hat{Y}}{{\rm d}\lambda} = \left( \frac J2 + \frac{\hat{A}_0}{\lambda} +O\big(x^{-1}\big) \right) \hat{Y}
\end{gather*}
under \eqref{7.1}. This is equivalent to
\begin{gather}\label{7.4}
\frac{{\rm d}U}{{\rm d}\lambda} =\left( \frac J2 + \frac{\Lambda}{\lambda} +E(x,\lambda) \right) U, \qquad
\Lambda := \frac 14(\sigma-\theta_{\infty})J + \gamma^0_+ c_0^{-1} \Delta_+ + \gamma^0_- c_0 \Delta_-
\end{gather}
with $E(x,\lambda) \ll x^{-1}$ for $|\lambda|< 2|x|^{1/2}$, in which
\begin{gather*}
U= {\rm e}^{(x/4) J} x^{(\sigma/4)J} c_0^{-J/2} \hat{Y}.
\end{gather*}
System \eqref{7.4} is a perturbation of the Whittaker system
\begin{gather}\label{7.5}
\frac{{\rm d}W}{{\rm d}\lambda} =\left( \frac J2 + \frac{\Lambda}{\lambda} \right) W,
\end{gather}
which admits the matrix solution
\begin{gather}\label{7.6}
W_{\infty}(\lambda) =\big(I+O\big(\lambda^{-1}\big)\big) {\rm e}^{(\lambda/2)J} \lambda^{((\sigma- \theta_{\infty})/4)J}
\end{gather}
as $\lambda\to \infty$ through the sector $|\arg\lambda- \pi/2 | <\pi-\delta$. By $W_{\infty}^{*}(\lambda)$ and $W_{\infty}^{**}(\lambda)$ we denote the solutions of~\eqref{7.5} having an asymptotic representation of the same form in the neighbouring sectors $|\arg\lambda +\pi/2|<\pi-\delta$ and $|\arg\lambda -3\pi/2|<\pi-\delta$, respectively.

\begin{lem}\label{lem7.2} In the domain $|\arg \lambda-3\pi/2 |<\pi/4$, $\log |x|^{1/4} < |\lambda|< 2|x|^{1/2}$, system \eqref{7.4} admits the matrix solution
\begin{gather*}
U_{\mathrm{out}}(x,\lambda) =\big(I+U_{\mathrm{out}}^*(x,\lambda)\big) {\rm e}^{(\lambda/2)J} \lambda^{((\sigma-\theta_{\infty})/4)J}
\end{gather*}
with $U^*_{\mathrm{out}}(x,\lambda) \ll (\log |x|)^{-1}$.
\end{lem}

\begin{proof}Write $W^{**}_{\infty}(\lambda) = \bar{W}(\lambda) {\rm e}^{(\lambda/2)J} \lambda^{((\sigma-\theta_{\infty})/4)J}$ with $\bar{W}(\lambda) =I+O\big(\lambda^{-1}\big)$ as $\lambda \to\infty$. Since $W=\bar{W}(\lambda) \hat{W}$ reduces \eqref{7.5} to ${\rm d}\hat{W}/{\rm d}\lambda =(1/2 +(\sigma-\theta_{\infty}) /(4\lambda) )J\hat{W}$, it is easy to see that, by $U=\bar{W}(\lambda)\hat{U}$, \eqref{7.4} becomes
\begin{gather*}
\frac{{\rm d}\hat{U}}{{\rm d}\lambda} = \left( \left(\frac 12 + \frac{\sigma-\theta _{\infty} }{4\lambda} \right) J + O\big(x^{-1}\big)\right) \hat{U},
\end{gather*}
and $\bar{W}(\lambda)= I+ O\big((\log|x|)^{-1}\big)$, if $|\arg \lambda -3\pi/2|<\pi/4$, $\log|x|^{1/4}<|\lambda|<2|x|^{1/2}$. By the same argument as in the proof of Lemma~\ref{lem7.1}, we may find a transformation of the form $\hat{U}=\big(I+O\big(x^{-1}\big)\big)\tilde{U}$ such that this system is changed into
\begin{gather*}
\frac{{\rm d}\tilde{U}}{{\rm d}\lambda} = \left( \left(\frac 12 + \frac{\sigma-\theta _{\infty} }{4\lambda} \right) J + \operatorname{diag}[e_1(x,\lambda), e_2(x,\lambda)] \right) \tilde{U}
\end{gather*}
with $e_1(x,\lambda), e_2(x,\lambda) \ll x^{-1} $, which has a matrix solution given by
\begin{gather*}
\tilde{U}= \operatorname{diag} \big[\hat{e}_1(x,\lambda), \hat{e}_2(x,\lambda)\big] {\rm e}^{(\lambda/2)J} \lambda^{((\sigma-\theta_{\infty})/4)J}
\end{gather*}
with $\hat{e}_1(x,\lambda)-1, \hat{e}_2(x,\lambda)-1 \ll \lambda x^{-1} \ll x^{-1/2}$. Thus we obtain the lemma.
\end{proof}

\begin{lem} \label{lem7.2a} In the domain $|\arg \lambda-3\pi/2 |<\pi/4$, $1 < |\lambda| <\log |x|^{1/3}$, system \eqref{7.4} admits the matrix solution
\begin{gather*}
U_{\mathrm{in}}(x,\lambda) =\big(I+U_{\mathrm{in}}^*(x,\lambda)\big) W_{\infty}^{**}(\lambda)
\end{gather*}
with $U^*_{\mathrm{in}}(x,\lambda) \ll x^{-1/4}$.
\end{lem}

\begin{proof}By $U= W^{**}_{\infty}(\lambda)\bar{U}$ system \eqref{7.4} is reduced to
\begin{gather*}
\frac{{\rm d}\bar{U}}{{\rm d}\lambda} = W^{**}_{\infty}(\lambda)^{-1} E(x,\lambda) W^{**}_{\infty}(\lambda) \bar{U}.
\end{gather*}
Here $W^{**}_{\infty}(\lambda)^{-1} E(x,\lambda) W^{**}_{\infty}(\lambda) \ll x^{-2/3 +\epsilon}$ for $1<|\lambda|<\log |x|^{1/3}$, $\epsilon$ being any positive number, since $W^{**}_{\infty}(\lambda)^{\pm}\ll x^{1/6+\epsilon}$. Using Gronwall's inequality, we may show that there exists a matrix solution such that $\bar{U}=I+\bar{U}^*(x,\lambda)$ with $\bar{U}^*(x,\lambda) \ll x^{-2/3+\epsilon}$. Then
\begin{gather*}
U=W^{**}_{\infty}(\lambda) \big(I+\bar{U}^*(x,\lambda)\big) = \big(I+W^{**}_{\infty}(\lambda)\bar{U}^*(x,\lambda)W^{**}_{\infty}(\lambda)^{-1}\big) W^{**}_{\infty}(\lambda)
\end{gather*}
solves \eqref{7.4}. Since $W^{**}_{\infty}(\lambda)\bar{U}^*(x,\lambda)W^{**}_{\infty}(\lambda)^{-1} \ll x^{-1/3+\epsilon}$, this is a desired solution as in the lemma.
\end{proof}

\begin{rem}\label{rem7.3} In the domain $|\arg \lambda +\pi/2|<\pi/4$, $\log |x|^{1/4}<|\lambda| <2|x|^{1/2}$ (respectively, $1<|\lambda|<\log|x|^{1/3}$) as well, we have the solution $\hat{U}_{\mathrm{out}}(x,\lambda)$ (respectively, $\hat{U}_{\mathrm{in}}(x,\lambda)$) with an analogous property, which is obtained by using $W^*_{\infty}(\lambda)$.
\end{rem}

\subsection{Whittaker system}\label{ssc7.3}

The right-hand side of $W_{\infty}(\lambda)$ (cf.~\eqref{7.6}) is given by
\begin{gather*}
\begin{pmatrix}
{\rm e}^{\pi {\rm i}(\sigma-\theta_{\infty}+2)/4} W_{(\sigma-\theta_{\infty}+2)/4,
\theta_0/2}\big({\rm e}^{-\pi {\rm i}}\lambda\big) & -\vartheta_+ W_{-(\sigma-\theta_{\infty}+2)/4, \theta_0/2}(\lambda)
\vspace{1mm}\\
\vartheta_{-} {\rm e}^{\pi {\rm i}(\sigma-\theta_{\infty}+2)/4}
 W_{(\sigma-\theta_{\infty}-2)/4,\theta_0/2}\big({\rm e}^{-\pi {\rm i}}\lambda\big) &
 W_{-(\sigma-\theta_{\infty}-2)/4,\theta_0/2}(\lambda)
\end{pmatrix}
\lambda^{-1/2},
\end{gather*}
where $\vartheta_{+}= (\sigma-\theta_{\infty}+2\theta_0) /4$, $\vartheta_{-}= (\sigma-\theta_{\infty}-2\theta_0) /4$, and $W_{\kappa, \nu}(z)$ is the Whittaker function such that $W_{\kappa,\nu} (z)\sim {\rm e}^{-z/2}z^{\kappa}$ as $z\to\infty$ through the sector $|\arg z|<3\pi/2$ (cf.\ \cite[formula~(13.1.33)]{AS}, \cite[Section~6.9]{HTF}, \cite[equation~(3.10)]{Jimbo}).
Around $\lambda=0$, \eqref{7.5} admits the matrix solution
\begin{gather}\label{7.7}
W_0(\lambda) =G_0(I+O(\lambda)) \lambda^{(\theta_0/2)J} \lambda^{\Delta_*},
\end{gather}
where $G_0\in {\rm GL}_2(\C)$, and $\Delta_*$ denotes $0$ if $\theta_0\not\in \Z$, $\Delta_+$ if $\theta_0\in \N\cup\{0\}$, and $\Delta_-$ if $-\theta_0\in \N$.

Let us compute connection formulas and Stokes multipliers. Using the formula
\begin{gather*}
z^{-1/2} W_{\kappa, \nu}(z) = \frac{\Gamma(-2\nu) z^{\nu}} {\Gamma(1/2-\nu-\kappa)} (1+O(z)) + \frac{\Gamma(2\nu) z^{-\nu} } {\Gamma(1/2+\nu-\kappa)} (1+O(z))
\end{gather*}
near $z=0$ (cf.\ \cite[formulas (13.1.2), (13.1.32), (13.1.34)]{AS}), we have

\begin{prop}\label{prop7.3} If $\theta_0\not\in \Z$, then $W_{\infty}(\lambda)=W_0(\lambda) V_0$, where $V_0$ is the matrix as in Theo\-rem~{\rm \ref{thm2.3}}.
\end{prop}

Furthermore, if $2\nu \in \Z\setminus\{0\}$,
\begin{gather*}
W_{\kappa, \nu}(z) = \frac{(-1)^{1+|2\nu|}z^{1/2+|\nu|} } {|2\nu|! \Gamma(1/2-|\nu|-\kappa)} \bigl( (1+O(z))\log z+ \psi(1/2+ |\nu|-\kappa)\\
\hphantom{W_{\kappa, \nu}(z) =}{} -\psi(1) -\psi(1+|2\nu|) +O(z) \bigr) + \frac{(|2\nu|-1)! z^{1/2-|\nu|} (1+O(z)) }{ \Gamma(1/2+|\nu|-\kappa)},
\end{gather*}
and
\begin{gather*}
W_{\kappa, 0}(z) = -\frac{z^{1/2} }{\Gamma(1/2-\kappa)} \bigl( (1+O(z))\log z+ \psi(1/2-\kappa) -2\psi(1) +O(z) \bigr)
\end{gather*}
near $z=0$ (cf.\ \cite[formulas (13.1.6), (13.1.7), (13.1.33)]{AS}). From these formulas we have

\begin{prop}\label{prop7.4} If $\theta_0 \in \Z$, then $W_{\infty}(\lambda)=W_0(\lambda) \hat{V}_0$, where $\hat{V}_0$ is the matrix as in Theo\-rem~{\rm \ref{thm2.4}}.
\end{prop}

To calculate the relation between $W_{\infty}(\lambda)$ and $W_{\infty}^*(\lambda)$, we use
\begin{gather*}
W_{\kappa, \nu}\big({\rm e}^{-\pi {\rm i}}\lambda\big) = {\rm e}^{-\pi {\rm i}(\nu+1/2)} {\rm e}^{\lambda/2} \lambda^{\nu+1/2}
 \bigg( \frac{(1- {\rm e}^{4\pi {\rm i}\nu})\Gamma(-2\nu)}{\Gamma(1/2-\nu-\kappa)} M\big(\nu-\kappa+1/2, 2\nu+1, {\rm e}^{\pi {\rm i}}\lambda\big)\\
\hphantom{W_{\kappa, \nu}\big({\rm e}^{-\pi {\rm i}}\lambda\big) =}{} + {\rm e}^{4\pi {\rm i}\nu} U\big(\nu-\kappa+1/2, 2\nu+1, {\rm e}^{\pi {\rm i}}\lambda\big)\bigg),
\end{gather*}
which is obtained from \cite[formulas (13.1.10), (13.1.33)]{AS} with $n=-1$. Here, by \cite[formulas~(13.5.1), (13.5.2)]{AS}
\begin{gather*}
M\big(\nu-\kappa +1/2, 2\nu +1, {\rm e}^{\pi {\rm i}}\lambda\big) = \frac{\Gamma(2\nu+1) \lambda^{-(\nu-\kappa +1/2)} }{\Gamma(\nu +\kappa+1/2)}
\big(1+O\big(\lambda^{-1}\big)\big)\\
\hphantom{M\big(\nu-\kappa +1/2, 2\nu +1, {\rm e}^{\pi {\rm i}}\lambda\big) =}{} + \frac{\Gamma(2\nu+1) }{\Gamma(\nu-\kappa +1/2)} {\rm e}^{-\lambda} \big({\rm e}^{\pi {\rm i}} \lambda\big)^{-(\nu+\kappa +1/2)}\big(1+O\big(\lambda^{-1}\big)\big),\\
 U\big(\nu-\kappa +1/2, 2\nu +1, {\rm e}^{\pi {\rm i}}\lambda\big) =\big({\rm e}^{\pi {\rm i}}\lambda\big)^{-(\nu-\kappa +1/2)} \big(1+O\big(\lambda^{-1}\big)\big),
\end{gather*}
in the sector $-3\pi/2 <\arg \lambda <\pi/2$. From \cite[formulas~(13.1.10), (13.1.33)]{AS} with $n=1$, it follows that
\begin{gather*}
W_{\kappa, \nu}(\lambda) = {\rm e}^{-\lambda/2} \lambda^{\nu+1/2} \bigg( \frac{(1- {\rm e}^{-4\pi {\rm i}\nu})\Gamma(-2\nu)}{\Gamma(1/2-\nu-\kappa)} M(\nu-\kappa+1/2, 2\nu+1, {\rm e}^{-2\pi {\rm i}}\lambda)\\
\hphantom{W_{\kappa, \nu}(\lambda) =}{} + {\rm e}^{-4\pi {\rm i}\nu} U(\nu-\kappa+1/2, 2\nu+1,{\rm e}^{-2\pi {\rm i}}\lambda) \bigg),
\end{gather*}
which yields the relation between $W_{\infty}(\lambda)$ and $W_{\infty}^{**}(\lambda)$.

\begin{prop}\label{prop7.5} We have $W_{\infty}(\lambda)=W_{\infty}^*(\lambda) S_*$ and $W_{\infty}(\lambda)=W_{\infty}^{**}(\lambda) S_{**}$, where~$S_*$ and~$S_{**}$ are the matrices as in Theo\-rem~{\rm \ref{thm2.3}}.
\end{prop}

\subsection{Completion of the proofs of Theorems \ref{thm2.3} and \ref{thm2.4}}\label{ssc7.4}

Recall the solution $Y(x,\lambda)=\big(I+O\big(\lambda^{-1}\big)\big) {\rm e}^{(\lambda/2)J} \lambda^{-(\theta_{\infty}/2)J}$ of \eqref{1.1} as $\lambda\to\infty$ through the sector $|\arg\lambda-\pi/2|<\pi$ and the monodromy matrices $M_0$, $M_x$ defined by the analytic continuation of~$Y(x,\lambda)$ along the loops $l_0$, $l_x$ as described in Section~\ref{ssc2.2}. Furthermore, for the solutions $Y_1(x,\lambda)$ and $Y_2(x,\lambda)$, respectively, in $|\arg\lambda +\pi/2|<\pi$ and $|\arg \lambda -3\pi/2|<\pi$, the Stokes multipliers~$S_1$ and~$S_2$ are given by $Y(x,\lambda)=Y_1(x,\lambda)S_1$, $Y_2(x,\lambda)=Y(x,\lambda)S_2$.

\subsubsection[Derivation of $M_0$]{Derivation of $\boldsymbol{M_0}$}\label{sssc7.4.1}

To compute $M_0$ let us examine the analytic continuation for $Y(x,\lambda)$ along $l_0$ by the matching procedure carried out according to the following scheme:
\begin{gather*}
Y(x,\lambda)=Y_2(x,\lambda)S^{-1}_2 \longleftrightarrow Z^0_{\mathrm{WKB}} (x,\lambda) \longleftrightarrow U_{\mathrm{out}}(x,\lambda) \longleftrightarrow U_{\mathrm{in}}(x,\lambda)
\end{gather*}
(cf.\ Lemmas \ref{lem7.1}, \ref{lem7.2} and \ref{lem7.2a}).

Suppose that $x$ satisfies $\arg x \sim \pi/2$ and \eqref{7.1}, and that the starting point $\lambda_{\mathrm{st}}$ of $l_0$ has the properties $\arg\lambda_{\mathrm{st}} \sim \pi/2$, $\arg(\lambda_{\mathrm{st}}-x)\sim \pi/2$ and $|\lambda_{\mathrm{st}}|>2 |x|$. Then the part of $l_0$ from $\lambda_{\mathrm{st}}$ up to a point near $\lambda=0$ may be regarded as $\Gamma_{\mathrm{left}} \cup L_-$ with $\Gamma_{\mathrm{left}}$: $\lambda=|\lambda_{\mathrm{st}}|{\rm e}^{{\rm i}t}$ $(\pi/2 \le t \le 3\pi/2)$ and~$L_-$: $\lambda ={\rm i}t$ ($-|\lambda_{\mathrm{st}}| \le t \le -1$).

{\it $1$-st step}: Continue $Y(x,\lambda)$ along the arc $\Gamma_{\mathrm{left}}$ entering into $\Sigma_{3\pi/2}(0) \cap \{ |\lambda|> 2 |x| \}$, in which $|\arg (\lambda - x) - 3\pi/ 2|<\pi/4$ (cf.\ Lemma~\ref{lem7.1} and Fig.~\ref{loops1}(a)).

\begin{figure}[htb]\small
\begin{center}
\unitlength=0.75mm
\begin{picture}(60,55)(-30,-30)
\put(0,13){\circle*{1}}
\put(0,0){\circle*{1}}
\put(0,0){\circle{4}}
\put(2.5,12.5){\makebox{$x$}}
\put(4,0.5){\makebox{$0$}}
\qbezier(-5,-5)(0,-10)(5,-5)
\put(-5,-5){\line(-1,-1){16}}
\put(5,-5){\line(1,-1){14}}
\put(0,-23){\line(0,1){21}}
\put(10,-23){\makebox{$\Sigma_{3\pi/2}(0)$}}
\put(0,23){\circle*{1.5}}
\put(3,22){\makebox{$\lambda_{\mathrm{st}}$}}
\put(-19,5){\makebox{$\Gamma_{\mathrm{left}}$}}
\put(1,-15){\makebox{$L_-$}}
\thicklines
\qbezier(0,23)(-9.52,23)(-16.26,16.26)
\qbezier(-23,0)(-23,9.52)(-16.26,16.26)
\qbezier(0,-23)(-9.52,-23)(-16.26,-16.26)
\qbezier(-23,0)(-23,-9.52)(-16.26,-16.26)
\put(-6,-34){\makebox{(a) $l_0$}}
\end{picture}
\quad
\begin{picture}(60,55)(-30,-30)
\put(0,-18){\circle*{1}}
\put(0,-5){\circle*{1}}
\put(0,-5){\circle{4}}
\put(4,-7){\makebox{$x$}}
\put(2.5,-21){\makebox{$0$}}
\qbezier(-5,0)(0,5)(5,0)
\put(-5,0){\line(-1,1){16}}
\put(5,0){\line(1,1){14}}
\put(0,12){\line(0,-1){15}}
\put(10,16){\makebox{$\Sigma_{\pi/2}(x)$}}
\put(0,23){\circle*{1.5}}
\put(-8,22){\makebox{$\lambda_{\mathrm{st}}$}}
\thicklines
\put(-0.2,23){\line(0,-1){11}}
\put(0,23){\line(0,-1){11}}
\put(0.2,23){\line(0,-1){11}}
\put(-6,-34){\makebox{(b) $l_x$}}
\end{picture}
\quad
\begin{picture}(60,55)(-30,-30)
\put(0,13){\circle*{1}}
\put(0,0){\circle*{1}}
\put(0,0){\circle{4}}
\put(-5,12.5){\makebox{$x$}}
\put(-6,0.5){\makebox{$0$}}
\qbezier(-5,-5)(0,-10)(5,-5)
\put(-5,-5){\line(-1,-1){14}}
\put(5,-5){\line(1,-1){16}}
\put(0,-23){\line(0,1){21}}
\put(-27,-23){\makebox{$\Sigma_{-\pi/2}(0)$}}
\put(0,23){\circle*{1.5}}
\put(-8,22){\makebox{$\lambda_{\mathrm{st}}$}}
\put(9,5){\makebox{$\Gamma_{\mathrm{right}}$}}
\put(-7,-15){\makebox{$L'_-$}}
\thicklines
\qbezier(0,23)(9.52,23)(16.26,16.26)
\qbezier(23,0)(23,9.52)(16.26,16.26)
\qbezier(0,-23)(9.52,-23)(16.26,-16.26)
\qbezier(23,0)(23,-9.52)(16.26,-16.26)
\put(-6,-34){\makebox{(c) $l_x^{-1}l_0l_x$}}
\end{picture}
\end{center}
\caption{$l_0$, $l_x$ and $l_x^{-1}l_0l_x$.}\label{loops1}
\end{figure}
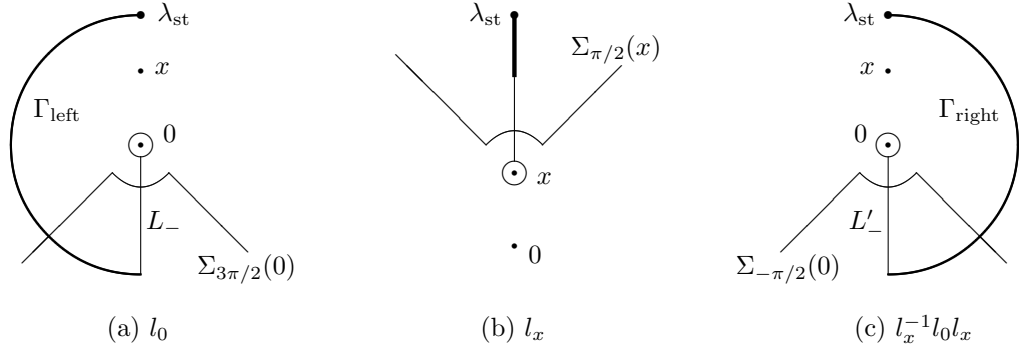

Let us match $Y_2(x,\lambda)$ with $Z^0_{\mathrm{WKB}}(x,\lambda)$ in this domain. Since $Z^0_{\mathrm{WKB}}(x,\lambda)$ solves \eqref{7.2} that follows from \eqref{1.1} by $Y={\rm e}^{(x/4)J} x^{-(\theta_{\infty}/4)J} \hat{Y}$, we have
\begin{gather}\label{7.8}
Y_2(x,\lambda) = {\rm e}^{(x/4)J} x^{-(\theta_{\infty}/4)J} Z^0_{\mathrm{WKB}}(x,\lambda) \Upsilon_1(x)
\end{gather}
with $\Upsilon_1(x) \in {\rm SL}_2(\C)$. By Lemma~\ref{lem7.1}, the right-hand side is
\begin{gather*}
{\rm e}^{(x/4)J} x^{-(\theta_{\infty}/4)J} \big(J+O\big(\lambda^{-1}\big) \big) {\rm e}^{(\lambda/2)J} \lambda^{\alpha(x)J} (\lambda-x)^{\beta(x) J} \Upsilon_1(x)\\
\qquad{} =\big(I+O\big(\lambda^{-1/2}\big) \big) {\rm e}^{(\lambda/2)J} \lambda^{-(\theta_{\infty}/2 +O\big(x^{-1}\big) )J} {\rm e}^{(x/4)J} x^{-(\theta_{\infty}/4)J} J \Upsilon_1(x),
\end{gather*}
provided that, e.g., $|\lambda|^{1/2} \gg |x|^{(|\sigma| +|\theta_{\infty}|)/2} \gg \big|{\rm e}^{x/2} x^{-\theta_{\infty}/2}\big|^{\pm 1}$ (cf.~\eqref{7.1}). Under $|\lambda|\!\ll\! \exp\big(|x|^{1/2}\big)$, which implies $\lambda^{O\big(x^{-1}\big)}=1+o(1)$, from~\eqref{7.8} we conclude
\begin{gather*}
\Upsilon_1(x) ={\rm e}^{-(x/4)J} x^{(\theta_{\infty}/4)J} (J+o(1)).
\end{gather*}

{\it $2$-nd step}: The line $L_-$ is contained in the sector $|\lambda-3\pi/2|<\pi/4$. Recall that $U= {\rm e}^{(x/4)J} x^{(\sigma/4)J} c_0^{-J/2} \hat{Y}$ takes \eqref{7.2} to \eqref{7.4}. Suppose that
\begin{gather}\label{7.9}
 {\rm e}^{(x/4)J} x^{(\sigma/4)J} c_0^{-J/2} Z^0_{\mathrm{WKB}}(x,\lambda)= U_{\mathrm{out}}(x,\lambda) \Upsilon_2(x)
\end{gather}
in the domain $|\arg\lambda -3\pi/2|<\pi/4$, $|x|^{1/2}<|\lambda|<2|x|^{1/2}$. By \eqref{7.1} the left-hand side is
\begin{gather*}
 {\rm e}^{(x/4)J} x^{(\sigma/4)J} c_0^{-J/2} \big(J+O\big(\lambda^{-1}\big)\big) {\rm e}^{(\lambda/2)J}
\lambda^{\alpha(x)J} \big({\rm e}^{\pi {\rm i}} x(1-\lambda/x)\big)^{\beta(x) J}\\
\qquad{} =(I+O\big(\lambda^{-1}\big) ) {\rm e}^{(\lambda/2)J} \lambda^{((\sigma-\theta_{\infty})/4 +
O\big(x^{-1}\big) )J} (1-\lambda/x)^{\beta(x) J} {\rm e}^{\beta(x)\pi {\rm i} J}\\
\qquad\quad{} \times {\rm e}^{(x/4)J} x^{(-\theta_{\infty}/4 +O\big(x^{-1}\big) )J} c_0^{-J/2} J,
\end{gather*}
since $|\arg(\lambda-x) -3\pi/2|<\pi/4$, $\arg x \sim \pi/2$. Using Lemma
\ref{lem7.2}, we derive
\begin{gather*}
\Upsilon_2(x)= {\rm e}^{-((\sigma+\theta_{\infty})\pi {\rm i}/4)J} {\rm e}^{(x/4)J} x^{-(\theta_{\infty}/4)J} c_0^{-J/2} (J+o(1)).
\end{gather*}

{\it $3$-rd step}: In the domain $|\arg\lambda-3\pi/2|<\pi/4$, $\log |x|^{1/4} <|\lambda|< \log |x|^{1/3}$, by Lemmas~\ref{lem7.2} and~\ref{lem7.2a}, we have
\begin{gather}\label{7.9a}
U_{\mathrm{out}}(x,\lambda)=U_{\mathrm{in}}(x,\lambda) \Upsilon_3(x)
\end{gather}
with $\Upsilon_3(x) =I+ o(1)$. By Lemma \ref{lem7.2} and Propositions \ref{prop7.3} through \ref{prop7.5},
\begin{gather}\label{7.10}
U_{\mathrm{in}}(x,\lambda) = \big(I+ O\big(x^{-1}\big)\big) W_{\infty}(\lambda) S_{**}^{-1} = \big(I+O\big(x^{-1}\big)\big) W_0(\lambda) V_*S_{**}^{-1}
\end{gather}
in the domain $|\arg\lambda -3\pi/2|<\pi/4$, $1<|\lambda|< 2$, where $V_*= V_0$ if $\theta_0\not\in \Z$ (respectively, $V_*=\hat{V}_0$ if $\theta_0\in \Z$). From \eqref{7.8}, \eqref{7.9}, \eqref{7.9a} and \eqref{7.10}, as a result of the matching procedure we obtain the following connection formula:
\begin{gather*}
Y(x,\lambda)= x^{-((\sigma+\theta_{\infty})/4)J} c_0^{J/2} \big(I+O\big(x^{-1}\big)\big) G_0 (I+O(\lambda)) \lambda ^{(\theta_{\infty}/2)J}\lambda^{\Delta_*} \Upsilon_0(x)
\end{gather*}
with
\begin{gather*}
\Upsilon_0(x)= V_* S_{**}^{-1} \Upsilon_3(x) \Upsilon_2(x) \Upsilon_1(x) S_2^{-1} = V_* S_{**}^{-1} {\rm e}^{-\pi {\rm i}((\sigma+\theta_{\infty})/4)J} c_0^{-J/2} (I+o(1)) S_2^{-1}
\end{gather*}
around $\lambda=0$ as $|x|\to \infty$, $\arg x\sim \pi/2$. Since $M_0$ does not depend on $x$, we derive
\begin{gather*}
M_0 =S_2(C^2_0)^{-1} {\rm e}^{\pi {\rm i} \theta_0 J} C_0^2 S_2^{-1}, \qquad
C_0^2 =V_0 S_{**}^{-1} {\rm e}^{-\pi {\rm i}((\sigma+\theta_{\infty})/4)J} c_0^{-J/2}
\end{gather*}
if $\theta_0\not\in \Z$, which is the second relation in~\eqref{2.3}. The case $\theta_0\in \Z$ is treated similarly, and that of~\eqref{2.5} follows.

\subsubsection[Derivation of $M_xM_0M_x^{-1}$]{Derivation of $\boldsymbol{M_xM_0M_x^{-1}}$}\label{sssc7.4.2}

The curve $\Gamma_{\mathrm{right}} \cup L_{-}'$ issuing from $\lambda_{\mathrm {st}}$, where $\Gamma_{\mathrm{right}}$: $\lambda=|\lambda_{\mathrm{st}}| {\rm e}^{{\rm i}(\pi/2-t)}$ $(0\le t\le \pi)$ and $L'_{-}$: $\lambda={\rm i}t$ $(-|\lambda_{\mathrm{st}}| \le t \le -1)$, corresponds to the part of $l_x^{-1} l_0l_x$ from $\lambda_{\mathrm{st}}$ up to $\lambda=-1$ (cf.\ Fig.~\ref{loops1}(c)). In this case the matching scheme
\begin{gather*}
Y(x,\lambda)=Y_1(x,\lambda)S_1 \longleftrightarrow \hat{Z}^0_{\mathrm{WKB}}
(x,\lambda) \longleftrightarrow \hat{U}_{\mathrm{out}}(x,\lambda) \longleftrightarrow \hat{U}_{\mathrm{in}}(x,\lambda)
\end{gather*}
(cf. Remarks \ref{rem7.1} and \ref{rem7.3}) yields the monodromy matrix $M_xM_0 M_x^{-1}$. Note that $\hat{U}_{\mathrm{out}}(x,\lambda)=(I+o(1)) W^*_{\infty}(\lambda) (I+o(1))$, and in matching $\hat{Z}^0_{\mathrm {WKB}}(x,\lambda)$ with $\hat{U}_{\mathrm{out}}(x,\lambda)$, that $ \lambda-x = {\rm e}^{-\pi {\rm i}} x(1-\lambda/x)$, since $|\arg(\lambda-x) +\pi/2|< \pi/4$, $\arg x \sim \pi/2$. Then we obtain
\begin{gather*}
M_xM_0 M_x^{-1} =S_1^{-1} \big(C_0^1\big)^{-1} {\rm e}^{\pi {\rm i}\theta_0 J}C_0^1 S_1, \qquad
C_0^1 =V_0 S_*^{-1} {\rm e}^{\pi {\rm i}((\sigma+\theta_{\infty})/4)J } c_0^{-J/2}
\end{gather*}
if $\theta_0\not\in \Z$. In this way the first relations in~\eqref{2.3} and~\eqref{2.5} are verified.

\subsubsection[Derivation of $M_x$]{Derivation of $\boldsymbol{M_x}$}\label{sssc7.4.3}
In the domain $|\lambda-x|<2|x|^{1/2}$, we write \eqref{7.2} in the form
\begin{gather*}
\frac{{\rm d}\hat{Y}}{{\rm d}\lambda} = \left( \frac J2 + \frac{\hat{A}_x}{\lambda-x}+O\big(x^{-1}\big) \right) \hat{Y},
\end{gather*}
which is changed into
\begin{gather}\label{7.11}
\frac{{\rm d}U}{{\rm d}\lambda} =\left( \frac J2 +\frac{ \tilde{\Lambda}}{\lambda-x} +O\big(x^{-1}\big) \right)U,
\qquad \tilde{\Lambda}= -\frac 14(\sigma+\theta_{\infty})J +\gamma^{x}_+c_x^{-1} \Delta_+ + \gamma^{x}_- c_x \Delta_-
\end{gather}
by $U= {\rm e}^{-(x/4)J} x^{-(\sigma/4)J} c_x^{-J/2} \hat{Y}$. Then instead of~\eqref{7.5} we treat
\begin{gather*}
\frac{{\rm d}W}{{\rm d}\lambda} = \left(\frac J2 +\frac{\tilde{\Lambda}}{\lambda-x} \right)W,
\end{gather*}
which has the matrix solution
\begin{gather}\label{7.12}
W^x_{\infty}(\lambda) =\big(I+O\big((\lambda-x)^{-1}\big) \big) {\rm e}^{(\lambda/2)J} (\lambda-x)^{-((\sigma+\theta_{\infty})/4)J}
\end{gather}
as $\lambda\to\infty$ through the sector $|\arg(\lambda-x)-\pi/2|<\pi-\delta$. Around $\lambda=x $ there exists the matrix solution
\begin{gather*}
W^x_0(\lambda) = G_x (I+O(\lambda-x)) (\lambda-x)^{(\theta_x/2)J} (\lambda-x) ^{\Delta_*}
\end{gather*}
with $G_x\in {\rm GL}_2(\C)$ and $\Delta_*$ as of~\eqref{7.7}. Then the connection formula is given by $W^x_{\infty}(\lambda) =W_0^x(\lambda)V_x$ (respectively, $ =W_0^x(\lambda)\hat{V}_x$) if $\theta_x \not\in \Z$ (respectively, $\theta_x \in \Z$). In the sector $|\arg(\lambda-x) -\pi/2 |<\pi/4$ equation \eqref{7.11} has the solution $U^x_{\mathrm{out}}(x,\lambda) =(I+U^{x*}_{\mathrm{out}}(x,\lambda)) {\rm e}^{(\lambda/2)J} \lambda^{((\sigma-\theta_{\infty})/4)J}$ with $U^{x*} _{\mathrm{out}}(x,\lambda) \ll (\log |x|)^{-1}$ for $\log|x|^{1/4}<|\lambda-x|<2|x|^{1/2}$, and $U^x_{\mathrm{in}}(x,\lambda) =(I+U^{x*}_{\mathrm{in}}(x,\lambda)) W_{\infty}^x(\lambda)$ with $U^{x*}_{\mathrm{in}}(x,\lambda) \ll x^{-1/4}$
for $1<|\lambda -x|< \log|x|^{1/3}$.

Consider the line joining $\lambda_{\mathrm{st}}$ with a point near $x$ contained in this sector (cf.\ Fig.~\ref{loops1}(b)). Then $M_x$ is obtained by the matching scheme
\begin{gather*}
Y(x,\lambda) \longleftrightarrow Z^x_{\mathrm{WKB}}(x,\lambda)
\longleftrightarrow U^x_{\mathrm{out}}(x,\lambda)
\longleftrightarrow U^x_{\mathrm{in}}(x,\lambda)
\end{gather*}
(cf. Lemma \ref{lem7.1}). Since $\arg x$, $\arg\lambda \sim \pi/2$, we may
write $\lambda=x (1+(\lambda-x)/x )$ in the domain $|x|^{1/2} <|\lambda-x|<
2|x|^{1/2}$. Using this fact we derive $M_x$ as in Theorem \ref{thm2.3}
or \ref{thm2.4}.

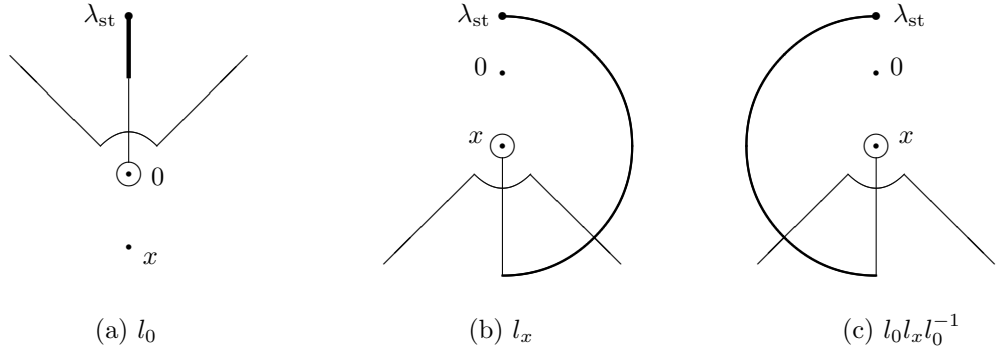
\begin{figure}[htb]\small
\begin{center}
\unitlength=0.75mm
\begin{picture}(60,55)(-30,-30)
\put(0,-18){\circle*{1}}
\put(0,-5){\circle*{1}}
\put(0,-5){\circle{4}}
\put(4,-7){\makebox{$0$}}
\put(2.5,-21){\makebox{$x$}}
\qbezier(-5,0)(0,5)(5,0)
\put(-5,0){\line(-1,1){16}}
\put(5,0){\line(1,1){16}}
\put(0,12){\line(0,-1){15}}
\put(0,23){\circle*{1.5}}
\put(-8,22){\makebox{$\lambda_{\mathrm{st}}$}}
\thicklines
\put(-0.2,23){\line(0,-1){11}}
\put(0,23){\line(0,-1){11}}
\put(0.2,23){\line(0,-1){11}}
\put(-6,-34){\makebox{(a) $l_0$}}
\end{picture}
\quad
\begin{picture}(60,55)(-30,-30)
\put(0,13){\circle*{1}}
\put(0,0){\circle*{1}}
\put(0,0){\circle{4}}
\put(-5,12.5){\makebox{$0$}}
\put(-6,0.5){\makebox{$x$}}
\qbezier(-5,-5)(0,-10)(5,-5)
\put(-5,-5){\line(-1,-1){16}}
\put(5,-5){\line(1,-1){16}}
\put(0,-23){\line(0,1){21}}
\put(0,23){\circle*{1.5}}
\put(-8,22){\makebox{$\lambda_{\mathrm{st}}$}}
\thicklines
\qbezier(0,23)(9.52,23)(16.26,16.26)
\qbezier(23,0)(23,9.52)(16.26,16.26)
\qbezier(0,-23)(9.52,-23)(16.26,-16.26)
\qbezier(23,0)(23,-9.52)(16.26,-16.26)
\put(-6,-34){\makebox{(b) $l_x$}}
\end{picture}
\quad
\begin{picture}(60,55)(-30,-30)
\put(0,13){\circle*{1}}
\put(0,0){\circle*{1}}
\put(0,0){\circle{4}}
\put(2.5,12.5){\makebox{$0$}}
\put(4,0.5){\makebox{$x$}}
\qbezier(-5,-5)(0,-10)(5,-5)
\put(-5,-5){\line(-1,-1){16}}
\put(5,-5){\line(1,-1){16}}
\put(0,-23){\line(0,1){21}}
\put(0,23){\circle*{1.5}}
\put(3,22){\makebox{$\lambda_{\mathrm{st}}$}}
\thicklines
\qbezier(0,23)(-9.52,23)(-16.26,16.26)
\qbezier(-23,0)(-23,9.52)(-16.26,16.26)
\qbezier(0,-23)(-9.52,-23)(-16.26,-16.26)
\qbezier(-23,0)(-23,-9.52)(-16.26,-16.26)
\put(-6,-34){\makebox{(c)} $l_0l_xl_0^{-1}$}
\end{picture}
\end{center}
\caption{$l_0$, $l_x$ and $l_0l_xl_0^{-1}$.}\label{loops2}
\end{figure}

\subsection{On Remark \ref{rem2.7}}\label{ssc7.5}

In the case $\arg x \sim -\pi/2$, the monodromy matrices are obtained in the same way as above. In the matching procedure to compute $M_0^{(-1)}$, we note the fact that $\lambda-x ={\rm e}^{\pi {\rm i}} x(1-\lambda/x)$ in the domain $|\arg\lambda -\pi/2|< \pi/4$, $|x|^{1/2} <|\lambda| < 2|x|^{1/2}$, since $\arg(\lambda-x) \sim \pi/2$ (cf.\ Fig.~\ref{loops2}(a)). The matrix $M_x^{(-1)}$ is obtained by using a~curve on the right-hand side of $\lambda=0$ entering into the domain $|\arg(\lambda-x)+\pi/2|<\pi/4$, $|x|^{1/2}<|\lambda-x|<2|x|^{1/2}$, in which $\lambda=x(1+(\lambda-x)/x)$, since $|\arg\lambda+\pi/2|<\pi/4$ (cf.\ Fig.~\ref{loops2}(b)). A curve on the left-hand side of $\lambda=0$ entering into the domain $|\arg(\lambda-x) -3\pi/2|<\pi/4$, $|x|^{1/2} <|\lambda-x|<2|x|^{1/2}$ (cf.\ Fig.~\ref{loops2}(c)) corresponds to the expression of $(M_0^{(-1)})^{-1} M_x^{(-1)}M_0^{(-1)}$, which is derived by using $\lambda={\rm e}^{2\pi {\rm i}}x(1+(\lambda-x)/x)$
in this domain, since $|\arg\lambda -3\pi/2 |<\pi/4$.

\subsection*{Acknowledgements} The author is grateful to the referees for valuable comments and for bringing the literature~\cite{L} to attention; and also appreciation goes to the editor for works of Andrei Kapaev.

\pdfbookmark[1]{References}{ref}
\LastPageEnding

\end{document}